\documentclass[11pt]{article}
\usepackage{mathrsfs}
\usepackage{amsfonts}
\usepackage{xcolor}
\usepackage{bm}
\usepackage{appendix}
\usepackage{extarrows}
\usepackage{indentfirst,amssymb,amsmath,graphicx,amsthm,amsfonts,mathrsfs,multicol,amscd}
\usepackage{pgfplots}
\usepackage{verbatim}
\usetikzlibrary{patterns,arrows,positioning,calc,fadings,shapes,decorations.markings}
\usepackage{tikz}
\usepackage{amsmath}
\allowdisplaybreaks[4]
\usepackage{geometry, color}
\numberwithin{equation}{section}

\newcommand{\R}{\mathbb{R}}

\newtheorem{theorem}{Theorem}[section]

\newtheorem{lemma}[theorem]{Lemma}
\newtheorem{proposition}[theorem]{Proposition}

\theoremstyle{definition}
\newtheorem{definition}[theorem]{Definition}
\newtheorem{remark}{Remark}

\title{Existence, non-degeneracy and local uniqueness of multi-peak solutions to the fractional Schr\"{o}dinger equation
 with nearly critical  exponent in $\mathbb{R}^N$}

\author{Yanyan Guo
\thanks{School of Mathematics and Statistics, Central China
Normal University, Wuhan, 430079, P. R. China. Email: \texttt{yanyangcx@126.com}.}
\ \ 
Ying Li
\thanks{School of Mathematics and Statistics, Central China
Normal University, Wuhan, 430079, P. R. China. Email: \texttt{yingli18162357559@mails.ccnu.edu.cn}.}
\ \ 
Zhongyuan Liu
\thanks {School of Mathematics and Statistics,
Henan University, Kaifeng 475004, China.  Email: \texttt{liuzy@henu.edu.cn}.}
\ \  
Pingping Yang\thanks{School of Mathematics and Statistics, Central China
Normal University, Wuhan, 430079, P. R. China. Email: \texttt{ypp15623175603@mails.ccnu.edu.cn}.}
}

\date{}

\begin{document}

\maketitle

\begin{center}
		\begin{minipage}{13cm}
			\par
			\small

  {\bf Abstract:} In this paper, we consider the following fractional Schr\"{o}dinger equation 
  \begin{equation*}
    \left\{
    \begin{array}{lcl}
         (-\Delta)^{s}u+V(x)u=u^{{p_s}-\epsilon}\ \ \ &\hbox{in}\ \mathbb{R}^N,\\
         u>0\ \ \ &\hbox{in}\ \mathbb{R}^N,
    \end{array}
    \right.
\end{equation*}
where $0<s<1$,  $\epsilon>0$,  $p_s=\frac{N+2s}{N-2s}$,  $N>4s$ and $V(x)\in C^1(\mathbb{R}^N)\cap L^\infty (\mathbb{R}^N)$ is non-negative. 
We first use the Lyapunov-Schmidt reduction method to construct multi-peak solutions to the above equation provided that 
$V(x)$ possesses $k$ stable critical points. Then we  establish  
the non-degeneracy and local uniqueness of the multi-peak solutions, for $\frac{1}{2}<s<1$, $N\geq 6s$, via the blow-up argument based on various local Pohozaev identities.
Due to the  nonlocal nature  of the fractional Laplacian, we need to perform a detailed analysis of the approximate solutions and build  various  local Pohozaev identities for the 
corresponding harmonic extension  of $u$. In contrast to the local case $(s=1)$, this approach not only entails developing refined estimates for several integrals in the local Pohozaev identities, but also applying such identities in a markedly different way.

			\vskip2mm
			\par
            \end{minipage}
                
            \end{center}

\section{Introduction and main results}
\subsection{Motivation and known results}
In this paper, we are concerned with  the following fractional Schr\"{o}dinger equation involving  nearly critical exponent:
\begin{equation}\label{eq:1}
    \left\{
    \begin{array}{lcl}
         (-\Delta)^{s}u+V(x)u=u^{{p_s}-\epsilon}\ \ \ &\hbox{in}\ \mathbb{R}^N,\\
         u>0\ \ \ &\hbox{in}\ \mathbb{R}^N,
    \end{array}
    \right.
\end{equation}
where $0<s<1$,  $\epsilon>0$,   $p_s=2^*_s-1$,  $2^*_s=\frac{2N}{N-2s}$, $N>4s$ and $V(x)\in C^1(\mathbb{R}^N)\cap L^\infty (\mathbb{R}^N)$ is non-negative.  
$(-\Delta)^s$ is the fractional Laplace operator in $\mathbb{R}^N$ defined by
\begin{equation}\label{eq-def}
    (-\Delta)^s u(x)=c (N,s)P.V.\int_{\R^N}\frac{u(x)-u(z)}{|x-z|^{N+2s}}\,dz,
\end{equation}
where P.V. denotes the Cauchy principal value and $c(N, s)$ is a positive constant depending only on $N$ and $s$.

As is well known, fractional Laplacians serve as the infinitesimal generators of L\'{e}vy stable diffusion processes.
They arise in the study of anomalous diffusion phenomena, including plasma transport, flame propagation, chemical
reactions in liquids, population dynamics, geophysical fluid flows and American options in finance  (see \cite{Ap,Be}).
A prominent feature of the fractional Laplacian is its nonlocality, which makes it more challenging to handle than the classical Laplacian.
In recent years, the equations involving the fractional Laplace have attracted much attention. 
The literature concerning the fractional Laplacian is very vast, and  it is impossible
to comprehensively list all relevant works on these problems. Here we only mention several recent works on fractional Schr\"{o}dinger equations that are closely relevant to the present work.
%

    Fractional Schrödinger equations have recently attracted considerable attention; see, e.g., \cite{DDDV,La,La1}. They play a fundamental role in the study of the stability and blow-up of solitary wave solutions to nonlinear dispersive equations. 
 In the celebrated work, Frank and Lenzmann \cite{FL} considered the following fractional Schr\"{o}dinger equation in one dimension, that is, 
\begin{equation}\label{a1.3}
    (-\Delta)^s u(x)+u=u^p\quad\text{in}\quad\mathbb{R},
\end{equation}
where $p>1$. The authors established uniqueness and nondegeneracy of ground states of \eqref{a1.3}.
The counterpart of these results in higher dimensions was obtained by Frank, Lenzmann and Silvestre \cite{FLS}, see also \cite{FV} for $s$ close to $1$.
Fall, Mahmoudi and Valdinoci \cite{FMV} studied the singularly perturbed fractional Schr\"{o}dinger equation 
\begin{equation}\label{a1.4}
   \epsilon^{2s} (-\Delta)^s u(x)+V(x)u=u^p\quad\text{in}\quad\mathbb{R}^N,
\end{equation}
where $N>2s$, $1<p<p_s$. They gave some necessary and sufficient conditions on $V(x)$
 to ensure the concentration of solutions of \eqref{a1.4} as $\epsilon\rightarrow0$. More precisely, they 
 showed that the concentration points must be critical points of $V(x)$. If the potential $V(x)$ is coercive and has a unique global minimum $x_0$, 
 then ground state solutions of \eqref{a1.4} concentrate at $x_0$.
D\'{a}vila, del Pino and Wei \cite{DPW} applied the Lyapunov-Schmidt reduction method to obtain the existence 
of concentrating solutions, such as multiple spikes and clusters, under suitable assumptions on the potential $V(x)$. 
Alves and Miyagaki \cite{AlM} used the penalization method to study the existence and concentration of positive solutions to \eqref{a1.4} for general subcritical nonlinearity $f(u)$.
Deng, Peng and Yang \cite{DPY} investigated   concentrating positive solutions of \eqref{a1.4} 
with  various assumptions on the asymptotic behavior of $V(x)$ at infinity, 
including different decay rates of 
$V(x)$. For more results concerning \eqref{a1.4},  we refer the interested readers to
\cite{Am,BKS,CW, CZ,DDDV,dMS,YGuo-2020,Se,SZ} and the references therein.

In a very recent work, Cassani and Wang \cite{CaW} studied the asymptotic behavior and local uniqueness
of the ground states  of \eqref{eq:1} as $\epsilon\rightarrow0$. Specifically, the authors proved that the ground state solution $u_\epsilon$ 
blows up with the exact rate 
\[
\lim\limits_{\epsilon\rightarrow0^+}\epsilon\|u_\epsilon\|_{L^\infty}^{\frac{4s}{N-2s}}=A_{N,s}\Big(V(x_0)+\frac{1}{2s}x_0\cdot\nabla V(x_0)\Big).
\]
In general,  the maximum points $\{x_\epsilon\}$ of $u_\epsilon$ are not necessarily bounded as $\epsilon\rightarrow0$. Under certain strong assumptions on $V(x)$, they further established that 
the maximum point $x_\epsilon$ of $u_\epsilon$ converges to a global minimum point of $V(x)$. 
Furthermore, they also obtained local uniqueness of ground state solutions for radial potential $V(x)$.
It should be mentioned that the main results in \cite{CaW} can be regarded as the nonlocal analogues of those in \cite{PW92,W1}.
Nevertheless, the nonlocal feature of the fractional operator gives rise to a host of technique difficulties. To overcome the lack of 
localization, the authors developed various delicate estimates for ground state solutions. For other related results, 
see  \cite{BDG,BGM,BM17,BM19,BMS} and references therein.

\subsection{Problem setup and main results}

Let $D^s(\R^N)$ be the completion of $C_0^\infty(\R^N)$ under the norm $\|(-\Delta)^{\frac{s}{2}}u\|_{L^2(\R^N)}$, where 
$$
\|(-\Delta)^{\frac{s}{2}}u\|_{L^2(\R^N)}^2=\int_{\R^N}|\xi|^{2s}|\mathcal{F}u(\xi)|^2\,d\xi,
$$
and $\mathcal{F}u(\xi)$ is the Fourier transform of $u$. It follows from \cite{DPV} that 
$D^s(\R^N)\subseteq L^{p_s+1}(\R^N)$ and 
$$
\|(-\Delta)^{\frac{s}{2}}u\|_{L^2(\R^N)}^2=\frac{1}{2}c(N,s)\int_{\R^N}\int_{\R^N}\frac{|u(x)-u(z)|^2}{|x-z|^{N+2s}}\,dx\,dz.
$$
Then the natural energy space for problem \eqref{eq:1} is 
$$
H_V^s(\R^N)=\Big\{u\in D^s(\R^N):\int_{\R^N}V(x)u^2\,dx<+\infty\Big\},
$$
where the norm in $H_V^s(\R^N)$ is defined by 
$$\displaystyle\|u\|_{H_V^s(\R^N)}=\left(\|(-\Delta)^{\frac{s}{2}}u\|_{L^2(\R^N)}^2+\int_{\R^N}V(x)u^2\,dx\right)^{\frac{1}{2}}.$$

To bypass the nonlocal nature of the fractional Laplacian, we recall the Caffarelli-Silvestre extension \cite{CS} to lift the problem to the half-space $\mathbb{R}^{N+1}_+=\{(x,t):x\in\mathbb{R}^N, t>0\}$. For every $u\in H_V^s(\R^N)$, there exists a unique $s$-harmonic extension $\widetilde{u}$ satisfying

\begin{equation}\label{a1.5}
\left\{
    \begin{array}{ll}
         div(t^{1-2s}\nabla \widetilde{u})=0, &(x,t)\in \mathbb{R}^{N+1}_+,\\
         \widetilde{u}(x,0)=u(x), &x\in \mathbb{R}^N,
    \end{array}
    \right.
\end{equation}
and
\begin{equation}\label{a1.6}
        (-\Delta )^s u(x) =-\kappa_s^{-1}\lim\limits_{t\to 0^+}t^{1-2s}\partial_t \widetilde{u}(x,t), \ \ \ 
\end{equation}
where $\kappa_s=\frac{2^{1-2s}\Gamma(1-s)}{\Gamma(s)}$ and $\widetilde{u}(x,t)$ is defined by
$$
\widetilde{u}(x,t)=\int_{\mathbb{R}^N} P_s(x-z,t)u(z)\,dz, \ \  (x,t)\in \mathbb{R}^{N+1}_+,
$$
with
$$
P_s(x,t)=d_{s,N}\frac{t^{2s}}{(|x|^2+t^2)^{\frac{N+2s}{2}}}
$$
and the constant $d_{s,N}$ chosen such that $\int_{\mathbb{R}^N} P_s(x,1)\,dx=1$. Then, under suitable regularity assumptions, $(-\Delta)^su$ is the  Dirichlet-to-Neumann map  for this problem. 
Since the explicit value of $\kappa_s$ is not particularly relevant in our framework,  we may, for simplicity, assume that $\kappa_s=1$ throughout this paper. 

Note that  $V(x)$ is nonnegative. Then there is no variational framework  for  problem \eqref{eq:1}  on $H_V^s(\R^N)$.
Thus we cannot apply variational methods directly to find 
higher energy concentrating solutions of \eqref{eq:1} for $\epsilon>0$ sufficiently small.
%
In the present paper, we employ a perturbation approach to investigate the existence of higher energy concentrating solutions for \eqref{eq:1}. 
To present our main results, we first introduce some notations. Let
$$
U_{\lambda,\zeta}(x)={\gamma_{s,N}}\left(\frac{\lambda}{1+\lambda^2|x-\zeta|^2}\right)^{\frac{N-2s}{2}},\ \ \lambda>0, \ \  \zeta\in \mathbb{R}^N,
$$
where $$\gamma_{s,N}=\frac{2^{\frac{N-2s}{2}}\Gamma(\frac{N+2s}{2})}{\Gamma(\frac{N-2s}{2})}.$$
It is well known that $U_{\lambda,\zeta}$ are  the only positive solutions of the problem (see \cite{L})
\begin{equation}\label{eq1}
    (-\Delta)^s u=u^{\frac{N+2s}{N-2s}},\ \ u>0 \;\;\; \hbox{in}\;\ \mathbb{R}^N.
\end{equation}
Moreover, D\'avila, del~Pino and  Sire \cite{Davila-2013} established the non-degeneracy result of $U_{\lambda,\zeta}$. 
More precisely, they showed that all bounded solutions of the linearized problem
	\begin{equation}\label{eq2}
		\begin{cases}
			(-\Delta)^s Z=p_s U_{\lambda,\zeta}^{p_s-1}Z\;\; \text{in}\ \mathbb{R}^N, \\
			Z\in D^{s}(\mathbb{R}^N),
		\end{cases}
	\end{equation}
	are linear combinations of $ Z_{\lambda,\zeta}^0$ and $Z_{\lambda,\zeta}^h$, where 
\[Z_{\lambda,\zeta}^0=\frac{\partial U_{\lambda,\zeta} }{\partial\lambda},\quad  Z_{\lambda,\zeta}^h=\frac{\partial U_{\lambda,\zeta} }{\partial{\zeta_h}},\;h=1,\cdots,N.
\]

Let  $\xi_i$, $i=1,\cdots, k$ be $k$ distinct points in $\mathbb{R}^N$ and  $d=\min\limits_{ i\neq j,i,j=1,\cdots,k}|\xi_i-\xi_j|$. 
We fix $\delta=\frac{d}{10}$.  Let  $\eta$ be a smooth cut-off function  satisfying $\eta(x)=1$ for $|x|\le \delta$, 
$\eta(x)=0$ for $ |x|\ge 2\delta$ and $0\le \eta\le 1$. Setting $\eta_i(x)=\eta(x-{\xi_i})$, $i=1,\cdots,k$, we define 
$$
W_{\lambda_i, \xi_i}=\eta_iU_{\lambda_i, \xi_i},~~ {\bm W}_{\bm \lambda, \bm\xi}=\sum_{i=1}^kW_{\lambda_i, \xi_i},
$$
where $\bm\lambda=(\lambda_1,\cdots, \lambda_k)\in\R^k$, $\bm\xi=(\xi_1,\cdots, \xi_k)\in\R^{kN}$.

Our first result concerning the existence of multi-peak solutions is the following.
\begin{theorem}\label{Mth}
    Let $N>4s$ and $0<s<1$. Assume that $V(x)$ is a bounded nonnegative $C^1(\mathbb{R}^N)$ function and $\xi^*_i$, $i=1,2,\cdots,k$ are the $k$ different stable critical points of $V(x)$ with $V(\xi^*_i)>0$. Then there exists  $\epsilon_0>0$ such that for $\epsilon\in (0,\epsilon_0)$, 
    problem \eqref{eq:1} admits a $k$-peak
solution $u_\epsilon$ with  the form 
\begin{equation}\label{1a}
u_\epsilon={\bm W_{\bm\lambda_\epsilon,\bm\xi_\epsilon}}+\phi_\epsilon=\sum_{i=1}^kW_{\lambda_i^\epsilon,\xi_i^\epsilon}+\phi_\epsilon,
\end{equation}
where $\bm\lambda_\epsilon=(\lambda_1^\epsilon,\cdots,\lambda_k^\epsilon)$, $\bm\xi_\epsilon=(\xi_1^\epsilon,\cdots,\xi_k^\epsilon)$ satisfying, for $i=1,2,\cdots,k$, $\lambda_i^\epsilon\approx\epsilon^{-\frac{1}{2s}}$, 
$\xi_i^\epsilon\to \xi_i^*$ and $\|\phi_\epsilon\|_*\to 0$ as $\epsilon\to 0$ $($the norm $\|\cdot\|_*$ is defined in \eqref{eq-norm*}$)$.
\end{theorem}

\begin{remark}
Compared with the results in \cite{CaW}, our assumptions  on the potential $V(x)$ here are substantially weaker. 
Indeed, in contrast to \cite{CaW}, we do not require $V(x)$ to be  positive or of class  $C^2$.
In fact, we only need certain local conditions near the critical points of $V(x)$.  
\end{remark}

\begin{remark}
To address the slow decay of $U_{\lambda,\xi}$ in low dimensions $({N\leq6s})$,  
we suitably truncate these functions near the concentration points, which yields sufficiently accurate approximate solutions.
\end{remark}

The second part of this paper focuses on the qualitative properties of the multi-peak solutions given in Theorem \ref{Mth}. 
Generally speaking,  the qualitative properties of concentrating solutions to nonlinear elliptic problems have long been a classic 
and appealing research topic,
see \cite{CH,CLL,CLP,CPY,CeG,DLY,Gl,G05,Gp,GS,GMPY,LPP,LTZ,LPW20}. 
Concerning \eqref{eq:1}, Cassani and Wang \cite{CaW}  established the local uniqueness of ground state solutions for radial potentials
 $V(|x|)$. Motivated by  \cite{CaW}, we  proceed to investigate the qualitative properties of the multi-peak solutions. 
Precisely,  we mainly deal with the non-degeneracy and local uniqueness of multi-peak solutions. 
For this purpose,  we introduce  the linearized problem  of \eqref{eq:1} at the peak solution $u_\epsilon$:
  \begin{equation}\label{eqNDin}
          \begin{cases}
             (-\Delta)^s\omega_\epsilon+V(x)\omega_\epsilon-(p_s-\epsilon)u_\epsilon^{p_s-1-\epsilon}\omega_\epsilon=0 \quad\text{in}\ \R^N,\\
             \omega_\epsilon\in H_V^s(\mathbb{R}^N).
             \end{cases}
         \end{equation}  
We say that the solution $u_\epsilon$ is non-degenerate if  
any solution $\omega_\epsilon$ to  \eqref{eqNDin} satisfies $\omega_\epsilon\equiv0$. 
    
Now we state the result on the non-degeneracy of multi-peak solutions.

\begin{theorem}\label{NDth}
    Let $N\geq6s$, $\frac{1}{2}<s<1$. Assume that $V(x)$ is a bounded nonnegative $C^1(\mathbb{R}^N)$ function and $\xi^*_i$, $i=1,2,\cdots,k$ are the $k$ different   non-degenerate critical points  of $V(x)$ with $V(\xi^*_i)>0$ and $V(x)\in C^2(B_{5\delta}(\xi^*_i))$.
Then there exists  $\epsilon_0>0$ such that for $\epsilon\in (0,\epsilon_0)$, 
     the $k$-peak solutions $u_\epsilon$ constructed in Theorem \ref{Mth} are non-degenerate. 
\end{theorem}

Next, we give the local uniqueness result for the multi-peak solutions in Theorem \ref{Mth}.
 We define the local uniqueness of the peak solution $u_\epsilon$
as follows: 
if  $u_\epsilon^{(1)}$ and $u_\epsilon^{(2)}$ are any two peak solutions to \eqref{eq:1} taking the form \eqref{1a}, then $u_\epsilon^{(1)} \equiv u_\epsilon^{(2)}$.



\begin{theorem}\label{thm1.3}
Suppose  $N\geq 6s$, $\frac{1}{2}<s<1$ and  $V(x)$ satisfies the same  assumptions  as in Theorem \ref{NDth}. 
 Then there exists  $\epsilon_0>0$ such that for $\epsilon\in (0,\epsilon_0)$, the $k$-peak solutions $u_\epsilon$
 with the form  \eqref{1a} is locally unique.
 \end{theorem}
 
 \begin{remark}
 In contrast to Theorem \ref{Mth},  stronger regularity conditions on  $V(x)$ near its critical points are 
 required to guarantee the non-degeneracy and local uniqueness of  the peak solutions. 
 These stronger conditions are needed to derive refined estimates for $u_\epsilon$. 
 To be precise, we employ  blow-up analysis combined with  several local Pohozaev identities to establish  refined estimates
 of $\bm\lambda_\epsilon, \bm\xi_\epsilon$ and $\phi_\epsilon$, which has its own interest. 
\end{remark}


\subsection{ The strategy of the proof}

The proof of Theorem\,\ref{Mth} relies on the finite-dimensional Lyapunov-Schmidt reduction method. As is well known, 
this method allows us to reduce the construction of solutions to a finite dimensional
variational problem. Roughly speaking, the finite-dimensional reduction argument consists of two key steps.
The first step is to construct appropriate approximate solutions, thereby reducing the original problem to a finite-dimensional one.
Here to obtain  sufficiently accurate approximate solutions, we need to appropriately truncate  the standard bubble $U_{\lambda,\xi}$ near the concentrating points, 
which is crucial to address the slow decay of the bubble in low dimensions. However, this poses substantial challenges in deriving estimates of the approximate solutions, owing to the non-locality of the fractional Laplacian.
Unlike the local case $(s=1)$, we perform a refined analysis of the approximate solutions.
 Moreover, although problem \eqref{eq:1} can be regarded  as a subcritical problem,  the fact that  $V(x)$ is merely  nonnegative 
 implies that \eqref{eq:1}  generally admits no variational structure  on the  Sobolev space $H_V^s(\mathbb{R}^N)$, 
 and thus variational methods cannot be applied directly to  \eqref{eq:1}.
 For this reason, we work in appropriate weighted $L^\infty$ spaces, following the approach in \cite{WY1}, to carry out the finite-dimensional reduction procedure.
 
The second step is to solve the corresponding reduced finite-dimensional  problem, thereby obtaining a genuine solution.
 The standard approach for this step is to derive sharp estimates for the error term 
$\phi_\epsilon$  obtained in the first step,  which works well for high dimensions  $(N>6s)$.
 In the current setting, the estimate of the error term is sufficiently accurate for dilation parameters $\lambda_i$ 
$(i=1,\cdots,k)$, but it is not a higher order term with respect to the translation parameters  $\xi_i$  $(i=1,\cdots,k)$. 
To overcome this obstacle, we will adopt  the local Pohozaev identities technique developed in \cite{CLL1,PWY}  to solve the corresponding  reduced finite-dimensional  problem.
Our results demonstrate that translation and dilation parameters play distinct roles in the construction of concentrating solutions.
 We expect that  the novel ideas and methods developed in this paper can be applied to the construction of peak solutions for  a broad class of nonlocal and higher-order elliptic equations with critical nonlinearities.

We will employ  the blow-up analysis based on Pohozaev identities to establish  the non-degeneracy and local uniqueness of peak solutions. 
It is well-known that blow-up analysis is a fundamental technique to obtain deeper insights into asymptotic behavior of solutions. 
 However, the standard blow-up analysis framework is not applicable in the present setting.
 To study qualitative properties of peak solutions, we exploit concentration properties of such solutions to derive sharp estimates for them. 
 Since Pohozaev identities reveal essential relationships inherent to peak solutions, 
 we utilize this blow-up approach combined with several local Pohozaev identities to obtain the desired refined estimates for peak solutions.
 Furthermore, unlike the case of bounded domains, we refine these estimates via a suitable projection, rather than relying on the Green function. 
Our findings shed new light on the qualitative properties of peak solutions for  general critical  fractional elliptic problems, and that the techniques 
 developed herein  have great potential for further extension and wide applicability to related research problems.

The rest of the paper is organized as follows. In Section 2,  we compute the fractional Laplacians of the approximate solutions. 
 In Sections 3, we perform the Lyapunov-Schmidt reduction argument to reduce the construction of peak solutions to
a finite dimensional variational problem. 
In Section 4, we solve the reduced finite-dimensional problem via local Pohozaev identities and prove Theorem \ref{Mth}.
 Section 5 establishes the non-degeneracy of peak solutions (Theorem \ref{NDth}), and 
 Section 6 proves the local uniqueness of such solutions (Theorem \ref{thm1.3}).
 The energy expansions and auxiliary technical estimates are collected   in Appendices A and B. Unless otherwise stated, $C$ 
 denotes a generic positive constant independent of  $\epsilon,\;\bm \lambda$ and $\bm\xi$.


\section{Preliminaries}

In this section, we compute the fractional Laplacians of our approximate solutions. 
This calculation plays an important role in deriving estimates for the error term and the energy expansions.
Recall that  ${\bm W}_{\bm\lambda, \bm\xi}=\displaystyle\sum_{i=1}^kW_{\lambda_i, \xi_i}$ are approximate solutions.
To  compute $(-\Delta)^s {\bm W}_{\bm\lambda, \bm\xi}$,  it suffices to compute $(-\Delta)^s W_{\lambda_i, \xi_i}$ for  $i=1,\cdots,k$.

\begin{lemma}\label{lem-J}
Let $0<s<1$ and $N> 4s$.  Then for each $i=1,\cdots,k$ and $x\in \mathbb{R}^N$,
\[
(-\Delta)^s W_{\lambda_i, \xi_i}(x)=\eta_i(x) U_{\lambda_i, \xi_i}^{p_s}(x)+\Upsilon_i(x)
\]
and
$$
|\Upsilon_i(x)|\leq C\frac{\lambda_i^{-\frac{N-2s}{2}}}{(1+|x-\xi_i|)^{N+2s}},
$$
where $C>0$ is a constant independent of $\lambda_i$ and $\xi_i$.
\end{lemma}

\begin{proof}
The proof is well-known (see \cite{YGuo-2020}), and  we include it here for completeness.
In view of \eqref{eq-def}, we have 
\begin{equation}\label{eq-0426-1}
\begin{aligned}
    &\ \quad(-\Delta)^s W_{\lambda_i, \xi_i}(x)
    \vspace{2mm}\\
    &\ =\displaystyle c(N,s)\lim_{\varepsilon\to 0^+}\int_{\R^N\setminus B_\varepsilon(x)}\Bigg(\frac{\eta_i(x) \Big(U_{\lambda_i, \xi_i}(x)- U_{\lambda_i, \xi_i}(z)\Big)}{|x-z|^{N+2s}}\,dz+\frac{\Big(\eta_i(x) -\eta_i(z)\Big) U_{\lambda_i, \xi_i}(z)}{|x-z|^{N+2s}}\Bigg)
    \vspace{2mm}\\
    &\ =\eta_i(x)(-\Delta)^s U_{\lambda_i, \xi_i}(x)+\displaystyle c(N,s)\lim_{\varepsilon\to 0^+}\int_{\R^N\setminus B_\varepsilon(x)}\frac{\Big(\eta_i(x) -\eta_i(z)\Big) U_{\lambda_i, \xi_i}(z)}{|x-z|^{N+2s}}
    \vspace{2mm}\\
    &\ =\eta_i(x) U_{\lambda_i, \xi_i}^{p_s}(x)+\displaystyle c(N,s)\lim_{\varepsilon\to 0^+}\int_{\R^N\setminus B_\varepsilon(x)}\frac{\Big(\eta_i(x) -\eta_i(z)\Big) U_{\lambda_i, \xi_i}(z)}{|x-z|^{N+2s}}
    \vspace{2mm}\\
    &:=\eta_i(x) U_{\lambda_i, \xi_i}^{p_s}(x)+\Upsilon_i(x).
\end{aligned}
\end{equation}
Moreover,
    \begin{equation}
        \begin{aligned}
            \Upsilon_i(x)
            &\ =\displaystyle c(N,s)\left(\lim_{\varepsilon\to 0^+}\int_{B_{\frac{\delta}{4}}(x)\setminus B_\varepsilon(x)}\frac{\Big(\eta_i(x) -\eta_i(z)\Big) U_{\lambda_i, \xi_i}(z)}{|x-z|^{N+2s}}\right)
            \vspace{2mm}\\
            &\ \quad
            +\displaystyle c(N,s)\left(\int_{\R^N\setminus B_{\frac{\delta}{4}}(x)}\frac{\Big(\eta_i(x) -\eta_i(z)\Big) U_{\lambda_i, \xi_i}(z)}{|x-z|^{N+2s}}\right)
            \vspace{2mm}\\
            &:=\displaystyle c(N,s)\left(\Upsilon_{i,1}(x)+\Upsilon_{i,2}(x)\right).
        \end{aligned}
    \end{equation}
Then  we have 
$$
\Upsilon_{i,1}(x)=0  \ \hbox{in}\  B_{\frac{\delta}{2}}(\xi_i)\cup (\R^N\setminus B_{\frac{5}{2}\delta}(\xi_i)). 
$$ 
It remains to consider the annular region $x\in B_{\frac{5}{2}\delta}(\xi_i)\setminus B_{\frac{\delta}{2}}(\xi_i)
$. For $z\in B_{\frac{\delta}{4}}(x)\setminus B_\varepsilon(x)$,
$$
  \frac{\delta}{4}\leq \frac{|x-\xi_i|}{2}\leq |z-\xi_i|\leq  \frac{3}{2}|x-\xi_i|\leq  \frac{15}{4}\delta .
$$
By the mean value theorem, we deduce 
  \begin{align*}
  &|\Upsilon_{i,1}(x)|\\ \leq&  C\left|\lim_{\varepsilon\to 0^+}\int_{B_{\frac{\delta}{4}}(x)\setminus B_\varepsilon(x)}\frac{\left[\nabla \eta_i(x)\cdot(x-z)+O(|x-z|^2) \right]U_{\lambda_i, \xi_i}(z)}{|x-z|^{N+2s}}\right|\\
   \leq& C\left|\lim_{\varepsilon\to 0^+}\int_{B_{\frac{\delta}{4}}(0)\setminus B_\varepsilon(0)}\frac{\nabla \eta_i(x)\cdot z }{|z|^{N+2s}}\frac{\lambda_i^{\frac{N-2s}{2}}}{(1+\lambda_i^2|-z+x-\xi_i|^2)^{\frac{N-2s}{2}}}\right|
  +C\frac{\lambda_i^{\frac{N-2s}{2}}}{(1+\lambda_i|x-\xi_i|)^{N-2s}}\\
   \leq& C\left|\lim_{\epsilon\to 0^+}\int_{B_{\frac{\delta}{4}}(0)\setminus B_\varepsilon(0)}\frac{\nabla \eta_i(x)\cdot z }{|z|^{N+2s}}\left(\frac{\lambda_i^{\frac{N-2s}{2}}}{(1+\lambda_i^2|-z+x-\xi_i|^2)^{\frac{N-2s}{2}}}-\frac{\lambda_i^{\frac{N-2s}{2}}}{(1+\lambda_i^2|z+x-\xi_i|^2)^{\frac{N-2s}{2}}}\right)\right|\\
  &+C\frac{\lambda_i^{\frac{N-2s}{2}}}{(1+\lambda_i|x-\xi_i|)^{N-2s}}\\
  \leq&  C\lambda_i^{\frac{N-2s}{2}}\int_{B_{\frac{\delta}{4}}(0)}\frac{|\nabla \eta_i(x)|| z| }{|z|^{N+2s}}\frac{\lambda_i^2|(1-2\theta)z+x-\xi_i| |z|}{(1+\lambda_i^2|(1-2\theta)z+x-\xi_i|^2)^{\frac{N-2s+2}{2}}}+C\frac{\lambda_i^{\frac{N-2s}{2}}}{(1+\lambda_i|x-\xi_i|)^{N-2s}}\\
  \leq&  C\lambda_i^{\frac{N-2s+2}{2}}\int_{B_{\frac{\delta}{4}}(0)}\frac{1 }{|z|^{N+2s-2}}\frac{1}{(1+\lambda_i|(1-2\theta)z+x-\xi_i|)^{N-2s+1}}+C\frac{\lambda_i^{\frac{N-2s}{2}}}{(1+\lambda_i|x-\xi_i|)^{N-2s}}\\
  \leq& C\frac{\lambda_i^{\frac{N-2s+2}{2}}}{(1+\lambda_i|x-\xi_i|)^{N-2s+1}}+C\frac{\lambda_i^{\frac{N-2s}{2}}}{(1+\lambda_i|x-\xi_i|)^{N-2s}}\\
  \leq&  C\frac{\lambda_i^{\frac{N-2s}{2}}}{(1+\lambda_i|x-\xi_i|)^{N-2s}}\leq C\lambda_i^{-\frac{N-2s}{2}},\\
  \end{align*}
  where $0<\theta<1$. Here we have used the fact that $|(1-2\theta)z+x-\xi_i|\geq |x-\xi_i|-|(1-2\theta)z|\geq \frac{1}{2}|x-\xi_i|\geq \frac{\delta}{4}$ for $z\in B_{\frac{\delta}{4}}(0)$.

 Next we estimate  $\Upsilon_{i,2}(x)$.\\
  Case 1: Suppose that $x\in B_\delta(\xi_i)$.  Then  $\eta_i(x)=1$  and  
\begin{equation}
      \begin{aligned}
          |\Upsilon_{i,2}(x)|\leq & C\int_{\R^N\setminus (B_{\delta/{4}}(x)\cup B_\delta(\xi_i))}\frac{1}{|x-z|^{N+2s}}\frac{\lambda_i^{\frac{N-2s}{2}}}{(1+\lambda_i^2|z-\xi_i|^2)^{\frac{N-2s}{2}}}\\
          \leq & C\lambda_i^{-\frac{N-2s}{2}}\int_{\R^N\setminus B_{\delta/{4}}(x)}\frac{1}{|x-z|^{N+2s}}\\
           \leq & C\lambda_i^{-\frac{N-2s}{2}}.
      \end{aligned}
  \end{equation}
 Case 2: For $\delta\leq |x-\xi_i|\leq 3\delta$, it follows from Lemma \ref{lemGA.3}  that
 \begin{equation}
     \begin{aligned}
          |\Upsilon_{i,2}(x)|\leq & C\int_{\R^N\setminus B_{\delta/{4}}(x)}\frac{1}{|x-z|^{N+2s}}\frac{\lambda_i^{\frac{N-2s}{2}}}{(1+\lambda_i^2|z-\xi_i|^2)^{\frac{N-2s}{2}}}\\
          \leq & C\lambda_i^{\frac{N+2s}{2}}\int_{\R^N\setminus B_{\frac{\lambda_i\delta}{4}}(\lambda_i (x-\xi_i))}\frac{1}{|\lambda_i(x-\xi_i)- z|^{N+2s}}\frac{1}{(1+|z|)^{N-2s}}\\
          \leq & C\lambda_i^{\frac{N+2s}{2}}\left(\frac{1}{(1+\lambda_i|x-\xi_i|)^{N}}+\frac{1}{\lambda_i^{2s}(1+\lambda_i|x-\xi_i|)^{N-2s}}\right)\\
          \leq &C\lambda_i^{-\frac{N-2s}{2}}.
     \end{aligned}
 \end{equation}\\
 Case 3: If $|x-\xi_i|>3\delta$, then $\eta_i(x)$ and $\eta_i(z)$ vanish for $z\in \R^N\setminus B_{2\delta}(\xi_i)$. Hence we obtain 
 \begin{equation}
     \begin{aligned}
         |\Upsilon_{i,2}(x)|\leq & C\int_{B_{2\delta}(\xi_i)}\frac{1}{|x-z|^{N+2s}}\frac{\lambda_i^{\frac{N-2s}{2}}}{(1+\lambda_i^2|z-\xi_i|^2)^{\frac{N-2s}{2}}}
         \vspace{2mm}\\
         \leq & C\lambda_i^{-\frac{N-2s}{2}}\int_{B_{2\delta}(\xi_i)}\frac{1}{|x-z|^{N+2s}}\frac{1}{|z-\xi_i|^{N-2s}}
         \vspace{2mm}\\
          \leq & C\lambda_i^{-\frac{N-2s}{2}}\frac{1}{(1+|x-\xi_i|)^{N+2s}}\int_{B_{2\delta}(\xi_i)}\frac{1}{|z-\xi_i|^{N-2s}}\\
          \leq & C\lambda_i^{-\frac{N-2s}{2}}\frac{1}{(1+|x-\xi_i|)^{N+2s}}\\      
     \end{aligned}
 \end{equation}
 where we used the fact that $|x-z|\geq \frac{1}{3}|x-\xi_i|$ for $x\in \R^N\setminus B_{3\delta}(\xi_i)$ and $z\in B_{2\delta}(\xi_i)$. 
As a result, the desired estimate can be derived from the above arguments. 
\end{proof}
Define
 \begin{equation}\label{eqA.2.16}
Z_{i,0}=\frac{\partial W_{\lambda_i, \xi_i}}{\partial \lambda_i}, \ \;\   Z_{i,l}=\frac{\partial W_{\lambda_i, \xi_i}}{\partial \xi_{i,l}},
 \end{equation}
 where $\xi_{i,l}$ is the $l$-th component of $\xi_i\in \mathbb{R}^N$, $i=1,\cdots,k$,  $l=1,2,\cdots,N$.
Concerning the  fractional Laplacians of $Z_{i,0}$ and $Z_{i,l}$, we have
\begin{lemma}\label{lemA}
    If $0<s<1$ and $N>4s$, then
 \begin{itemize}
\item[$(i)$] 
$$
  \begin{aligned}
    &\quad \ (-\Delta)^s \left(\eta_i\frac{\partial U_{\lambda_i, \xi_i}}{\partial\lambda_i}\right)(x)
   =p_s\eta_i(x) U_{\lambda_i, \xi_i}^{p_s-1}\frac{\partial U_{\lambda_i, \xi_i}}{\partial\lambda_i}(x)+J_{i,0}(x),
    \end{aligned}
    $$
    and 
    \begin{equation}\label{eqA.1}
|J_{i,0}(x)|\leq C\frac{\lambda_i^{-\frac{N+2-2s}{2}}}{(1+|x-\xi_i|)^{N+2s}},
\end{equation}
where $J_{i,0}(x)=\displaystyle c(N,s)\lim_{\varepsilon\to 0^+}\int_{\R^N\setminus B_\varepsilon(x)}\frac{\Big(\eta_i(x) -\eta_i(z)\Big) \frac{\partial U_{\lambda_i, \xi_i}}{\partial\lambda_i}(z)}{|x-z|^{N+2s}}$.
\item [$(ii)$]
$$
\begin{aligned}
    &(-\Delta)^s\left(\eta_i\frac{\partial U_{\lambda_i, \xi_i}}{\partial \xi_{i,l}}\right)(x)
   =p_s\eta_i(x) U_{\lambda_i, \xi_i}^{p_s-1}\frac{\partial U_{\lambda_i, \xi_i}}{\partial \xi_{i,l}}(x)+J_{i,l}(x),
\end{aligned}
$$
and 
    \begin{equation}\label{eqA.2}
|J_{i,l}(x)|\leq C\frac{\lambda_i^{-\frac{N-2-2s}{2}}}{(1+|x-\xi_i|)^{N+2s}},\ \ i=1,2,\cdots,k,\ \ l=1,2,\cdots,N,
\end{equation}
where $J_{i,l}(x)=\displaystyle c(N,s)\lim_{\varepsilon\to 0^+}\int_{\R^N\setminus B_\varepsilon(x)}\frac{\Big(\eta_i(x) -\eta_i(z)\Big) \frac{\partial U_{\lambda_i, \xi_i}}{\partial \xi_i^l}(z)}{|x-z|^{N+2s}}$.
\item[$(iii)$] 
$$
\begin{aligned}
    &(-\Delta)^s\left(U_{\lambda_i, \xi_i}\frac{\partial \eta_i }{\partial \xi_{i,l}}\right)(x)
   =\frac{\partial \eta_i }{\partial \xi_{i,l}}(x)U_{\lambda_i, \xi_i}^{p_s}(x)+\Lambda_{i,l}(x),
\end{aligned}
$$
and 
    \begin{equation}\label{eqA.3}
|\Lambda_{i,l}(x)|\leq C\frac{\lambda_i^{-\frac{N-2s}{2}}}{(1+|x-\xi_i|)^{N+2s}},\ \ i=1,2,\cdots,k,\ \ l=1,2,\cdots,N,
\end{equation}
where $\Lambda_{i,l}(x)=\displaystyle c(N,s)\lim_{\varepsilon\to 0^+}\int_{\R^N\setminus B_\varepsilon(x)}\frac{\Big(\frac{\partial \eta_i}{\partial \xi_{i,l}}(x) -\frac{\partial \eta_i}{\partial \xi_{i,l}}(z)\Big) U_{\lambda_i, \xi_i}(z)}{|x-z|^{N+2s}}$.
\end{itemize}
\end{lemma}

\begin{proof}
    Note that  $$\left|\frac{\partial U_{\lambda_i, \xi_i}}{\partial\lambda_i}(x)\right|\leq C\lambda_i^{-1}U_{\lambda_i, \xi_i}\;\;\;
      \left|\frac{\partial U_{\lambda_i, \xi_i}}{\partial \xi_{i,l}}(x)\right|\leq C\lambda_i U_{\lambda_i, \xi_i}.$$
 Then      estimates \eqref{eqA.1} and \eqref{eqA.2} follow by the same argument as in Lemma \ref{lem-J}. 
 Since $\frac{\partial \eta_i }{\partial \xi_{i,l}}=0$ in $B_\delta(\xi_i)\cup(\R^N\setminus B_{2\delta}(\xi_i))$ and $|\nabla \eta_i|\leq C$, we can obtain estimate \eqref{eqA.3} analogously.
\end{proof}
	Finally, we give some well-known elementary inequalities that will be used in the subsequent sections.
		\begin{lemma}\label{lemL2.8}
			Let $q>0$. Then for $a>0$, $b\in\mathbb{R}$, we have 
			\begin{equation}\label{aa2.15}
				|(a+b)_+^q-a^q|\leq C\begin{cases}
					\min\{a^{q-1}|b|,|b|^q\},\;&0<q<1,\\
					a^{q-1}|b|+|b|^q,\;&q\geq1,
				\end{cases}
			\end{equation}
			\begin{equation}\label{aa2.16}
				|(a+b)_+^q-a^q-qa^{q-1}b|\leq C\begin{cases}
					\min\{a^{q-2}|b|^2,|b|^q\},\;&0<q<2,\\
					a^{q-2}|b|^2+|b|^q,\;&q\geq2.
				\end{cases}
			\end{equation}
		\end{lemma}
		
\section{The Lyapunov–Schmidt reduction}
The aim of this section is to carry out  the Lyapunov-Schmidt
reduction argument to reduce the construction of peak solutions to a finite dimensional variational problem.
To begin with, we introduce the following weighted norms
\begin{equation}\label{eq-norm*}
\|\phi\|_*=\sup\limits_{x\in \mathbb{R}^N}\left(\sum_{j=1}^k\frac{\lambda_j^{\frac{N-2s}{2}}}{(1+\lambda_j|x-\xi_j|)^{\frac{N-2s}{2}+\sigma}}\right)^{-1}|\phi(x)|,
\end{equation}
and
$$
\|h\|_{**}=\sup\limits_{x\in \mathbb{R}^N}\left(\sum_{j=1}^k\frac{\lambda_j^{\frac{N+2s}{2}}}{(1+\lambda_j|x-\xi_j|)^{\frac{N+2s}{2}+\sigma}}\right)^{-1}|h(x)|,
$$
where $\lambda_j>0$, $\xi_j\in \mathbb{R}^N$, $j=1,2, \cdots, k$, $
   \sigma\in \Big(0,\min\big\{\frac{s}{2},\frac{N-4s}{2}\big\}\Big).
$

We denote by $ C_{*}(\R^N)$ and $ C_{**}(\R^N)$ the spaces of continuous functions on $ \R^N$ with finite norms $\|\cdot\|_*$ and $\|\cdot\|_{**}$, respectively. 

\begin{definition}[\cite{CLL1}]
    Let ${\bm \lambda}=(\lambda_1, \cdots,\lambda_k)\in (\mathbb{R_+})^k$. We say that $\bm \lambda$ goes to infinity equivalently if $\lambda_i, \ i=1,\cdots,k$ goes to infinity and there exists a positive constant $C$ satisfying
$$
\frac{1}{C}\le \frac{\lambda_i}{\lambda_j}\le C,\ \  i,j=1,\cdots,k.
$$
\end{definition}
Now we consider the linear
problem
\begin{equation}\label{eq:2}
    \left\{
    \begin{array}{lcl}
         (-\Delta)^{s}\phi+V(x)\phi-p_s{\bm W}_{\bm \lambda, \bm \xi}^{p_s-1}\phi=h+\sum\limits_{i=1}^k\sum\limits_{l=0}^N c_i^l W_{\lambda_i, \xi_i}^{p_s-1}Z_{i,l}\ \ \hbox{in}\ \mathbb{R}^N,\\
         \displaystyle \int_{\mathbb{R}^N} W_{\lambda_i, \xi_i}^{p_s-1}Z_{i,l}\phi=0,\ \ i=1,2,\cdots,k, \ \  l=0,1,\cdots, N,
    \end{array}
    \right.
\end{equation}
where  $c_i^l$ are constants.
\begin{lemma}\label{l1}
    Let $\bm \lambda_n$, $\bm \xi_n$ be sequences such that $\bm \lambda_n$ goes to infinity equivalently,  $\bm \xi_n\to \bm \xi^*$ as $n\to \infty$ and $\phi_n$ solve equation \eqref{eq:2} with $h=h_n$, $\bm \xi=\bm \xi_n$ and $\bm \lambda=\bm \lambda_n$. If $\|h_n\|_{**}\to 0$ as $n\to \infty$, then $\|\phi_n\|_{*}\to 0$.
\end{lemma}
\begin{proof}
    We argue by contradiction. Suppose that there exist sequences $\bm \lambda_n\to \infty$, $\bm \xi_n\to \bm \xi^*$ and $\phi_n$ solving $\eqref{eq:2}$ for $h=h_n$, $\bm \lambda=\bm \lambda_n$ and $\bm \xi= \bm \xi_n$ with $\|h_n\|_{**}\to 0$  and $\|\phi_n\|_{*}=1$. For simplicity, we omit the subscript $n$. \\
\textbf{Step 1.} Prove $c_i^0=o(\lambda_1)$, $c_i^l=o(\lambda^{-1}_1)$, $i=1,\cdots,k$, $l=1,\cdots, N$.\\
Multiplying \eqref{eq:2} by $Z_{j,m}$, $j=1,\cdots,k$, $m=0,1,\cdots,N$, we have
\begin{equation}\label{eq:3}
    \sum_{i=1}^k\sum_{l=0}^Nc_i^l\int_{\mathbb{R}^N} W_{\lambda_i,\xi_i}^{p_s-1} Z_{i,l}Z_{j,m}=\left<(-\Delta)^s\phi+V(x)\phi-p_s{\bm W}_{\bm \lambda, \bm \xi}^{p_s-1}\phi, Z_{j,m}\right>-\left<h, Z_{j,m}\right>,
\end{equation}
where $\langle u, v\rangle=\int_{\mathbb{R}^N} uv$.
Then it follows from direct computation that
\begin{equation}
    \begin{aligned}
         \displaystyle \bigg|\int_{\mathbb{R}^N} Z_{j,0}h\bigg|&=\displaystyle \bigg|\int_{B_{2\delta}(\xi_j)} h\eta_j\frac{\partial U_{\lambda_j,\xi_j}}{\partial \lambda_j}\bigg|
         \vspace{3mm}\\
         &\le  \displaystyle C\|h\|_{**}\sum\limits_{i=1}^k \int_{B_{2\delta}(\xi_j)} \frac{\lambda_i^{\frac{N+2s}{2}}}{(1+\lambda_i|x-\xi_i|)^{\frac{N+2s}{2}+\sigma}}\frac{\lambda_j^{\frac{N-2s}{2}-1}}{(1+\lambda_j|x-\xi_j|)^{N-2s}}
         \vspace{3mm}\\
         &\le \displaystyle C\|h\|_{**} \int_{\mathbb{R}^N} \frac{\lambda_j^{N-1}}{(1+\lambda_j|x-\xi_j|)^{\frac{3N}{2}-s+\sigma}}\le C\lambda_j^{-1}\|h\|_{**}.
    \end{aligned}
\end{equation}
 Since
$$
\frac{1}{\lambda_j}\le \frac{C}{1+\lambda_j|x-\xi_j|}\ \ \hbox{for}\;\;x\in B_{2\delta}(\xi_j),
$$
 we obtain 
\begin{equation}
    \begin{aligned}
         \displaystyle 
         \bigg|\int_{\mathbb{R}^N} V(x)Z_{j,0}\phi\bigg|
         &\le \displaystyle C\|\phi\|_{*}\sum\limits_{i=1}^k \int_{B_{2\delta}(\xi_j)} \frac{\lambda_i^{\frac{N-2s}{2}}}{(1+\lambda_i|x-\xi_i|)^{\frac{N-2s}{2}+\sigma}}\frac{\lambda_j^{\frac{N-2s}{2}-1}}{(1+\lambda_j|x-\xi_j|)^{N-2s}} 
         \vspace{3mm}\\
         &\le  \displaystyle C\|\phi\|_{*} \int_{B_{2\delta}(\xi_j)} \frac{\lambda_j^{N-1-2s}}{(1+\lambda_j|x-\xi_j|)^{\frac{3N}{2}-3s+\sigma}}
         \vspace{3mm}\\
         &\le  \displaystyle C\|\phi\|_{*}\lambda_j^{-1-s} \int_{B_{2\delta}(\xi_j)} \frac{\lambda_j^{N}}{(1+\lambda_j|x-\xi_j|)^{\frac{3N}{2}-2s+\sigma}}
         \vspace{3mm}\\
         &\le  \displaystyle C\|\phi\|_{*}\lambda_j^{-1-s}.
    \end{aligned}
\end{equation}
 Moreover, by Lemma \ref{lemA}, we have 
\begin{equation}
\begin{aligned}
     &\bigg|\left<(-\Delta)^s\phi-p_s{\bm W}_{\bm \lambda, \bm \xi}^{p_s-1}\phi, Z_{j,0}\right>\bigg|=\bigg|\left<(-\Delta)^sZ_{j,0}-p_s{\bm W}_{\bm \lambda, \bm \xi}^{p_s-1}Z_{j,0}, \phi\right>\bigg|
     \vspace{3mm}\\
     =&\displaystyle \bigg|\left<p_sU^{p_s-1}_{\lambda_j,\xi_j}Z_{j,0}+J_j^0-p_s{\bm W}_{\bm \lambda, \bm \xi}^{p_s-1}Z_{j,0}, \phi\right>\bigg|
     \vspace{3mm}\\
     \le &\displaystyle C\lambda_j^{-1} \int_{B_{2\delta}(\xi_j)\setminus B_{\delta}(\xi_j)} U^{p_s}_{\lambda_j,\xi_j}|\phi|+C \int_{\R^N}\frac{\lambda_j^{-\frac{N+2-2s}{2}}}{(1+|x-\xi_j|)^{N+2s}}|\phi|
     \vspace{3mm}\\
     \leq& C\lambda_j^{-1}\|\phi\|_*\sum_{i=1}^k\int_{B_{2\delta}(\xi_j)\setminus B_{\delta}(\xi_j)} \frac{\lambda_i^{\frac{N-2s}{2}}}{(1+\lambda_i|x-\xi_i|)^{\frac{N-2s}{2}+\sigma}}\frac{\lambda_j^{\frac{N+2s}{2}}}{(1+\lambda_j|x-\xi_j|)^{N+2s}}\\
     &+C\|\phi\|_*\sum_{i=1}^k\int_{\R^N}\frac{\lambda_j^{-1+2s}}{(1+\lambda_j|x-\xi_j|)^{\frac{N+2s}{2}}}\frac{\lambda_i^{\frac{N-2s}{2}}}{(1+\lambda_i|x-\xi_i|)^{\frac{N-2s}{2}+\sigma}}\\
     \le&\displaystyle C\lambda_j^{-\frac{N-2s}{2}-1}\|\phi\|_*.
\end{aligned}
\end{equation}
A direct calculation yields  that
$$
\begin{aligned}
     \displaystyle \bigg|\int_{\mathbb{R}^N} Z_{j,m}h\bigg|&\le C\lambda_j\|h\|_{**},
     \vspace{3mm}\\
     \displaystyle \bigg|\int_{\mathbb{R}^N} V(x)Z_{j,m}\phi\bigg|&\le C \lambda_j^{1-s}\|\phi\|_{*},
     \vspace{3mm}\\
     \bigg|\left<(-\Delta)^s\phi-p_s{\bm W}_{\bm \lambda, \bm \xi}^{p_s-1}\phi, Z_{j,m}\right>\bigg|&\le C\lambda_j^{-\frac{N-2s}{2}+1}\|\phi\|_{*},   \ \  m=1,\cdots,N.
\end{aligned}
$$
Thus we infer 
$$
\left<(-\Delta)^s\phi+V(x)\phi-p_s{\bm W}_{\bm \lambda, \bm \xi}^{p_s-1}\phi-h, Z_{j,m}\right>=
\left\{
\begin{array}{lcl}
     C\lambda_j^{-1-s}\|\phi\|_*+C\lambda_j^{-1}\|h\|_{**},  \ m=0,
     \vspace{2mm}\\
     C\lambda_j^{1-s}\|\phi\|_*+C\lambda_j\|h\|_{**}, \  m=1,\cdots, N.
\end{array}
\right.
$$
On the other hand, we have 
$$
\int_{\mathbb{R}^N}W^{p_s-1}_{\lambda_i,\xi_i}Z_{i,l}Z_{j,m}=
\left\{
\begin{array}{ll}
     c_0\lambda_j^{-2}+O(\lambda_j^{-N-2}), &i=j, \ l=m=0,
     \vspace{2mm}\\
     O(\lambda_j^{-N-1}),  &i=j, \ l=0, m=1,\cdots, N,
     \vspace{2mm}\\
     c\lambda_j^{2}+O(\lambda_j^{2-N}),  &i=j, \ l=m=1, \ \cdots, N,
     \vspace{2mm}\\
     O(\lambda_j^{-N}),  &i=j, \ l\neq m, l, m=1,\cdots, N,
     \vspace{2mm}\\
     0, &i\neq j,
\end{array}
\right.
$$
where $c>0$ and $c_0>0$. Consequently, we conclude that
$$
c_i^0=o(\lambda_i), \ \  c_i^l=o(\lambda_i^{-1}), \ \ i=1,2\cdots, k, \ \ l=1,\cdots,N.
$$
\\
\textbf{Step 2.} There exist  $R>0$ and $a>0$, both independent of $n$, such that for some $j_0\in \{1,\cdots k\}$, 
\begin{equation}\label{eq:3-1}
    \Big\|\lambda^{-\frac{N-2}{2}}_{j_0}\phi(x)\Big\|_{L^\infty(B_{R/{\lambda_{j_0}}}(\xi_{j_0}))}\ge a>0.
\end{equation}
With the aid of the Green's representation formula, we have
\begin{equation}\label{eq:4}
    \begin{aligned}
         |\phi(x)|&\le \displaystyle C\int_{\mathbb{R}^N}\frac{1}{|x-y|^{N-2s}}{\bm W_{\bm \lambda, \bm \xi}^{p_s-1}}(y)|\phi(y)|+ C\int_{\mathbb{R}^N}\frac{1}{|x-y|^{N-2s}}|h(y)|
         \vspace{3mm}\\
         & \displaystyle +\sum_{i=1}^k\sum_{l=0}^N|c_i^l|\int_{\mathbb{R}^N}\frac{1}{|x-y|^{N-2s}}W_{\lambda_i, \xi_i}^{p_s-1}(y)|Z_{i,l}(y)|\vspace{3mm} \\
         &:=\displaystyle A_1+A_2+\sum_{i=1}^k\sum_{l=0}^N|c_i^l|A^l_{3,i}.
    \end{aligned}
\end{equation}
Using Lemma \ref{lemLL2.3}, we find
\begin{equation}
    \begin{aligned}
         A_1&\le \displaystyle C\sum_{i=1}^k \int_{B_{2\delta}(\xi_i)}\frac{1}{|x-y|^{N-2s}}{U_{\lambda_i, \xi_i}^{p_s-1}}(y)|\phi(y)|
         \vspace{3mm}\\
         &\le \displaystyle C\|\phi\|_*\sum_{i=1}^k \int_{B_{2\delta}(\xi_i)}\frac{1}{|x-y|^{N-2s}}\frac{\lambda_i^{\frac{N+2s}{2}}}{(1+\lambda_i|y-\xi_i|)^{\frac{N+6s+2\sigma}{2}}}
         \vspace{3mm}\\
         &\le \displaystyle C\|\phi\|_*\sum_{i=1}^k \int_{\mathbb{R}^N}\frac{1}{|\lambda_i(x-\xi_i)-y|^{N-2s}}\frac{\lambda_i^{\frac{N-2s}{2}}}{(1+|y|)^{\frac{N+6s+2\sigma}{2}}}
         \vspace{3mm}\\
         &\le \displaystyle
          C \sum_{i=1}^k \frac{\lambda_i^{\frac{N-2s}{2}}}{(1+\lambda_i|x-\xi_i|)^\frac{N}{2}}. 
    \end{aligned}
\end{equation}
In a similar way, we get
\begin{equation}
    \begin{aligned}
         A_2&\le \displaystyle C\|h\|_{**}\sum_{i=1}^k \int_{\mathbb{R}^N}\frac{1}{|x-y|^{N-2s}}\frac{\lambda_i^{\frac{N+2s}{2}}}{(1+\lambda_i|y-\xi_i|)^{\frac{N+2s}{2}+\sigma}}
         \vspace{3mm}\\
         &\le \displaystyle C\|h\|_{**}\sum_{i=1}^k \frac{\lambda_i^{\frac{N-2s}{2}}}{(1+\lambda_i|x-\xi_i|)^{\frac{N-2s}{2}+\sigma}}.
    \end{aligned}
\end{equation}
Moreover, for $i=1,\cdots, k$,
\begin{equation}
    \begin{aligned}
         A_{3,i}^0&\le \displaystyle C\lambda^{-1}_i \int_{\mathbb{R}^N}\frac{1}{|x-y|^{N-2s}}{U_{\lambda_i, \xi_i}^{p_s}}(y)
         \vspace{3mm}\\
         &\le \displaystyle C\lambda^{\frac{N-2-2s}{2}}_i \int_{\mathbb{R}^N}\frac{1}{|\lambda_i(x-\xi_i)-y|^{N-2s}}\frac{1}{(1+|y|)^{N+2s}}
         \vspace{3mm}\\
         &\le \displaystyle
         C\frac{\lambda^{\frac{N-2-2s}{2}}_i}{(1+\lambda_i|x-\xi_i|)^{N-2s}}
    \end{aligned}
\end{equation}
and 
\begin{equation}
    \begin{aligned}
         A_{3,i}^l&\le \displaystyle C\int_{\mathbb{R}^N}\frac{1}{|x-y|^{N-2s}}{W_{\lambda_i, \xi_i}^{p_s-1}}(y)\left(\eta_i(y)\frac{\partial U_{\lambda_i, \xi_i}}{\partial \xi_{i,l}}(y)+U_{\lambda_i, \xi_i}(y)\frac{\partial \eta_i}{\partial \xi_{i,l}}(y)\right)
         \vspace{3mm}\\
         &\le \displaystyle C\int_{\mathbb{R}^N}\frac{1}{|x-y|^{N-2s}}{U_{\lambda_i, \xi_i}^{p_s-1}(y)}\bigg|\frac{\partial U_{\lambda_i, \xi_i}}{\partial \xi_{i,l}}(y)\bigg|+C\int_{\mathbb{R}^N}\frac{1}{|x-y|^{N-2s}}{U_{\lambda_i, \xi_i}^{p_s}}(y)\bigg|\frac{\partial \eta_i}{\partial \xi_{i,l}}(y)\bigg|
         \vspace{3mm}\\
         &\le \displaystyle
         C\frac{\lambda^{\frac{N+2-2s}{2}}_i}{(1+\lambda_i|x-\xi_i|)^{N-2s}}.
    \end{aligned}
\end{equation}
By virtue of \eqref{eq:4}, we have
\begin{equation}\label{eq:5}
    \begin{aligned}
         &\displaystyle
         \left(\sum_{j=1}^k\frac{\lambda_j^{\frac{N-2s}{2}}}{(1+\lambda_j|x-\xi_j|)^{\frac{N-2s}{2}+\sigma}}\right)^{-1}|\phi(x)|
         \vspace{3mm}\\
         \le& \displaystyle C\left(\sum_{j=1}^k\frac{\lambda_j^{\frac{N-2s}{2}}}{(1+\lambda_j|x-\xi_j|)^{\frac{N-2s}{2}+\sigma}}\right)^{-1}\sum_{i=1}^k\frac{\lambda_i^{\frac{N-2s}{2}}}{(1+\lambda_i|x-\xi_i|)^{\frac{N}{2}} }+\|h\|_{**}+o(1).
    \end{aligned}
\end{equation}
Since $\|\phi\|_*=1$,  there exists some $x_0\in \mathbb{R}^N$ satisfying 
$$
\left(\sum_{j=1}^k\frac{\lambda_j^{\frac{N-2s}{2}}}{(1+\lambda_j|x_0-\xi_j|)^{\frac{N-2s}{2}+\sigma}}\right)^{-1}|\phi(x_0)|=1.
$$
It follows from \eqref{eq:5} and the definition of $\|\phi\|_*$ that
\begin{equation}\label{eq:6}
    1\le C\left(\sum_{j=1}^k\frac{\lambda_j^{\frac{N-2s}{2}}}{(1+\lambda_j|x_0-\xi_j|)^{\frac{N-2s}{2}+\sigma}}\right)^{-1}\sum_{i=1}^k\frac{\lambda_i^{\frac{N-2s}{2}}}{(1+\lambda_i|x_0-\xi_i|)^{ \frac{N}{2}}}+\|h\|_{**}+o(1).
\end{equation}
Given that $\|h\|_{**}\to 0$, we infer that there exists $R>0$, such that $|\lambda_{j_0}(x_0-\xi_{j_0})|\le R$ for some $j_0$. 
Hence we derive 
\begin{equation}\label{eq:7}
    \begin{array}{ll}
        \displaystyle 1
        &\le \sup\limits_{x\in B_{R/\lambda_{j_0}}(\xi_{j_0})}\left(\sum\limits_{j=1}^k\frac{\lambda_j^{\frac{N-2s}{2}}}{(1+\lambda_j|x-\xi_j|)^{\frac{N-2s}{2}+\sigma}}\right)^{-1} |\phi(x)|  
        \vspace{3mm}\\
        &\le  (1+R)^{\frac{N-2s}{2}+\sigma}\Big\|\lambda_{j_0}^{-\frac{N-2s}{2}}\phi(x)\Big\|_{L^\infty\big(B_{R/\lambda_{j_0}}(\xi_{j_0})\big)}.
        \end{array}
\end{equation}
Consequently, \eqref{eq:3-1} holds.\\
\textbf{Step 3.} We thereby reach a contradiction.

Let $\hat{\phi}(x)=\lambda_{j_0}^{-\frac{N-2s}{2}}\phi(\xi_{j_0}+\lambda_{j_0}^{-1}x)$. Then $\hat{\phi}$ satisfies
\begin{equation}\label{eq:8}
    \begin{aligned}
         &(-\Delta)^{s} \hat{\phi}(x)+\lambda_{j_0}^{-2s} V(\xi_{j_0}+\lambda_{j_0}^{-1}x) \hat{\phi}(x)-p_s\lambda_{j_0}^{-2s}{\bm W_{\bm \lambda, \bm \xi}^{p_s-1}}(\xi_{j_0}+\lambda_{j_0}^{-1}x)\hat{\phi}(x)
         \vspace{3mm}\\
         =&\lambda_{j_0}^{-\frac{N+2s}{2}}h(\xi_{j_0}+\lambda_{j_0}^{-1}x)+\sum\limits_{i=1}^k\sum\limits_{l=0}^N c_i^l\lambda_{j_0}^{-\frac{N+2s}{2}}W^{p_s-1}_{\lambda_i,\xi_i}(\xi_{j_0}+\lambda_{j_0}^{-1}x)Z_{i,l}(\xi_{j_0}+\lambda_{j_0}^{-1}x).
    \end{aligned}
\end{equation}
Since  $\|\phi\|_*=1$ and $\|h\|_{**}=o(1)$, we see that 
\[
|\hat{\phi}(x)|=\Big|\lambda_{j_0}^{-\frac{N-2s}{2}}\phi(\xi_{j_0}+\lambda_{j_0}^{-1}x)\Big|\leq C
\]
and
\[
\Big|\lambda_{j_0}^{-\frac{N+2s}{2}}h(\xi_{j_0}+\lambda_{j_0}^{-1}x)\Big|\leq C\|h\|_{**}.
\]
It follows from  $c_i^0=o(\lambda_i)$, $c_i^l=o(\lambda_i^{-1}), l=1,\cdots,N$ that  the right hand side of \eqref{eq:8} tends to $0$ uniformly  on the compact sets of $\mathbb{R}^N$. By the elliptic regularity theory,  there exists a subsequence such that $\hat{\phi}(x)$ converges uniformly on any compact set to a solution $\phi_*$ of
$$
(-\Delta)^su-p_sU_{1,0}^{p_s-1}u=0\ \ \ \hbox{in}\ \mathbb{R}^N.
$$
By virtue of  \eqref{eq:3-1}, we infer that
\begin{equation}\label{a3.17}
\|\phi_{*}\|_{L^\infty(B_R(0))}\ge a>0.
\end{equation}
Thus the nondegeneracy of $U_{1,0}$ implies that
$$
\phi_*=\sum_{l=0}^N\alpha_l Z_l,
$$
where $Z_0=\frac{\partial U_{\lambda,0}}{\partial\lambda}\Big|_{\lambda=1},  Z_l=\frac{\partial U_{1,0}}{\partial x_l}$, $l=1,\cdots,N$.
Using the fact that 
$$
\int_{\mathbb{R}^N} W^{p_s-1}_{\lambda_{j_0}, \xi_{j_0}}Z_{j_0,l}\phi=0, \ \ l=0,1,\cdots,N,
$$
we deduce that
$$
\int_{\mathbb{R}^N} U_{1,0}^{p_s-1}Z_l\phi_*=0, \ \  l=0,1,\cdots,N.
$$
Hence, $\phi_*=0$, which contradicts \eqref{a3.17}. This completes the proof of Lemma \ref{l1}.
\end{proof}

Using the same argument as in the proof of Proposition 4.1 in \cite{PFM-2003}, we can obtain the following proposition.
\begin{proposition}\label{prop2.3}
    There exist $\lambda_0>0$ and $C>0$ independent of $\bm{\lambda}$, $\bm{\xi}$ such that for $\lambda_i\geq \lambda_0$, $i=1,2,\cdots,k$ and $h\in C_{**}(\R^N)$, problem \eqref{eq:2} admits a unique solution $\phi:=L(h)$. Moreover,
    $$
    \|L(h)\|_*\leq C\|h\|_{**},\ \  |c_i^0|\leq C\lambda_i \|h\|_{**},\ \  |c_i^l|\leq C\lambda_i^{-1} \|h\|_{**},\ \ i=1,2,\cdots,k,\ \  l=1,2,\cdots,N.
    $$
\end{proposition}

In what follows, we will choose $\bm{\lambda}$ appropriately such that  it goes to infinity equivalently. 
Specifically, we take $\bm{\lambda}$ satisfying  $\bm{\lambda}\sim\epsilon^{-\frac{1}{2s}}$.
 We then proceed to deal with   the nonlinear projected problem:
\begin{equation}\label{eq-2.22}
    \left\{
    \begin{array}{lcl}
         (-\Delta)^{s}({\bm W}_{\bm \lambda, \bm \xi}+\phi)+V(x)({\bm W}_{\bm \lambda, \bm \xi}+\phi)=({\bm W}_{\bm \lambda, \bm \xi}+\phi)_+^{p_s-\epsilon}+\sum\limits_{i=1}^k\sum\limits_{l=0}^N c_i^l W_{\lambda_i, \xi_i}^{p_s-1}Z_{i,l} \ \ \hbox{in }\ \mathbb{R}^N,\\
         \displaystyle \int_{\mathbb{R}^N} W_{\lambda_i, \xi_i}^{p_s-1}Z_{i,l}\phi=0, \ \ i=1,2,\cdots,k,\; l=0,1,\cdots, N,
    \end{array}
    \right.
\end{equation}
where $u_+=\max\{u,0\}$.

The main result of this section is the following result on the unique solvability of \eqref{eq-2.22}, which reduces the solvability of \eqref{eq:1} to a finite-dimensional problem.
\begin{proposition}\label{Prop2.4}
    Let $0<s<1$ and $N> 4s$. Then there exist $\epsilon_0>0$ and $C>0$ such that for $\epsilon\in (0,\epsilon_0)$, problem \eqref{eq-2.22} has a unique solution $\phi_{\epsilon}\in C_*(\R^N)\cap H^s_V(\R^N)$  satisfying 
    \begin{equation}
        \|\phi_{\epsilon}\|_*\leq   C\epsilon^{\frac{1}{2}(1+\tau_0)}  ,
    \end{equation}
    where $\tau_0$ is a small positive constant.
\end{proposition}
To study \eqref{eq-2.22}, we  rewrite it  as
\begin{equation}\label{eq-2.24}
    \left\{
    \begin{array}{lcl}
         (-\Delta)^{s}\phi+V(x)\phi-p_s{\bm W}_{\bm \lambda, \bm \xi}^{p_s-1}\phi=\mathcal{N}(\phi)+\mathcal{R}+\sum\limits_{i=1}^k\sum\limits_{l=0}^N c_i^l W_{\lambda_i, \xi_i}^{p_s-1}Z_{i,l}\ \ \hbox{in}\ \mathbb{R}^N,\\
         \phi\in C_*(\R^N)\cap H^s_V(\R^N),\\
         \displaystyle \int_{\mathbb{R}^N} W_{\lambda_i, \xi_i}^{p_s-1}Z_{i,l}\phi=0, \;i=1,2,\cdots,k,\; l=0,1,\cdots, N,
    \end{array}
    \right.
\end{equation}
where $\mathcal{N}(\phi)=\mathcal{N}_1(\phi)+\mathcal{N}_2(\phi)$, $\mathcal{R}=\mathcal{R}_1+\mathcal{R}_2+\mathcal{R}_3$,
\begin{equation}\label{eq-2.24R}
\begin{aligned}
   & \mathcal{N}_1(\phi)=({\bm W}_{\bm \lambda, \bm \xi}+\phi)_+^{p_s-\epsilon}-{\bm W}_{\bm \lambda, \bm \xi}^{p_s-\epsilon}-(p_s-\epsilon){\bm W}_{\bm \lambda, \bm \xi}^{p_s-1-\epsilon}\phi,\\
   & \mathcal{N}_2(\phi)=(p_s-\epsilon){\bm W}_{\bm \lambda, \bm \xi}^{p_s-1-\epsilon}\phi-p_s{\bm W}_{\bm \lambda, \bm \xi}^{p_s-1}\phi,\\
   &\mathcal{R}_1= {\bm W}_{\bm \lambda, \bm \xi}^{p_s-\epsilon}-\sum_{i=1}^k\eta_i U_{\lambda_i,\xi_i}^{p_s},\quad \mathcal{R}_2=-\sum_{i=1}^k \Upsilon_i,\quad
   \mathcal{R}_3=-V(x){\bm W}_{\bm \lambda, \bm \xi}.\\
\end{aligned}
\end{equation}

Next, we apply the contraction mapping theorem to obtain the unique solvability of  problem \eqref{eq-2.24} 
when  $\|\phi\|_*$ is  sufficiently small.
\begin{lemma}\label{lem2.5}
    If $0<s<1$ and $N>4s$, then 
    $$
    \|\mathcal{N}(\phi)\|_{**}\leq C\left(\|\phi\|_*^2+\|\phi\|_*^{p_s-\epsilon}+\epsilon|\log\epsilon|\|\phi\|_*\right).
    $$
\end{lemma}
\begin{proof}
 For $p_s-\epsilon\le 2$, by H\"older inequality, we have  
 $$
 \begin{aligned}    
 |\mathcal{N}_1(\phi)|
 \leq& C|\phi|^{p_s-\epsilon}\leq C\|\phi\|_*^{p_s-\epsilon}\left(\sum_{j=1}^k\frac{\lambda_j^{\frac{N-2s}{2}}}{(1+\lambda_j|x-\xi_j|)^{\frac{N-2s}{2}+\sigma}}\right)^{p_s-\epsilon}\\
 \leq & C\|\phi\|_*^{p_s-\epsilon}\sum_{j=1}^k\frac{\lambda_j^{\frac{N+2s}{2}-\frac{N-2s}{2}\epsilon}}{(1+\lambda_j|x-\xi_j|)^{\frac{N+2s}{2}+(p_s-\epsilon)\sigma-\frac{N-2s}{2}\epsilon}}\\
 \leq & C\|\phi\|_*^{p_s-\epsilon}\sum_{j=1}^k\frac{\lambda_j^{\frac{N+2s}{2}}}{(1+\lambda_j|x-\xi_j|)^{\frac{N+2s}{2}+\sigma}},\\
 \end{aligned}
 $$
 which gives that 
 $$\|\mathcal{N}_1(\phi)\|_{**}\leq C\|\phi\|_*^{p_s-\epsilon}.$$
If  $p_s-\epsilon > 2$, we find 
 $$
 \begin{aligned}
     |\mathcal{N}_1(\phi)|
 \leq& C\left({\bm W}_{\bm \lambda, \bm \xi}^{p_s-2-\epsilon}\phi^2+|\phi|^{p_s-\epsilon}\right)
 \leq C(\|\phi\|_*^2+\|\phi\|_*^{p_s-\epsilon})\sum_{j=1}^k\frac{\lambda_j^{\frac{N+2s}{2}}}{(1+\lambda_j|x-\xi_j|)^{\frac{N+2s}{2}+\sigma}}.
 \end{aligned}
 $$
 Thus we obtain
  $$\|\mathcal{N}_1(\phi)\|_{**}\leq C(\|\phi\|_*^2+\|\phi\|_*^{p_s-\epsilon}).$$
 On the other hand, it follows from  the mean value theorem that 
 $$
 \begin{aligned}
     |\mathcal{N}_2(\phi)|=&\left|(p_s-\epsilon){\bm W}_{\bm \lambda, \bm \xi}^{p_s-1-\epsilon}\phi-p_s{\bm W}_{\bm \lambda, \bm \xi}^{p_s-1}\phi\right|\\
     \leq &C\epsilon \sum_{i=1}^k (\eta_iU_{\lambda_i,\xi_i})^{p_s-1-\theta\epsilon}\left(1+|\log(\eta_iU_{\lambda_i,\xi_i})|\right)|\phi|\\
     \leq &C\epsilon|\log\epsilon|\|\phi\|_*\sum_{j=1}^k\frac{\lambda_j^{\frac{N+2s}{2}}}{(1+\lambda_j|x-\xi_j|)^{\frac{N+2s}{2}+\sigma}},
 \end{aligned}
 $$
 where $\theta\in (0,1)$.
As a result, we obtain 
 $$
 \|\mathcal{N}_2(\phi)\|_{**}\leq C\epsilon|\log \epsilon|\|\phi\|_*.
 $$
The claimed estimate then follows from the above calculations.
\end{proof}
\begin{lemma}\label{lem2.6}
    If $0<s<1$ and $N> 4s$, then
    $$
    \|\mathcal{R}\|_{**}\leq  C\epsilon^{\frac{1}{2}(1+\tau_0)},
    $$
     where $\tau_0$ is a small positive constant.
\end{lemma}
\begin{proof}
    Applying the mean value theorem, we have
    $$
    \begin{aligned}
        |\mathcal{R}_1|
        \leq &\sum_{i=1}^k \chi_{B_{\delta}(\xi_i)}|U_{\lambda_i,\xi_i}^{p_s-\epsilon}-U_{\lambda_i,\xi_i}^{p_s}|+\sum_{i=1}^k \chi_{B_{2\delta}(\xi_i)\setminus B_{\delta}(\xi_i)}(U_{\lambda_i,\xi_i}^{p_s-\epsilon}+U_{\lambda_i,\xi_i}^{p_s})\\
        \leq &C\epsilon \sum_{i=1}^k \chi_{B_{\delta}(\xi_i)}U_{\lambda_i,\xi_i}^{p_s-\theta\epsilon}\log U_{\lambda_i,\xi_i}+C\sum_{i=1}^k \chi_{B_{2\delta}(\xi_i)\setminus B_{\delta}(\xi_i)}\frac{\lambda_i^{\frac{N+2s}{2}}}{(1+\lambda_i|x-\xi_i|)^{N+2s}}\\
        \leq &C\epsilon|\log\epsilon|\sum_{i=1}^k \frac{\lambda_i^{\frac{N+2s}{2}}}{(1+\lambda_i|x-\xi_i|)^{\frac{N+2s}{2}+\sigma}}+C\sum_{i=1}^k \lambda_i^{^{-\frac{N+2s}{2}+\sigma}}\frac{\lambda_i^{\frac{N+2s}{2}}}{(1+\lambda_i|x-\xi_i|)^{\frac{N+2s}{2}+\sigma}},
    \end{aligned}
    $$
  where $\theta\in (0,1)$. Consequently, we arrive at 
    $$
    \|\mathcal{R}_1\|_{**}\leq C(\epsilon|\log\epsilon|+\epsilon^\frac{N+2s-2\sigma}{4s}).    $$
    From Lemma \ref{lem-J}, it is easy to check that
    $$
    \begin{aligned}
    |\mathcal{R}_2|\leq& C\sum_{i=1}^k \frac{\lambda_i^{-\frac{N-2s}{2}}}{(1+|x-\xi_i|)^{N+2s}} 
    \leq  C\sum_{i=1}^k\lambda_i^{-\frac{N-2s}{2}+\sigma}\frac{\lambda_i^{\frac{N+2s}{2}}}{(1+\lambda_i|x-\xi_i|)^{\frac{N+2s}{2}+\sigma}},
    \end{aligned}
    $$
    which yields that 
    $$
    \|\mathcal{R}_2\|_{**}\leq C\epsilon^\frac{N-2s-2\sigma}{4s}.
    $$
   Note that 
    $$
    \begin{aligned}
        |\mathcal{R}_3|\leq & C\sum_{i=1}^k\frac{\lambda_i^{\frac{N-2s}{2}}}{(1+\lambda_i|x-\xi_i|)^{N-2s}}\chi_{B_{2\delta}(\xi_i)}(x)\\
        \leq &\begin{cases}
           \displaystyle C\sum_{i=1}^k\lambda_i^{-2s+\sigma}\frac{\lambda_i^{\frac{N+2s}{2}}}{(1+\lambda_i|x-\xi_i|)^{\frac{N+2s}{2}+\sigma}}\chi_{B_{2\delta}(\xi_i)}(x), &\text{if}~~N>6s,\vspace{2mm}\\
           \displaystyle C\sum_{i=1}^k\lambda_i^{-\frac{N-2s}{2}+\sigma}\frac{\lambda_i^{\frac{N+2s}{2}}}{(1+\lambda_i|x-\xi_i|)^{\frac{N+2s}{2}+\sigma}}\chi_{B_{2\delta}(\xi_i)}(x), &\text{if}~~4s<N\leq 6s,\vspace{2mm}\\
        \end{cases}\\
        \leq &\begin{cases}
           \displaystyle C\epsilon^{1-\frac{\sigma}{2s}}\sum_{i=1}^k\frac{\lambda_i^{\frac{N+2s}{2}}}{(1+\lambda_i|x-\xi_i|)^{\frac{N+2s}{2}+\sigma}}, &\text{if}~~N>6s,\vspace{2mm}\\
          C\epsilon^\frac{N-2s-2\sigma}{4s} \displaystyle \sum_{i=1}^k\frac{\lambda_i^{\frac{N+2s}{2}}}{(1+\lambda_i|x-\xi_i|)^{\frac{N+2s}{2}+\sigma}}, &\text{if}~~4s<N\leq 6s.\vspace{2mm}\\
        \end{cases}
    \end{aligned}
    $$
     Hence, we conclude 
    $$
    \|\mathcal{R}_3\|_{**}\leq C\epsilon^{\frac{1}{2}(1+\tau_0)},
    $$
    where $\tau_0>0$ is a  small  constant.
\end{proof}
We now turn to the proof of Proposition \ref{Prop2.4}.

\begin{proof}[Proof of Proposition \ref{Prop2.4}]
By Proposition \ref{prop2.3}, we see that $\phi$ solves \eqref{eq-2.24} if and only if $\phi$ is a fixed point of  the operator 
$$
T:C_*(\R^N)\to C_*(\R^N),
$$
defined by 
$$
T(\phi):=L(\mathcal{N}(\phi))+L(\mathcal{R}), 
$$    
where $L$ is as in Proposition \ref{prop2.3}.

Let $B_*:=\{\phi\in C_*(\R^N):\|\phi\|_*\leq A\epsilon^{\frac{1}{2}(1+\tau_0)}\},$
where  $A$  is a positive constant to be determined later.  We will prove that $T$ is a contraction map from $B_*$ to $B_*$. 

By Proposition \ref{prop2.3}, Lemmas \ref{lem2.5} and  \ref{lem2.6}, we have 
$$
\|T(\phi)\|_*\leq C(\|\mathcal{N}(\phi)\|_{**}+\|\mathcal{R}\|_{**})\leq  C\left(\|\phi\|_*^2+\|\phi\|_*^{p_s-\epsilon}+\epsilon|\log\epsilon|\|\phi\|_*+\epsilon^{\frac{1}{2}(1+\tau_0)}\right).
$$
Taking $A>2C$, it follows that for sufficiently small $\epsilon$
$$
\|T(\phi)\|_*\leq A\epsilon^{\frac{1}{2}(1+\tau_0)}.
$$
This shows that $T$ maps $B_*$ to  $B_*$.

Next,  for any $\phi_1,\;\phi_2\in B_*$, we have 
$$
\begin{aligned}
    T(\phi_1)-T(\phi_2)=&L\left(({\bm W}_{\bm \lambda, \bm \xi}+\phi_1)_+^{p_s-\epsilon}-({\bm W}_{\bm \lambda, \bm \xi}+\phi_2)_+^{p_s-\epsilon}-(p_s-\epsilon){\bm W}_{\bm \lambda, \bm \xi}^{p_s-1-\epsilon}(\phi_1-\phi_2)\right)\\
    &+L\left(\left((p_s-\epsilon){\bm W}_{\bm \lambda, \bm \xi}^{p_s-1-\epsilon}-p_s{\bm W}_{\bm \lambda, \bm \xi}^{p_s-1}\right)(\phi_1-\phi_2)\right)\\
    :=&L(\mathcal{N}_1(\phi_1,\phi_2))+L(\mathcal{N}_2(\phi_1,\phi_2)),
\end{aligned}
$$
which, together with Proposition \ref{prop2.3}, gives that
$$
\begin{aligned}
    \|T(\phi_1)-T(\phi_2)\|_*\leq C(\|\mathcal{N}_1(\phi_1,\phi_2)\|_{**}+\|\mathcal{N}_2(\phi_1,\phi_2)\|_{**}).
\end{aligned}
$$
If $p_s-\epsilon\le 2$, then we deduce 
$$
\begin{aligned}
    |\mathcal{N}_1(\phi_1,\phi_2)|\leq & C(|\phi_1|^{p_s-1-\epsilon}+|\phi_2|^{p_s-1-\epsilon})|\phi_1-\phi_2|\\
    \leq& C(\|\phi_1\|_*^{p_s-1-\epsilon}+\|\phi_2\|_*^{p_s-1-\epsilon})\|\phi_1-\phi_2\|_*\left(\sum_{i=1}^k\frac{\lambda_i^{\frac{N-2s}{2}}}{(1+\lambda_i|x-\xi_i|)^{\frac{N-2s}{2}+\sigma}}\right)^{p_s-\epsilon}\\
    \leq& C(\|\phi_1\|_*^{p_s-1-\epsilon}+\|\phi_2\|_*^{p_s-1-\epsilon})\|\phi_1-\phi_2\|_*\sum_{i=1}^k\frac{\lambda_i^{\frac{N+2s}{2}}}{(1+\lambda_i|x-\xi_i|)^{\frac{N+2s}{2}+\sigma}}.
\end{aligned}
$$
Hence we obtain  
$$
\|\mathcal{N}_1(\phi_1,\phi_2)\|_{**}\leq C(\|\phi_1\|_*^{p_s-1-\epsilon}+\|\phi_2\|_*^{p_s-1-\epsilon})\|\phi_1-\phi_2\|_*.
$$
For  $p_s-\epsilon> 2$, similar arguments lead to
$$
\begin{aligned}
|\mathcal{N}_1(\phi_1,\phi_2)|
\leq&  C{\bm W}_{\bm \lambda, \bm \xi}^{p_s-2-\epsilon}(|\phi_1|+|\phi_2|)|\phi_1-\phi_2|+ C(|\phi_1|^{p_s-1-\epsilon}+|\phi_2|^{p_s-1-\epsilon})|\phi_1-\phi_2|\\
\leq& C(\|\phi_1\|_*+\|\phi_2\|_*)\|\phi_1-\phi_2\|_*\sum_{i=1}^k\frac{\lambda_i^{\frac{N+2s}{2}}}{(1+\lambda_i|x-\xi_i|)^{\frac{N+2s}{2}+\sigma}},
\end{aligned}
$$
which implies 
$$
\|\mathcal{N}_1(\phi_1,\phi_2)\|_{**}\leq C(\|\phi_1\|_*+\|\phi_2\|_*)\|\phi_1-\phi_2\|_*.
$$
On the other hand, we have 
$$
\begin{aligned}
|\mathcal{N}_2(\phi_1,\phi_2)|\leq& \left|(p_s-\epsilon){\bm W}_{\bm \lambda, \bm \xi}^{p_s-1-\epsilon}-p_s{\bm W}_{\bm \lambda, \bm \xi}^{p_s-1}\right||\phi_1-\phi_2|\\
 \leq &C\epsilon|\log\epsilon|\|\phi_1-\phi_2\|_*\sum_{i=1}^k\frac{\lambda_i^{\frac{N+2s}{2}}}{(1+\lambda_i|x-\xi_i|)^{\frac{N+2s}{2}+\sigma}}.
\end{aligned}
$$
Then we find 
$$
\|\mathcal{N}_2(\phi_1,\phi_2)\|_{**}\leq C\epsilon|\log\epsilon|\|\phi_1-\phi_2\|_*.
$$
As a result, for $\epsilon$ sufficiently small, we derive 
\[
\begin{split}
\|T(\phi_1)-T(\phi_2)\|_*\leq & C(\|\phi_1\|_*^{\min\{p_s-1-\epsilon,1\}}+\|\phi_2\|_*^{\min\{p_s-1-\epsilon,1\}}+\epsilon|\log\epsilon|)\|\phi_1-\phi_2\|_*\\
\leq& \frac{1}{2}\|\phi_1-\phi_2\|_*.
\end{split}
\]
Thus $T$ is a contraction mapping from $ B_*$ into itself.
		By  the contraction mapping theorem,  there exists a unique solution $\phi_{\epsilon}\in B_*$ to problem \eqref{eq-2.24}.  This completes the proof of Proposition \ref{Prop2.4}.
        \end{proof}
        
        \section{Existence of multi-peak solutions}
      In the preceding section, we have proved that for sufficiently small $\epsilon$, there exists a unique $\phi_\epsilon$ such that 
       $u={\bm W}_{\bm \lambda, \bm \xi}+\phi_\epsilon$ satisfies \eqref{eq-2.22}. 
In order to guarantee that ${\bm W}_{\bm \lambda, \bm \xi}+\phi_\epsilon$ 
 solves \eqref{eq:1}, we just need to  choose $\bm \lambda, \bm \xi$
 suitably such that the corresponding Lagrange multipliers $c_i^l=0$, $i=1,\cdots,k,\ l=0,1,\cdots,N$
 for $\epsilon$  small enough. For this purpose,  we define the $s$-harmonic extension of $u$ by  $\widetilde{u}=\widetilde{\bm W}_{\bm \lambda, \bm \xi}+\widetilde{\phi}_\epsilon$. 
 Then we have 
        \begin{equation}\label{eq-3.1}
        \begin{cases}
            \text{div}(t^{1-2s}\nabla\widetilde{u})=0 \quad &\text{in} \quad\R^{N+1}_+,\\
            -\lim\limits_{t\to 0^+}t^{1-2s}\partial_t\widetilde{u}=-V(x){u}+({u})_+^{p_s-\epsilon}+\sum\limits_{i=1}^k\sum\limits_{l=0}^N c_i^l W_{\lambda_i, \xi_i}^{p_s-1}Z_{i,l}\quad &\text{on}  \quad\R^{N}.
        \end{cases}
         \end{equation}
Denote
$$
X=(x,t)\in \R^{N+1},\; \nabla\widetilde{u}=\Big(\frac{\partial \widetilde{u}}{\partial x_1}, \cdots, \frac{\partial \widetilde{u}}{\partial x_N},\frac{\partial \widetilde{u}}{\partial t}\Big), \; 
\nabla u=\Big(\frac{\partial u}{\partial x_1}, \cdots,\frac{\partial u}{\partial x_N}\Big).
$$
     We now employ local Pohozaev identities to  solve the above reduced problem.
      Multiplying \eqref{eq-3.1} by $\frac{\partial\widetilde{u}}{\partial x_m}$, $m=1,2,\cdots,N$ 
         and integrating by parts, we obtain
         \begin{equation}\label{eq3.2}
         \begin{aligned}
             &-\int_{\partial''\mathfrak{B}_\rho^+(\xi_j)}t^{1-2s}\frac{\partial\widetilde{u}}{\partial\nu}\frac{\partial\widetilde{u}}{\partial x_m}+\frac{1}{2}\int_{\partial''\mathfrak{B}_\rho^+(\xi_j)}t^{1-2s}|\nabla\widetilde{u}|^2\nu_m\\
             =&\int_{B_\rho(\xi_j)}\left(-V(x){u}+({u})_+^{p_s-\epsilon}+\sum\limits_{i=1}^k\sum\limits_{l=0}^N c_i^l W_{\lambda_i, \xi_i}^{p_s-1}Z_{i,l}\right)\frac{\partial u}{\partial x_m},
              \end{aligned}
         \end{equation}
   where 
        $$
        \displaystyle \partial'' \mathfrak{B}_\rho^+(\xi_j)=\{X=(x,t):|X-(\xi_j,0)|= \rho \text{ and } t>0\}\subseteq\R^{N+1}_+,
        $$
        $j=1,2,\cdots k$, and  $\nu$ is the outward unit normal vector on $ \partial'' \mathfrak{B}_\rho^+(\xi_j)$. 

         \begin{lemma}
             Suppose that $\bm{\lambda}$, $\bm{\xi}$ satisfy 
              \begin{equation}\label{eq3.4}
             -\int_{\partial''\mathfrak{B}_\rho^+(\xi_j)}t^{1-2s}\frac{\partial\widetilde{u}}{\partial\nu}\frac{\partial\widetilde{u}}{\partial x_m}+\frac{1}{2}\int_{\partial''\mathfrak{B}_\rho^+(\xi_j)}t^{1-2s}|\nabla\widetilde{u}|^2\nu_m
             =\int_{B_\rho(\xi_j)}\left(-V(x){u}+({u})_+^{p_s-\epsilon}\right)\frac{\partial u}{\partial x_m}
         \end{equation}
         and
         \begin{equation}\label{eq3.5}
             \int_{\R^N}\big((-\Delta)^s u+V(x)u-(u)_+^{p_s-\epsilon}\big)\frac{\partial W_{\lambda_j,\xi_j}}{\partial\lambda_j}=0,
         \end{equation}
         where $\rho\in (2\delta,\,5\delta)$, $j=1,2,\cdots,k$, $m=1,2,\cdots,N$. Then 
         $$
         c^l_i=0,\;i=1,2,\cdots,k,\;l=0,1,\cdots,N.
         $$
         \end{lemma}
         \begin{proof}
             It follows from \eqref{eq3.2} and \eqref{eq3.4} that 
             \begin{equation}\label{eq3.33}
                 \sum\limits_{i=1}^k\sum\limits_{l=0}^N c_i^l\int_{B_\rho(\xi_j)} W_{\lambda_i, \xi_i}^{p_s-1}Z_{i,l}\frac{\partial u}{\partial x_m}=0,\;j=1,2,\cdots,k,\;m=1,2,\cdots,N.
             \end{equation}
             By virtue of \eqref{eq-2.22}, \eqref{eq3.5} and \eqref{eq3.33},  we conclude that
             \begin{equation}
                 \sum\limits_{i=1}^k\sum\limits_{l=0}^N c_i^l\int_{B_\rho(\xi_j)} W_{\lambda_i, \xi_i}^{p_s-1}Z_{i,l}v=0,
             \end{equation}
             where $v$ denotes either $\frac{\partial u}{\partial x_m}$ or $\frac{\partial W_{\lambda_j,\xi_j}}{\partial\lambda_j}$ for $j=1,2,\cdots,k$ and $m=1,2,\cdots,N.$
             Direct computation gives
             \begin{equation}\label{eq3.35}
                \sum\limits_{l=0}^N c_j^l\int_{B_\rho(\xi_j)} W_{\lambda_j, \xi_j}^{p_s-1}Z_{j,l}\frac{\partial W_{\lambda_j, \xi_j}}{\partial x_m}=-\sum\limits_{l=0}^N c_j^l\int_{B_\rho(\xi_j)} W_{\lambda_j, \xi_j}^{p_s-1}Z_{j,l}\frac{\partial \phi_\epsilon}{\partial x_m},
             \end{equation}
              \begin{equation}\label{3.36}
                \sum\limits_{l=0}^N c_j^l\int_{B_\rho(\xi_j)} W_{\lambda_j, \xi_j}^{p_s-1}Z_{j,l}\frac{\partial W_{\lambda_j, \xi_j}}{\partial \lambda_j}=0.
             \end{equation}
               By integrating by parts, we infer
             \begin{equation}\label{eq3.37}
             \int_{B_\rho(\xi_j)} W_{\lambda_j, \xi_j}^{p_s-1}Z_{j,l}\frac{\partial \phi_\epsilon}{\partial x^m}=-\int_{B_\rho(\xi_j)} \frac{\partial (W_{\lambda_j, \xi_j}^{p_s-1}Z_{j,l}) }{\partial x^m}\phi_\epsilon= \begin{cases}
                     O(\|\phi_\epsilon\|_*),\quad&l=0,\\
                     O(\lambda_j^2\|\phi_\epsilon\|_*),\quad&l=1,2,\cdots,N.\\
                 \end{cases}
             \end{equation}
            Moreover, we have 
             \begin{equation}\label{eq3.38}
                 \int_{B_\rho(\xi_j)} W_{\lambda_j, \xi_j}^{p_s-1}Z_{j,l}\frac{\partial W_{\lambda_j, \xi_j}}{\partial x_m}=\begin{cases}
                 O(\lambda_j^{-N-1}), &l=0,\\
                     -c\lambda_j^2\delta_{lm}+O(\lambda_j^{2-N}), &l=1,2,\cdots,N,\\
                 \end{cases}
             \end{equation}
             and 
             \begin{equation}\label{eq3.39}
                  \int_{B_\rho(\xi_j)} W_{\lambda_j, \xi_j}^{p_s-1}Z_{j,l}\frac{\partial W_{\lambda_j, \xi_j}}{\partial \lambda_j}=
                  \begin{cases}
                 c_0\lambda_j^{-2}+O(\lambda_j^{-N-2}), &l=0,\\
                    O(\lambda_j^{-N-1}), &l=1,2,\cdots,N,\\
                 \end{cases}
             \end{equation}
             where $c_0>0$ and $c>0$.
Collecting \eqref{eq3.35}-\eqref{eq3.39}, we obtain 
$$
c\lambda_j^2|c_j^m|= O\left(\|\phi_\epsilon\|_*+\lambda_j^{-N-1}\right)|c_j^0|+O\left(\lambda_j^2\|\phi_\epsilon\|_*+\lambda_j^{2-N}\right)\sum_{l=1}^N|c_j^l|,
$$
and 
$$
(c_0\lambda_j^{-2}+O(\lambda_j^{-N-2}))|c_j^0|=O\left(\lambda_j^{-N-1}\sum_{l=1}^N|c_j^l|\right).
$$
Hence, $c_i^l=0$ for  $i=1,2,\cdots,k,\;l=0,1,\cdots,N.$
         \end{proof}
         To solve \eqref{eq3.4} and \eqref{eq3.5}, we start by  deriving the  following energy expansion  of \eqref{eq3.5}.
         \begin{lemma}\label{lem3.2}
             It holds that
             \begin{equation}\label{eq2810-1}
             \begin{aligned}
              &\int_{\R^N} \big((-\Delta)^s u+V(x)u-(u)_+^{p_s-\epsilon}\big)\frac{\partial W_{\lambda_i,\xi_i}}{\partial\lambda_i}\vspace{3mm}\\
              =&A\epsilon\lambda_i^{-1}+BV(\xi_i)\lambda_i^{-2s-1}+O\left(\epsilon\lambda_i^{-1-N}\log \lambda_i\right)+O\left(\epsilon^2\lambda_i^{-1}\log ^2\lambda_i\right)+O(\lambda_i^{-2s-1-\tau})\vspace{3mm}\\
              &+C\epsilon\lambda_i^{-1}\log\lambda_i\|\phi_\epsilon\|_*
              +C\lambda_i^{-\frac{N}{2}-1+s}\|\phi_\epsilon\|_*+C\lambda_i^{-1-2s}\|\phi_\epsilon\|_*+C\lambda_i^{-1}\|\phi_\epsilon\|_*^2,
               \end{aligned}
     \end{equation}
    where $A>0$, $B<0$ and $\tau>0$ are defined in Lemma\,\ref{lemA.1}.         
         \end{lemma}
         \begin{proof}
             Firstly, we have 
             $$
             \begin{aligned}
              &\int_{\R^N}((-\Delta)^s u+V(x)u-(u)_+^{p_s-\epsilon})\frac{\partial W_{\lambda_i,\xi_i}}{\partial\lambda_i}\\
              =&\Big\langle I'(W_{\lambda_i,\xi_i}),\,\frac{\partial W_{\lambda_i, \xi_i}}{\partial\lambda_i}\Big\rangle+\int_{\R^N}\left((-\Delta)^s\phi_\epsilon\frac{\partial W_{\lambda_i,\xi_i}}{\partial\lambda_i}-(p_s-\epsilon)W_{\lambda_i,\xi_i}^{p_s-1-\epsilon}\frac{\partial W_{\lambda_i,\xi_i}}{\partial\lambda_i}\phi_\epsilon\right)\\
              &+\int_{\R^N}V(x)\frac{\partial W_{\lambda_i,\xi_i}}{\partial\lambda_i}\phi_\epsilon-\int_{\R^N}\left((W_{\lambda_i, \xi_i}+\phi_\epsilon)_+^{p_s-\epsilon}-W_{\lambda_i, \xi_i}^{p_s-\epsilon}-(p_s-\epsilon)W_{\lambda_i, \xi_i}^{p_s-1-\epsilon}\phi_\epsilon\right)\frac{\partial W_{\lambda_i,\xi_i}}{\partial\lambda_i}\\
             : =&\Big\langle I'(W_{\lambda_i,\xi_i}),\,\frac{\partial W_{\lambda_i, \xi_i}}{\partial\lambda_i}\Big\rangle+I_1+I_2+I_3.
             \end{aligned}
             $$
             By Lemma \ref{lemA}, we deduce 
             $$
             \begin{aligned}
                 |I_1|\leq &\int_{\R^N}\left|(-\Delta)^s\frac{\partial W_{\lambda_i,\xi_i}}{\partial\lambda_i}-(p_s-\epsilon)W_{\lambda_i,\xi_i}^{p_s-1-\epsilon}\frac{\partial W_{\lambda_i,\xi_i}}{\partial\lambda_i}\right||\phi_\epsilon|\\
                 \leq& \int_{\R^N}\left|p_s\eta_i(x) U_{\lambda_i, \xi_i}^{p_s-1}\frac{\partial U_{\lambda_i, \xi_i}}{\partial\lambda_i}(x)+J_{i,0}(x)-(p_s-\epsilon)W_{\lambda_i,\xi_i}^{p_s-1-\epsilon}\frac{\partial W_{\lambda_i,\xi_i}}{\partial\lambda_i}\right||\phi_\epsilon|\\
                      \leq &\int_{B_\delta(\xi_i)}\left|p_s U_{\lambda_i, \xi_i}^{p_s-1}-(p_s-\epsilon)U_{\lambda_i,\xi_i}^{p_s-1-\epsilon}\right|\left|\frac{\partial U_{\lambda_i,\xi_i}}{\partial\lambda_i}\phi_\epsilon\right|
                 +C\lambda_i^{-1}\int_{B_{2\delta}(\xi_i)\setminus B_\delta(\xi_i)}U_{\lambda_i,\xi_i}^{p_s}|\phi_\epsilon|+\int_{\R^N}|J_{i,0}\phi_\epsilon|\\
    \leq&C\epsilon\lambda_i^{-1}\int_{B_\delta(\xi_i)}U_{\lambda_i,\xi_i}^{p_s-\theta\epsilon}(1+|\log U_{\lambda_i,\xi_i}|)|\phi_\epsilon|+C\lambda_i^{-\frac{N}{2}-1-s-\sigma}\|\phi_\epsilon\|_*+C\lambda_i^{-\frac{N}{2}-1+s}\|\phi_\epsilon\|_*\\
                 \leq&C\epsilon\lambda_i^{-1}\log\lambda_i\|\phi_\epsilon\|_*+C\lambda_i^{-\frac{N}{2}-1+s}\|\phi_\epsilon\|_*,\\
             \end{aligned}
             $$
             here we have used the estimate from Lemma \ref{lemA} 
             $$
             |J_{i,0}(x)|\leq C\frac{\lambda_i^{-\frac{N+2-2s}{2}}}{(1+|x-\xi_i|)^{N+2s}}\leq C\frac{\lambda_i^{-1+2s}}{(1+\lambda_i|x-\xi_i|)^{\frac{N+2s}{2}}}.
             $$
             In a similar way, we find 
            $$
            \begin{aligned}
                |I_2|\leq& C\lambda_i^{-1}\int_{B_{2\delta}(\xi_i)}U_{\lambda_i,\xi_i}|\phi_\epsilon|\\
                \leq& C\lambda_i^{-1}\|\phi_\epsilon\|_*\sum_{j=1}^k\int_{B_{2\delta}(\xi_i)}\frac{\lambda_i^{\frac{N-2s}{2}}}{(1+\lambda_i|x-\xi_i|)^{N-2s}}\frac{\lambda_j^{\frac{N-2s}{2}}}{(1+\lambda_j|x-\xi_j|)^{\frac{N-2s}{2}+\sigma}}\\
                \leq& C\lambda_i^{-1}\|\phi_\epsilon\|_*\int_{B_{2\delta}(\xi_i)}\frac{\lambda_i^{N-2s}}{(1+\lambda_i|x-\xi_i|)^{\frac{3}{2}N-3s+\sigma}}\\
                \leq& C
                \begin{cases}
                    \lambda_i^{-1-2s}\|\phi_\epsilon\|_*,\quad& \text{if}\quad N> 6s,\vspace{2mm}\\
                    \lambda_i^{-\frac{N}{2}-1+s}\|\phi_\epsilon\|_*,\quad& \text{if}\quad 4s<N\leq 6s.\vspace{2mm}\\
                \end{cases}
            \end{aligned}
            $$
     On the other hand,  for  $N>6s$, we have 
            $$
            \begin{aligned}
            |I_3|\leq &C\int_{\R^N}W_{\lambda_i, \xi_i}^{p_s-2-\epsilon}\phi_\epsilon^2\left|\frac{\partial W_{\lambda_i,\xi_i}}{\partial\lambda_i}\right|\leq C\lambda_i^{-1}\int_{\R^N}W_{\lambda_i, \xi_i}^{p_s-1-\epsilon}\phi_\epsilon^2\\
            \leq &C\|\phi_\epsilon\|_*^2\int_{\R^N}\frac{\lambda_i^{N-1}}{(1+\lambda_i|x-\xi_i|)^{N+2s+2\sigma-(N-2s)\epsilon}}\\
            \leq &C\lambda_i^{-1}\|\phi_\epsilon\|_*^2.
            \end{aligned}
            $$
            If ${N\leq 6s}$, then
            $$
            \begin{aligned}
                |I_3|\leq &C\int_{\R^N}W_{\lambda_i, \xi_i}^{p_s-2-\epsilon}\phi_\epsilon^2\left|\frac{\partial W_{\lambda_i,\xi_i}}{\partial\lambda_i}\right|+|\phi_\epsilon|^{p_s-\epsilon}\left|\frac{\partial W_{\lambda_i,\xi_i}}{\partial\lambda_i}\right|\\
                \leq& C\lambda_i^{-1}\|\phi_\epsilon\|_*^2+C\|\phi_\epsilon\|_*^{p_s-\epsilon}\int_{\R^N}\frac{\lambda_i^{N-1-\frac{N-2s}{2}\epsilon}}{(1+\lambda_i|x-\xi_i|)^{\frac{3}{2}N-s+(p_s-\epsilon)\sigma-\frac{N-2s}{2}\epsilon}}\\
\leq&C\lambda_i^{-1}\|\phi_\epsilon\|_*^2+C\lambda_i^{-1}\|\phi_\epsilon\|_*^{p_s-\epsilon}\\
                \leq &C\lambda_i^{-1}\|\phi_\epsilon\|_*^2.
            \end{aligned}
            $$
        The above estimates together with Lemma \ref{lemA.1} yield \eqref{eq2810-1}.
         \end{proof}
         
             By integration by parts, we have
         $$
        - \int_{B_\rho(\xi_j)}V(x)u\frac{\partial u}{\partial x_m}=-\frac{1}{2}\int_{\partial B_\rho(\xi_j)}V(x)u^2\nu_m+\frac{1}{2}\int_{B_\rho(\xi_j)}\frac{\partial V(x)}{\partial x_m}u^2,
         $$
         and 
         $$
         \int_{B_\rho(\xi_j)}({u})_+^{p_s-\epsilon}\frac{\partial u}{\partial x_m}=\frac{1}{2^*_s-\epsilon}\int_{\partial B_\rho(\xi_j)}({u})_+^{2^*_s-\epsilon}\nu_m.
         $$
         Then \eqref{eq3.4} is equivalent to 
         \begin{equation}\label{eq3.17}
             \begin{aligned}
                \frac{1}{2}\int_{B_\rho(\xi_j)}\frac{\partial V(x)}{\partial x_m}u^2=& -\int_{\partial''\mathfrak{B}_\rho^+(\xi_j)}t^{1-2s}\frac{\partial\widetilde{u}}{\partial\nu}\frac{\partial\widetilde{u}}{\partial x_m}
                +\frac{1}{2}\int_{\partial''\mathfrak{B}_\rho^+(\xi_j)}t^{1-2s}|\nabla\widetilde{u}|^2\nu_m\\
                &+\frac{1}{2}\int_{\partial B_\rho(\xi_j)}V(x)u^2\nu_m-\frac{1}{2^*_s-\epsilon}\int_{\partial B_\rho(\xi_j)}({u})_+^{2^*_s-\epsilon}\nu_m,\\
             \end{aligned}
         \end{equation}
         where $\rho\in(2\delta,\,5\delta)$, $ m=1,2,\cdots,N$ and $j=1,2,\cdots,k$.

        Based on \eqref {eq3.17}, we establish the following lemma.
         \begin{lemma}\label{lem3.3}
            \eqref{eq3.4} is equivalent to
           \begin{equation}\label{eq3.41}
         \frac{\partial V(\xi_j)}{\partial x_m}=o(1),\quad j=1,2,\cdots,k,\quad m=1,2,\cdots,N.
         \end{equation}
         \end{lemma}
         
         \begin{proof}
            By  Lemma \ref{lemGA.5} and Lemma \ref{newlemGA.6}, there exists  $\rho\in (2\delta,5\delta)$ such that
             $$
                 \begin{aligned}
                     \left|\int_{\partial''\mathfrak{B}_\rho^+(\xi_j)}t^{1-2s}\frac{\partial\widetilde{u}}{\partial\nu}\frac{\partial\widetilde{u}}{\partial x_m}\right|\leq &C\int_{\partial''\mathfrak{B}_\rho^+(\xi_j)}t^{1-2s}|\nabla \widetilde{\bm W}_{\bm \lambda, \bm \xi}|^2+C\int_{\partial''\mathfrak{B}_\rho^+(\xi_j)}t^{1-2s}|\nabla \widetilde{\phi}_\epsilon|^2\\
                     \leq & C\int_{\partial''\mathfrak{B}_\rho^+(\xi_j)}t^{1-2s}\left(\sum_{i=1}^k\frac{1}{\lambda_i^{\frac{N-2s}{2}}}\frac{1}{(1+|x-\xi_i|)^{N-2s+1}}\right)^2\\
                     &+C\lambda_j^{-2\sigma}\|\phi_\epsilon\|_*^2+C\lambda_j^{-2\sigma}\|\mathcal{R}\|_{**}\|\phi_\epsilon\|_*\\
                     \leq& C\lambda_j^{-N+2s}+C\lambda_j^{-2\sigma}\|\phi_\epsilon\|_*^2+C\lambda_j^{-2\sigma}\|\mathcal{R}\|_{**}\|\phi_\epsilon\|_*,
                 \end{aligned}
             $$
             where $\mathcal{R}$ is defined in \eqref{eq-2.24R}.
             Similarly, we find 
             $$
                 \left|\int_{\partial''\mathfrak{B}_\rho^+(\xi_j)}t^{1-2s}|\nabla\widetilde{u}|^2\nu_m\right|\leq C\lambda_j^{-N+2s}+C\lambda_j^{-2\sigma}\|\phi_\epsilon\|_*^2
    +C\lambda_j^{-2\sigma}\|\mathcal{R}\|_{**}\|\phi_\epsilon\|_*.
             $$
             Since $\eta_i=0$ on $\partial B_\rho$ for $i=1,2,\cdots,k$, we have $u=\phi_\epsilon$ on $\partial B_\rho$. 
             Then it holds  that
            $$
                 \left|\int_{\partial B_\rho(\xi_j)}V(x)u^2\nu_m\right|\leq C\|\phi_\epsilon\|_*^2\int_{\partial B_\rho(\xi_j)}\left(\sum_{j=1}^k\frac{\lambda_j^{\frac{N-2s}{2}}}{(1+\lambda_j|x-\xi_j|)^{\frac{N-2s}{2}+\sigma}}\right)^2\leq  C\lambda_j^{-2\sigma}\|\phi_\epsilon\|_*^2,
            $$
              and 
              $$
                  \left|\int_{\partial B_\rho(\xi_j)}({u})_+^{2^*_s-\epsilon}\nu_m\right|
                  \leq C\lambda_j^{-(2^*_s-\epsilon)\sigma}\|\phi_\epsilon\|_*^{2^*_s-\epsilon}.
              $$
              Combining the above estimates, we find that \eqref{eq3.17} is equivalent to 
              \begin{equation}\label{eq:20251104}
                  \int_{B_\rho(\xi_j)}\frac{\partial V(x)}{\partial x_m}u^2=O\left(\lambda_j^{-N+2s}\right)+O\left(\lambda_j^{-2\sigma}\|\phi_\epsilon\|_*^2\right)
+O\left(\lambda_j^{-2\sigma}\|\mathcal{R}\|_{**}\|\phi_\epsilon\|_*\right).
              \end{equation}
         It remains to estimate the left-hand side. Using the decay of $U_{\lambda_j,\xi_j}$, we obtain
         $$
         \begin{aligned}
             &\int_{B_{\rho}(\xi_j)}\frac{\partial V(x)}{\partial x_m}u^2\\
             =&\int_{B_{\rho}(\xi_j)}\frac{\partial V(x)}{\partial x_m}\eta_j^2U_{\lambda_j,\xi_j}^2+O\left(\int_{B_{\rho}(\xi_j)}(U_{\lambda_j,\xi_j}|\phi_\epsilon|+\phi_\epsilon^2)\right)\\
             =&\int_{B_{\delta}(\xi_j)}\frac{\partial V(x)}{\partial x_m}U_{\lambda_j,\xi_j}^2+O(\lambda_j^{-N+2s})+O\left(\lambda_j^{-2\sigma}\|\phi_\epsilon\|_*^2\right)+
                \begin{cases}
                   O\left( \lambda_j^{-2s}\|\phi_\epsilon\|_*\right),& \text{if}\quad N> 6s,\vspace{2mm}\\
                    O(\lambda_j^{-\frac{N}{2}+s}\|\phi_\epsilon\|_*),& \text{if}\quad N\leq 6s,\vspace{2mm}\\
                \end{cases}\\
                =&\gamma_{s,N}^2\lambda_j^{-2s}\frac{\partial V(\xi_j)}{\partial x_m}\int_{\R^N}\frac{1}{(1+|x|^2)^{N-2s}}+o(\lambda_j^{-2s}).
         \end{aligned}
         $$ 
As a result, we conclude
         \begin{equation}
         \frac{\partial V(\xi_j)}{\partial x_m}=o(1),\quad j=1,2,\cdots,k,\quad m=1,2,\cdots,N.
         \end{equation}
         \end{proof}   
         
         Now we are ready to give the proof of Theorem \ref{Mth}.
         \begin{proof}[Proof of Theorem \ref{Mth}]
          By     Lemma \ref{lem3.2} and Lemma \ref{lem3.3},  \eqref{eq3.4} and \eqref{eq3.5} are equivalent to 
             \begin{equation}\label{eqL3.15}
                 \frac{\partial V(\xi_j)}{\partial x_m}=o(1),\quad j=1,2,\cdots,k,\quad m=1,2,\cdots,N,
             \end{equation}
          and   \begin{equation}\label{eqL3.16}
                 A\epsilon\lambda_j^{-1}+BV(\xi_j)\lambda_j^{-2s-1}+o(\epsilon^{1+\frac{1}{2s}})=0,\quad j=1,2,\cdots,k.
             \end{equation}
             Let $\lambda_j=(t_j\epsilon)^{-\frac{1}{2s}}$, $t_j\in (-\frac{A}{2BV(\xi_j)},\,-\frac{2A}{BV(\xi_j)})$. 
           Then \eqref{eqL3.16} reduces to 
             \begin{equation}\label{eqL3.17}
                 A+BV(\xi_j)t_j+o(1)=0,\quad j=1,2,\cdots,k.
             \end{equation}
             Since $\xi_j^*$, $j=1,2,\cdots,k$ are stable critical points of $V(x)$ with $V(\xi_j^*)>0$. Hence there exists $\epsilon_0>0$ such that for $\epsilon\in (0,\epsilon_0)$, \eqref{eqL3.15} and \eqref{eqL3.17} have a solution $(\xi_{j}^\epsilon,\,\lambda_{j}^\epsilon),\; j=1,2,\cdots,k$ with  $\xi_{j}^\epsilon$  close to $\xi_j^*$ and $\lambda_{j}^\epsilon$ close to $\left(-\frac{A}{BV(\xi_j)}\epsilon\right)^{-\frac{1}{2s}}$ and the proof of Theorem \ref{Mth} is completed.
         \end{proof}

         \section{Non-degeneracy of multi-peak solutions}
         In this section, we study the non-degeneracy of multi-peak solutions $u_\epsilon={\bm W_{\bm\lambda_\epsilon,\bm\xi_\epsilon}}+\phi_\epsilon$ constructed previously. To this end, we consider the following linearized problem
          \begin{equation}\label{eqND}
          \begin{cases}
             (-\Delta)^s\omega_\epsilon+V(x)\omega_\epsilon-(p_s-\epsilon)u_\epsilon^{p_s-1-\epsilon}\omega_\epsilon=0 \quad\text{in}\ \R^N,\\
             \omega_\epsilon\in H_V^s(\mathbb{R}^N).
             \end{cases}
         \end{equation} 
         We aim to show that \eqref{eqND} admits only the trivial solution for sufficiently small $\epsilon$.
          Arguing by contradiction,  we suppose that there exists a nontrivial solution $\omega_\epsilon$ to \eqref{eqND} with   $\|\omega_\epsilon\|_*=1$. Throughout the subsequent sections, we omit the subscript $\epsilon$ on  $\bm{\lambda}_\epsilon, \bm{\xi}_\epsilon$ for the sake of simplicity.

\subsection{Improved estimates  and Lyapunov-Schmidt decomposition of $\omega_\epsilon$}

In this subsection, we proceed to derive improved estimates for the multi-peak solutions. 
To obtain finer estimates for $\omega_\epsilon$, we need to make Lyapunov-Schmidt type decomposition of $\omega_\epsilon$, 
which can be viewed as
the inverse procedure of the Lyapunov–Schmidt reduction.

\begin{lemma}\label{thm1.1}
Let $s>\frac{1}{2}$, $N\geq 6s$. 
Under the same assumptions as in Theorem \ref{NDth}, 
assume $u_\epsilon$ is the peak solutions of problem \eqref{eq:1} with the form \eqref{1a}, obtained in Theorem \ref{Mth}. Then 
\begin{equation*}\label{1q}
\|\phi_{\epsilon}\|
_{*}=O(\epsilon^{1-\frac{\sigma}{2s}})  \ \ \text{and} \ \ \frac{\partial V}{\partial x_{m}}(\xi_{i})=o\big(\epsilon^{\frac{1}{2s}}\big), \ \ m=1,\cdots,N,
\end{equation*}
where $\sigma\in\left(0,\min\big\{\frac{s}{2},2s-1\big\}\right)$.
 \end{lemma}
 \begin{proof}
     Since $N\geq 6s$, by repeating the proof of Proposition \ref{Prop2.4}, we obtain that 
      $\|\phi_\epsilon\|_*=O(\epsilon^{1-\frac{\sigma}{2s}})$. Note that  $V(x)\in C^2(B_{5\delta}(\xi^*_i))$. Then we have 
     $$
         \begin{aligned}
             \int_{B_{\rho}(\xi_j)}\frac{\partial V(x)}{\partial x_m}u_{\epsilon}^2=&\int_{B_{\rho}(\xi_j)}\frac{\partial V(x)}{\partial x_m}\eta_j^2U_{\lambda_j,\xi_j}^2+O\left(\int_{B_{\rho}(\xi_j)}(U_{\lambda_j,\xi_j}|\phi_\epsilon|+\phi_\epsilon^2)\right)\\
             =&\int_{B_{\delta}(\xi_j)}\frac{\partial V(x)}{\partial x_m}U_{\lambda_j,\xi_j}^2+O(\lambda_j^{-N+2s})+O\left(\lambda_j^{-2\sigma}\|\phi_\epsilon\|_*^2\right)+ O\left( \lambda_j^{-2s}\|\phi_\epsilon\|_*\right)\\
            =&\gamma_{s,N}^2\lambda_j^{-2s}\frac{\partial V}{\partial x_m}(\xi_j)\int_{\R^N}\frac{1}{(1+|x|^2)^{N-2s}}+o(\lambda_j^{-1-2s}).
         \end{aligned}
    $$
 By means of  \eqref{eq:20251104}, we derive that $\frac{\partial V}{\partial x_{m}}(\xi_{i})=o\big(\epsilon^{\frac{1}{2s}}\big)$.
 \end{proof}
         
         \begin{lemma}\label{lem4.1}
             There exists $C>0$ such that 
             $$
             \begin{aligned}
                |u_\epsilon(x)| \leq &C\sum_{i=1}^k\frac{\lambda_i^{\frac{N-2s}{2}}}{(1+\lambda_i|x-\xi_i|)^{N-2s}},\quad |\nabla u_\epsilon(x)| \leq C\sum_{i=1}^k\frac{\lambda_i^{\frac{N+2-2s}{2}}}{(1+\lambda_i|x-\xi_i|)^{N+1-2s}},\\
                 |\omega_\epsilon(x)| \leq &C\sum_{i=1}^k\frac{\lambda_i^{\frac{N-2s}{2}}}{(1+\lambda_i|x-\xi_i|)^{N-2s}},\quad |\nabla \omega_\epsilon(x)| \leq C\sum_{i=1}^k\frac{\lambda_i^{\frac{N+2-2s}{2}}}{(1+\lambda_i|x-\xi_i|)^{N+1-2s}}.\\
             \end{aligned}
             $$
         \end{lemma}
         \begin{proof}
           We only treat the estimates for $u_\epsilon$. The other estimates follow similarly.  
           It follows from  the Green's  representation formula  that 
             $$
             u_\epsilon(x)=\int_{\R^N}G(x,y)u_\epsilon^{p_s-\epsilon}(y)\;dy.
             $$
             Since 
             $$
             |u_\epsilon(x)|\leq C \sum_{i=1}^k\frac{\lambda_i^{\frac{N-2s}{2}}}{(1+\lambda_i|x-\xi_i|)^{\frac{N-2s}{2}+\sigma}},
             $$
by Lemma \ref{lemLL2.3}, we have
             $$
             \begin{aligned}
             |u_\epsilon(x)|
       \leq &C\sum_{i=1}^k\int_{\R^N}\frac{1}{|x-y|^{N-2s}}\frac{\lambda_i^{\frac{N+2s}{2}-\frac{N-2s}{2}\epsilon}}{(1+\lambda_i|y-\xi_i|)^{\frac{N+2s}{2}+p_s\sigma-(\frac{N-2s}{2}+\sigma)\epsilon}}\\
                  \leq &C\sum_{i=1}^k\int_{\R^N}\frac{1}{|\lambda_i(x-\xi_i)-y|^{N-2s}}\frac{\lambda_i^{\frac{N-2s}{2}-\frac{N-2s}{2}\epsilon}}{(1+|y|)^{\frac{N+2s}{2}+p_s\sigma-(\frac{N-2s}{2}+\sigma)\epsilon}}\\
                   \leq &C\sum_{i=1}^k\frac{\lambda_i^{\frac{N-2s}{2}}}{(1+\lambda_i|x-\xi_i|)^{\frac{N-2s}{2}+p_s\sigma-(\frac{N-2s}{2}+\sigma)\epsilon}}.\\
             \end{aligned}
             $$
           Note that 
           $$
           \frac{N-2s}{2}+p_s\sigma-\left(\frac{N-2s}{2}+\sigma\right)\epsilon>\frac{N-2s}{2}+\sigma \quad \text{for} \;\text{sufficiently \;small}  \; \epsilon, 
           $$
           iterating the above procedure yields the desired estimate for $u_\epsilon$. In addition, we have
           \[
            |\nabla u_\epsilon(x)|\leq C\int_{\R^N}\frac{1}{|x-y|^{N-2s+1}}u_\epsilon^{p_s-\epsilon}(y).
           \]
           Applying the similar arguments as above, we deduce the decay estimate for $\nabla u_\epsilon$. 
           This completes the proof of Lemma \ref{lem4.1}.
         \end{proof}
        Next, we perform a Lyapunov–Schmidt type decomposition for  $\omega_\epsilon$. More precisely, 
        we decompose $\omega_\epsilon$ into 
         \begin{equation}\label{eqY1}
\begin{cases}
\omega_\epsilon(x)=\displaystyle\sum_{i=1}^{k}
d_{i,0}\lambda_{i}Z_{i,0}
+\displaystyle\sum_{i=1}^{k}\displaystyle\sum_{l=1}
^{N}d_{i,l}\lambda_{i}^{-1}Z_{i,l}
+\omega_\epsilon^*(x),\\[3mm]

\displaystyle\int_{\mathbb{R}^N} W_{\lambda_{i},\xi_{i}}^{p-1}Z_{i,l}
\omega_\epsilon^*=0,\;i=1,2,\cdots,k,\; l=0,1,\cdots,N.
\end{cases}
\end{equation}
Then it is easy to check that  the projected coefficients $d_{i,l}$ are bounded for  $i=1,\cdots,k$ and $l=0,1,\cdots,N$. 
A direct computation  shows that 
         \begin{equation}
             \begin{aligned}
                 &(-\Delta)^s\omega_\epsilon^*+V(x)\omega_\epsilon^*-(p_s-\epsilon)u_\epsilon^{p_s-1-\epsilon}\omega_\epsilon^*\\
=&\sum_{i=1}^{k}
d_{i,0}\eta_{i}\lambda_{i}Z_{\lambda_i,\xi_i}^0\left((p_s-\epsilon)u_\epsilon^{p_s-1-\epsilon}-p_s U_{\lambda_i,\xi_i}^{p_s-1}\right)\\
&+\displaystyle\sum_{i=1}^{k}\displaystyle\sum_{l=1}
^{N}d_{i,l}\eta_{i}\lambda_{i}^{-1}Z_{\lambda_i,\xi_i}^l\left((p_s-\epsilon)u_\epsilon^{p_s-1-\epsilon}-p_s U_{\lambda_i,\xi_i}^{p_s-1}\right)\\
&-\sum_{i=1}^{k}
d_{i,0}\lambda_{i}J_{i,0}-\sum_{i=1}^{k}\displaystyle\sum_{l=1}
^{N}d_{i,l}\lambda_{i}^{-1}J_{i,l}-\sum_{i=1}^{k}\displaystyle\sum_{l=1}
^{N}d_{i,l}\lambda_{i}^{-1}\Lambda_{i,l}\\
&-V(x)\sum_{i=1}^{k}
d_{i,0}\eta_{i}\lambda_{i}Z_{\lambda_i,\xi_i}^0
-V(x)\displaystyle\sum_{i=1}^{k}\displaystyle\sum_{l=1}
^{N}d_{i,l}\eta_{i}\lambda_{i}^{-1}Z_{\lambda_i,\xi_i}^l
\\
&-V(x)\displaystyle\sum_{i=1}^{k}\displaystyle\sum_{l=1}
^{N}d_{i,l}\lambda_{i}^{-1}
U_{\lambda_i, \xi_i}\frac{\partial \eta_i }{\partial \xi_{i,l}}
-\displaystyle\sum_{i=1}^{k}\displaystyle\sum_{l=1}
^{N}d_{i,l}\lambda_{i}^{-1}
U^{p_s}_{\lambda_i, \xi_i}\frac{\partial \eta_i }{\partial \xi_{i,l}}\\
&+(p_s-\epsilon)u_\epsilon^{p_s-1-\epsilon}\displaystyle\sum_{i=1}^{k}\displaystyle\sum_{l=1}
^{N}d_{i,l}\lambda_{i}^{-1}
U_{\lambda_i, \xi_i}\frac{\partial \eta_i }{\partial \xi_{i,l}}\\
:=&\beta=\beta_1+\beta_2+\cdots+\beta_{10},\\
             \end{aligned}
         \end{equation}
where $J_{i,l}$ and $\Lambda_{i,l}$ are defined  in Lemma \ref{lemA}.
We first estimate the projected remainder term $\omega_\epsilon^*$.
\begin{lemma}\label{lem4.2}
There exists $C>0$ such that
\[
\|\omega_\epsilon^*\|_*\leq C
\epsilon^{1-\frac{\sigma}{2s}}.
\]
\end{lemma}

\begin{proof}
We first estimate  $\beta_1$ and divide the estimate into two regions.\\
    Case 1: $x\in B_{\delta}(\xi_i)$.  
    $$
    \begin{aligned}
        |\beta_1|\leq& C\sum_{i=1}^kU_{\lambda_i,\xi_i}\left|(p_s-\epsilon)(U_{\lambda_i,\xi_i}+\phi_\epsilon)^{p_s-1-\epsilon}-p_s U_{\lambda_i,\xi_i}^{p_s-1}\right|\\
        \leq &C\sum_{i=1}^kU_{\lambda_i,\xi_i}\left|(U_{\lambda_i,\xi_i}+\phi_\epsilon)^{p_s-1-\epsilon}-U_{\lambda_i,\xi_i}^{p_s-1-\epsilon}\right|+C\sum_{i=1}^kU_{\lambda_i,\xi_i}\left|(p_s-\epsilon)U_{\lambda_i,\xi_i}^{p_s-1-\epsilon}-p_s U_{\lambda_i,\xi_i}^{p_s-1}\right|\\
        :=&\beta_{11}+\beta_{12}.
    \end{aligned}
    $$ 
It follows from   Lemma \ref{lemL2.8} that 
    $$
    \begin{aligned}
        |\beta_{11}|\leq&C\sum_{i=1}^kU_{\lambda_i,\xi_i}^{p_s-1-\epsilon}|\phi_\epsilon|
        \leq C\|\phi_\epsilon\|_*\sum_{i=1}^k\frac{\lambda_i^{\frac{N+2s}{2}}}{(1+\lambda_i|x-\xi_i|)^{\frac{N+2s}{2}+\sigma}}.
    \end{aligned}
    $$
    By the mean value theorem, we obtain 
    $$
    \begin{aligned}
        |\beta_{12}|\leq &  C\epsilon \sum_{i=1}^k U_{\lambda_i,\xi_i}^{p_s-\theta\epsilon}\left(1+|\log U_{\lambda_i,\xi_i}|\right)
     \leq C\epsilon|\log\epsilon|\sum_{i=1}^k\frac{\lambda_i^{\frac{N+2s}{2}}}{(1+\lambda_i|x-\xi_i|)^{\frac{N+2s}{2}+\sigma}},
    \end{aligned}
    $$
    where $\theta\in (0,1)$.\\
Case 2: $x\in B_{2\delta}(\xi_i)\setminus B_{\delta}(\xi_i)$. Applying Lemma \ref{lem4.1}, we find that  $|u_\epsilon|\leq C\lambda_i^{-\frac{N-2s}{2}}$.
Then we have 
     
    $$
    |\beta_{1}|\leq C\epsilon^{1-\frac{\sigma}{2s}}\displaystyle \frac{\lambda_i^{\frac{N+2s}{2}}}{(1+\lambda_i|x-\xi_i|)^{\frac{N+2s}{2}+\sigma}},
    $$
 which gives that  $\|\beta_1\|_{**}\leq C\epsilon^{1-\frac{\sigma}{2s}}$. 
Similarly, we derive  $\|\beta_2\|_{**}\leq C\epsilon^{1-\frac{\sigma}{2s}}$.

    Using Lemma \ref{lemA},  we infer 
    $$
    |\lambda_iJ_{i,0}|+
    |\lambda_i^{-1}J_{i,l}|+ |\Lambda_{i,l}|\leq \frac{C\lambda_i^{-\frac{N-2s}{2}}}{(1+|x-\xi_i|)^{N+2s}}.
    $$
  Hence we obtain 
    $$
    \begin{aligned}
        |\beta_3|+|\beta_4|&\leq C\sum_{i=1}^k\frac{\lambda_i^{-\frac{N-2s}{2}}}{(1+|x-\xi_i|)^{N+2s}}
        \leq C\epsilon^{1-\frac{\sigma}{2s}} \displaystyle \sum_{i=1}^k\frac{\lambda_i^{\frac{N+2s}{2}}}{(1+
        \lambda_i|x-\xi_i|)^{\frac{N+2s}{2}+\sigma}}
    \end{aligned}
    $$
    and $$
    \begin{aligned}
        |\beta_5|&\leq C\sum_{i=1}^k\frac{\lambda_i^{-\frac{N-2s+2}{2}}}{(1+|x-\xi_i|)^{N+2s}}
        \leq C\epsilon^{1-\frac{\sigma}{2s}} \displaystyle \sum_{i=1}^k\frac{\lambda_i^{\frac{N+2s}{2}}}{(1+\lambda_i|x-\xi_i|)^{\frac{N+2s}{2}+\sigma}},
    \end{aligned}
    $$
which implies that 
    $$
    \|\beta_3\|_{**}+\|\beta_4\|_{**}+\|\beta_5\|_{**}\leq C\epsilon^{1-\frac{\sigma}{2s}}.
    $$
A direct calculation leads to 
    $$
    \begin{aligned}
        |\beta_6|+ |\beta_7|\leq  C\sum_{i=1}^k\frac{\lambda_i^{\frac{N-2s}{2}}}{(1+\lambda_i|x-\xi_i|)^{N-2s}}\chi_{B_{2\delta}(\xi_i)}(x)
        \leq C\displaystyle \sum_{i=1}^k\lambda_i^{-2s+\sigma}\frac{\lambda_i^{\frac{N+2s}{2}}}{(1+\lambda_i|x-\xi_i|)^{\frac{N+2s}{2}+\sigma}}.
    \end{aligned}
    $$
Thanks to Lemma \ref{lem4.1}, we get 
$$
    \begin{aligned}
        |\beta_{8}|+ |\beta_9|+ |\beta_{10}|&\leq C\sum_{i=1}^k
        \chi_{B_{2\delta}(\xi_i)\setminus B_{\delta}(\xi_i)}
        \lambda_i^{-\frac{N+2s+2}{2}}+C\sum_{i=1}^k
        \chi_{B_{2\delta}(\xi_i)\setminus B_{\delta}(\xi_i)}
        \lambda_i^{-\frac{N-2s+2}{2}}
        \\
        &\leq C\epsilon^{1-\frac{\sigma}{2s}} \displaystyle \sum_{i=1}^k\frac{\lambda_i^{\frac{N+2s}{2}}}{(1+\lambda_i|x-\xi_i|)^{\frac{N+2s}{2}+\sigma}}.
    \end{aligned}
    $$
As a result, we arrive at  
  $$ \|\beta\|_{**}\leq\sum_{i=1}^{10} \|\beta_{i}\|_{**}\leq C\epsilon^{1-\frac{\sigma}{2s}}.
    $$
Following the argument in   Proposition\,\ref{prop2.3}, we conclude that 
\[
\|\omega_\epsilon^*\|_*\leq C\epsilon^{1-\frac{\sigma}{2s}}.
\]
The proof of Lemma \ref{lem4.2} is completed. 
\end{proof}

\subsection{The estimates of the projected coefficients via local Pohozaev identities}

In this subsection, we first establish  several local Pohozaev identities and  use them  to estimate the projected coefficients $d_{i,l}$,
 $i=1,\cdots,k$, $l=0,1,\cdots,N$. Recall that the $s$-harmonic extension  $\widetilde{u}_\epsilon$ and $\widetilde\omega_\epsilon$
  satisfy
    \begin{equation}\label{eq-ND1}
        \begin{cases}
            \text{div}(t^{1-2s}\nabla\widetilde{u}_\epsilon)=0 \quad &\text{in} \quad \R^{N+1}_+,\\
            -\lim\limits_{t\to 0^+}t^{1-2s}\partial_t\widetilde{u}_\epsilon=-V(x){u}_\epsilon+{u}_\epsilon^{p_s-\epsilon}\quad &\text{on}  \quad\R^{N},
        \end{cases}
         \end{equation}
         and 
    \begin{equation}\label{eq-ND2}
        \begin{cases}
            \text{div}(t^{1-2s}\nabla\widetilde\omega_\epsilon)=0 \quad &\text{in} \quad\R^{N+1}_+,\\
            -\lim\limits_{t\to 0^+}t^{1-2s}\partial_t\widetilde\omega_\epsilon=-V(x)\omega_\epsilon+(p_s-\epsilon)u_\epsilon^{p_s-1-\epsilon}\omega_\epsilon\quad &\text{on}  \quad\R^{N}.
        \end{cases}
         \end{equation}
         
        We first derive the following local Pohozaev identities associated with  $\widetilde{u}_\epsilon$ and $\widetilde\omega_\epsilon$. 

\begin{lemma}\label{lem4.3}
Assume that $V(x)\in C^2\big(B_\rho(\xi_{i})\big)$, where $\rho\in(2\delta,\,5\delta)$. Then for $i=1,\cdots,k$ and $l=1,\cdots,N$,
\begin{equation}\label{4.13}
\begin{split}
&\frac{N-2}{2}\epsilon\int_{B_\rho(
\xi_{i})} u_\epsilon^{p_{s}-\epsilon}\omega_\epsilon
-2s\int_{B_\rho(\xi_{i})}V(x)u_\epsilon
\omega_\epsilon
-\int_{B_\rho(\xi_{i})}\big\langle x-\xi_{i},\nabla V(x)\big\rangle u_\epsilon\omega_\epsilon\\
=&\int_{\partial^{''} \mathfrak{B}^{+}_\rho(\xi_{i})}
t^{1-2s}\frac{\partial \widetilde{u}_\epsilon}{\partial\nu}\langle X-(\xi_{i},0),\nabla\widetilde\omega_\epsilon\rangle
+\int_{\partial^{''} \mathfrak{B}^{+}_\rho(\xi_{i})}
t^{1-2s}\frac{\partial \widetilde\omega_\epsilon}{\partial\nu}\langle X-(\xi_{i},0),\nabla \widetilde{u}_\epsilon\rangle
\\
&+
\frac{N-2s}{2}\int_{\partial^{''} \mathfrak{B}^{+}_\rho(\xi_{i})}
t^{1-2s}\left(\frac{\partial \widetilde{u}_\epsilon}{\partial\nu}\widetilde\omega_\epsilon+\frac{\partial \widetilde\omega_\epsilon}{\partial\nu}\widetilde{u}_\epsilon\right)
-\int_{\partial^{''} \mathfrak{B}^{+}_\rho(\xi_{i})}
t^{1-2s}\langle\nabla \widetilde{u}_\epsilon,\nabla\widetilde\omega_\epsilon\rangle\langle X-(\xi_{i},0),\nu\rangle
\\
&+\int_{\partial B_\rho(\xi_{i})} u_\epsilon^{p_{s}-\epsilon}\omega_\epsilon \langle x-\xi_{i},\nu\rangle -\int_{\partial B_\rho(\xi_{i})}V(x)u_\epsilon
\omega_\epsilon\langle x-\xi_{i},\nu\rangle,
\end{split}
\end{equation}
\begin{equation}\label{4.3}
\begin{split}
\int_{B_\rho(\xi_{i})}\frac{\partial V(x)}{\partial x_l}u_\epsilon\omega_\epsilon
=&\int_{\partial B_\rho(\xi_{i})}V(x)u_\epsilon
\omega_\epsilon
\nu_{l}
-\int_{\partial B_\rho(\xi_{i})} u_\epsilon^{p_{s}-\epsilon}\omega_\epsilon \nu_{l}
+
\int_{\partial^{''} \mathfrak{B}^{+}_\rho(\xi_{i})}
t^{1-2s}\langle\nabla \widetilde{u}_\epsilon,\nabla\widetilde\omega_\epsilon\rangle\nu_{l}\\
&- \int_{\partial^{''} \mathfrak{B}^{+}_\rho(\xi_{i})}
t^{1-2s}\frac{\partial \widetilde{u}_\epsilon}{\partial x_l}\frac{\partial\widetilde\omega_\epsilon}{\partial\nu}
-\int_{\partial^{''} \mathfrak{B}^{+}_\rho(\xi_{i})}
t^{1-2s}\frac{\partial
\widetilde\omega_\epsilon}{\partial x_l}\frac{\partial \widetilde{u}_\epsilon}{\partial\nu}.
\end{split}
\end{equation}
Hereafter, $\langle \cdot, \cdot \rangle$ denotes the Euclidean inner product. 
\end{lemma}

\begin{proof}
The proof can be found in Lemmas \ref{Flem4.3} and \ref{Flem4.30}. 
\end{proof}

Next we employ the above local Pohozaev identities to estimate  the projected coefficients $d_{i,l}$.
\begin{lemma}\label{lem4.4}Let $N\geq6s$ and $V(x)\in C^{2}\big(B_\rho(\xi_{i})\big)$ with $\rho\in(2\delta,\,5\delta)$. Then there exists a positive constant $C$ such that 
\begin{equation*}
|d_{i,0}|\leq C\epsilon^{1-\frac{\sigma}{2s}}, \ \ i=1,\cdots,k.
\end{equation*}
\end{lemma}

\begin{proof}
Note that  $$\displaystyle\int_{\mathbb{R}^N} U_{\lambda_i,\xi_i}^{p_{s}}Z_{\lambda_i,\xi_i}^0=0\;  \;\;\text{and}\;\;\;\displaystyle\int_{B_\delta(\xi_i)}U_{\lambda_i,\xi_i}^{p_{s}-\epsilon}Z_{\lambda_i,\xi_i}^l=0.$$  
Then it follows from  Lemma \ref{lem4.1}  that 
\begin{equation}\label{a4.4}
\begin{split}
&\int_{B_\rho(\xi_{i})}   u_\epsilon^{p_{s}-\epsilon}\omega_\epsilon\\
=&\int_{B_\delta(\xi_{i})}  (U_{\lambda_i,\xi_i}+\phi_\epsilon)^{p_{s}-\epsilon}
\Big(\displaystyle
d_{i,0}\lambda_iZ_{\lambda_i,\xi_i}^0
+\displaystyle\sum_{l=1}
^{N}d_{i,l}\lambda_i^{-1}
Z_{\lambda_i,\xi_i}^l
+\omega_\epsilon^*(x)\Big)+\int_{B_\rho(\xi_{i})\setminus B_\delta(\xi_{i})}   u_\epsilon^{p_{s}-\epsilon}\omega_\epsilon\\
=&d_{i,0}\lambda_i
\int_{B_\delta(\xi_{i})} \left[(U_{\lambda_i,\xi_i}+\phi_\epsilon)
^{p_{s}-\epsilon}
-U_{\lambda_i,\xi_i}^{p_{s}-\epsilon}\right]Z_{\lambda_i,\xi_i}^0
+d_{i,0}\lambda_i
\int_{B_\delta(\xi_{i})} \left[U_{\lambda_i,\xi_i}^{p_{s}-\epsilon}
-U_{\lambda_i,\xi_i}^{p_{s}}\right]Z_{\lambda_i,\xi_i}^0\\
&+\displaystyle\sum_{l=1}^{N}d_{i,l}
\lambda_i^{-1}
\int_{B_\delta(\xi_{i})} \left[(U_{\lambda_i,\xi_i}+\phi_\epsilon)^{p_{s}-\epsilon}
-U_{\lambda_i,\xi_i}^{p_{s}-\epsilon}\right]Z_{\lambda_i,\xi_i}^l+O(\lambda_i^{-N})+O(\|\omega_\epsilon^*\|_*)
\\
:=&\Gamma_{1}+\Gamma_{2}+\Gamma_{3}
+O(\lambda_i^{-N})+O(\|\omega_\epsilon^*\|_*).
\end{split}
\end{equation}
A direct computation shows that 
\[
\begin{split}
|\Gamma_1|\leq& C\int_{B_\delta(\xi_{i})} U_{\lambda_i,\xi_{i}
}^{p_{s}-\epsilon}|\phi_\epsilon|+C
\int_{B_\delta(\xi_{i})} U_{\lambda_i,\xi_{i}}|
\phi_\epsilon|^{p_{s}-\epsilon}
\leq C\big(\|\phi_\epsilon\|_*+\|\phi_\epsilon
\|_*^{p_{s}-\epsilon}\big).
\end{split}
\]
Using the mean value theorem, we see that
\[
|\Gamma_2|\leq C\epsilon \int_{B_\delta(\xi_{i})} U_{\lambda_i,\xi_{i}
}^{p_{s}+1-\theta\epsilon}|\log U_{\lambda_i,\xi_{i}}|\leq C\epsilon|\log\epsilon|,
\]
where $\theta\in(0,1)$.
Similar to the estimate of $\Gamma_1$, we  get
$$|\Gamma_3|\leq  C\|\phi_\epsilon\|_*+C\|\phi_\epsilon
\|_*^{p_{s}-\epsilon}.$$
Combining Lemma \ref{thm1.1} and Lemma \ref{lem4.2},  we derive 
\begin{equation*}
\int_{B_\delta(\xi_{i})} u_\epsilon^{p_{s}-\epsilon}\omega_\epsilon\leq C\epsilon^{1-\frac{\sigma}{2s}}.
\end{equation*}

On the other hand, we have 
\[
\begin{split}
&\int_{B_\rho(\xi_{i})}V(x)
u_\epsilon\omega_\epsilon\\
=&d_{i,0}\lambda_i
\int_{B_\delta(\xi_{i})} V(x)U_{\lambda_i,\xi_{i}}
Z_{\lambda_i,\xi_i}^0
+\displaystyle\sum_{l=1}
^{N}d_{i,l}\lambda_i^{-1}
\int_{B_\delta(\xi_{i})} V(x)U_{\lambda_i,\xi_{i}}
Z_{\lambda_i,\xi_i}^l+O(\lambda_i^{-2s}\|\phi_{\epsilon}\|_*)\\
&+O(\lambda_i^{-2s}\|\omega_\epsilon^*\|_*)+O(\lambda_i^{-N+2s})\vspace{3mm}\\
:=&T_{1}+T_{2}+O(\lambda_i^{-2s}\|\phi_{\epsilon}\|_*)
+O(\lambda_i^{-2s}\|\omega_\epsilon^*\|_*)+O(\lambda_i^{-N+2s}).
\end{split}
\]
Then we infer 
\[
\begin{split}
T_{1}
=& d_{i,0}\lambda_iV(\xi_{i} ) \int_{B_\delta(\xi_{i})}
U_{\lambda_i,\xi_{i}}
Z_{\lambda_i,\xi_i}^0+O\left( |d_{i,0}|\lambda_i^{-1-2s}\right)
\\
=&\frac{N-2s}{2}\gamma_{s,N}^2 d_{i,0} V(\xi_{i})\lambda_i^{-2s}\int_{\mathbb{R}^N}\frac{1-|x|^{2}}
{(1+|x|^{2})^{N-2s+1}}+
O\big(\lambda_i^{2s-N}\big)
+O\left( |d_{i,0}|\lambda_i^{-1-2s}
\right). \\
\end{split}
\]
By virtue of  Lemma \ref{thm1.1}, we find that   $\nabla V(\xi_{i})=o(\epsilon^{\frac{1}{2s}})$. Hence we obtain 
\[
\begin{split}
T_{2}
=&\displaystyle\sum_{l=1}
^{N}d_{i,l}\lambda_i^{-1}\frac{\partial V(\xi_{i})}{\partial x_l}
\int_{B_\delta(\xi_{i})}
U_{\lambda_i,\xi_{i}}
Z_{\lambda_i,\xi_i}^l
(x_{l}-\xi_{i,l})
+O\Big(\lambda_i^{-2-2s}\Big)
\\
=&(N-2s)\gamma_{s,N}^2\displaystyle\sum_{l=1}
^{N}d_{i,l}\lambda_i^{-1-2s}\frac{\partial V(\xi_{i})}{\partial x_{l}}
\int_{B_{\lambda_i\delta}(0)}
\frac{x_l^{2}}
{(1+|x|^{2})^{N-2s+1}}
+O\Big(\lambda_i^{-2-2s}\Big)
\\
=&O\Big(\lambda_i^{-1-2s}
\big|\nabla V(\xi_i)\big|\Big)
+O\Big(\lambda_i^{-2-2s}\Big)
\\
=&O\Big(\lambda_i^{-2-2s}\Big).
\end{split}
\]
Moreover, we see that 
\begin{equation*}\label{a4.6}
\begin{split}
&\int_{B_\rho(\xi_{i})}\big\langle x-\xi_{i},\nabla V(x)\big\rangle u_\epsilon\omega_\epsilon\\
=&
d_{i,0}\lambda_i
\int_{B_\delta(\xi_{i})}\big\langle x-\xi_{i},\nabla V(x)\big\rangle U_{\lambda_i,\xi_{i}}Z_{\lambda_i,\xi_i}^0+\displaystyle\sum_{l=1}
^{N}d_{i,l}\lambda_i^{-1}
\int_{B_\delta(\xi_{i})}\big\langle x-\xi_{i},\nabla V(x)\big\rangle U_{\lambda_i,\xi_{i}}Z_{\lambda_i,\xi_i}^l\\
&
+\int_{B_\delta(\xi_{i})}\big\langle x-\xi_{i},\nabla V(x)\big\rangle U_{\lambda_i,\xi_{i}}
\omega_\epsilon^*
+d_{i,0}\lambda_i
\int_{B_\delta(\xi_{i})}\big\langle x-\xi_{i},\nabla V(x)\big\rangle  Z_{\lambda_i,\xi_i}^0\phi_\epsilon
\\
&
+
\displaystyle\sum_{l=1}
^{N}d_{i,l}\lambda_i^{-1}
\int_{B_\delta(\xi_{i})}\big\langle x-\xi_{i},\nabla V(x)\big\rangle   Z_{\lambda_i,\xi_i}^l\phi_\epsilon
+\int_{B_\delta(\xi_{i})}\big\langle x-\xi_{i},\nabla V(x)\big\rangle \phi_{\epsilon}\omega_\epsilon^*+O(\lambda_i^{-N+2s})\\
:=&D_{1}+D_{2}+D_{3}+D_{4}+D_{5}+D_{6}+O(\lambda_i^{-N+2s}).
\end{split}
\end{equation*}
Since $V(x)\in C^{2}\big(B_\delta(\xi_{i})\big)$ and $\displaystyle\int_{B_\delta(\xi_i)}\langle x-\xi_i, \nabla V(\xi_i)\rangle U_{\lambda_i,\xi_i}\frac{\partial U_{\lambda_i,\xi_i}}
{\partial\lambda_i}=0$, we have
\[
\begin{split}
D_{1}=&
d_{i,0}\lambda_i
\int_{B_\delta(\xi_{i})}\big\langle x-\xi_{i},\nabla V(x)-\nabla V(\xi_{i})\big\rangle U_{\lambda_i,\xi_{i}}Z_{\lambda_i,\xi_i}^0
\\
=& O\left(\lambda_i\displaystyle\sum_{l=1}
^{N}
\int_{B_\delta(\xi_{i})} |x_{l}-\xi_{i,l}|^{2} U_{\lambda_i,\xi_{i}}
Z_{\lambda_i,\xi_i}^0\right) 
\\
=&O\big(\epsilon^{2-\frac{\sigma}{2s}}\big). 
\end{split}
\]
Thanks  to Lemma \ref{thm1.1}, we derive 
\[
\begin{split}
D_{2}=&
\displaystyle\sum_{l=1}
^{N}d_{i,l}\lambda_i^{-1}\frac{\partial V(\xi_{i})}{\partial x_{l}}
\int_{B_\delta(\xi_{i})} U_{\lambda_i,\xi_{i}}\frac{\partial U_{\lambda_i,\xi_{i}}}
{\partial\xi_{i,l}}
\big( x_{l}-\xi_{i,l}\big) 
+
O(\lambda_i^{-2-2s})
\\
=&  \displaystyle\sum_{l=1}
^{N}d_{i,l}\lambda_i^{-1-2s}\frac{\partial V(\xi_{i})}{\partial x_{l}}
\int_{B_{\delta\lambda_i}(0)}
\frac{x_l^{2}}
{(1+|x|^{2})^{N-1}}+
O(\lambda_i^{-2-2s})
\\
=&O(\lambda_i^{-2-2s}).
\end{split}
\]
A straightforward calculation leads to 
\[
\begin{split}
|D_{3}|
\leq C\lambda_i^{-2s}\|\omega_\epsilon^*\|_*
=O\big(\epsilon^{2-\frac{\sigma}{2s}}\big).
\end{split}
\]
In a similar way, we get 
\[
D_{4}=O\big(\lambda_i^{-2s}\|\phi_\epsilon\|_*\big),\ \ D_{5}=O\big(\lambda_i^{-1-2s}\|\phi_\epsilon\|_*\big)
\]
and
\[
\begin{split}
|D_{6}|
\leq C\lambda_i^{-2\sigma}\|\omega_\epsilon^*\|_*
\|\phi_\epsilon\|_*=O\big(\epsilon^{2}\big). 
\end{split}
\]
Thus, we conclude that 
$$
\text{LHS\;of\;}\eqref{4.13}=\frac{N-2s}{2}\gamma_{s,N}^2d_{i,0}(1+o(1)) V(\xi_{i})\lambda_i^{-2s}\int_{\mathbb{R}^N}\frac{1-|x|^{2}}
{(1+|x|^{2})^{N-2s+1}}
+O\big(\epsilon^{2-\frac{\sigma}{2s}}
\big).
$$
It follows from  Lemmas \ref{lem4.1},  \ref{lemGA.5} and \ref{newlemGA.6} that
 $$\text{RHS\;of\;}\eqref{4.13}=
 O\big(\epsilon^{2-\frac{\sigma}{2s}}
\big).$$
As a result, we obtain 
$$|d_{i,0}|\leq C\epsilon^{1-\frac{\sigma}{2s}},\ \ i=1,\cdots,k,$$ and  the desired estimate follows.
\end{proof}

\begin{lemma}\label{lem4.04}Let $N\geq6s$ and $V(x)\in C^{2}\big(B_\rho(\xi_{i})\big) $ with $\rho\in(2\delta,\,5\delta)$ for $i=1,\cdots,k$. Then there exists a positive constant $C$ such that  \begin{equation*}
|d_{i,l}|\leq C\epsilon^{1-\frac{1+\sigma}{2s}}
, \ \  i=1,\cdots,k, \ \ l=1,\cdots,N.\end{equation*}
\end{lemma}

\begin{proof}[\bf Proof.]
By Lemma \ref{lem4.1}, we have 
\[
\begin{split}
&\int_{B_\rho(\xi_{i})}\frac{\partial V(x)}{\partial x_j}u_\epsilon\omega_\epsilon
\\
=&d_{i,0}\lambda_i\int_{B_\delta(\xi_{i})}
\frac{\partial V(x)}{\partial x_j}U_{\lambda_i,\xi_{i}}Z_{\lambda_i,\xi_i}^0
+\displaystyle\sum_{l=1}
^{N}d_{i,l}\lambda_i^{-1}
\int_{B_\delta(\xi_{i})} \frac{\partial V(x)}{\partial x_j}
U_{\lambda_i,\xi_{i}}Z_{\lambda_i,\xi_i}^l\\
&+O(\lambda_i^{-2s}\|\omega_\epsilon^*\|_*)+O(\lambda_i^{-2s}\|\phi_{\epsilon}\|_*)+O(\lambda_i^{-N+2s})\vspace{3mm}\\
:=&M_{1}+M_{2}
+O(\lambda_i^{-2s}\|\omega_\epsilon^*\|_*)+O(\lambda_i^{-2s}\|\phi_{\epsilon}\|_*)+O(\lambda_i^{-N+2s})
.
\end{split}
\]
Then it follows from direct computation that
\[
\begin{split}
M_{1}%
=& d_{i,0}\lambda_i\frac{\partial V(\xi_{i})}{\partial x_j} \int_{B_\delta(\xi_{i})}
U_{\lambda_i,\xi_{i}}Z_{\lambda_i,\xi_i}^0
+O\left(|d_{i,0}|\lambda_i^{-1-2s}\right)
\\
=&\frac{N-2s}{2}\gamma_{s,N}^2d_{i,0} \frac{\partial V(\xi_{i}) }{\partial x_j} \lambda_i^{-2s}\int_{\mathbb{R}^N}
\frac{1-|x|^{2}}
{(1+|x|^{2})^{N+1-2s}}
+O\left(|d_{i,0}|\lambda_i^{-1-2s}\right)
\\
=&O\left(|d_{i,0}|\lambda_i^{-1-2s}\right) 
\end{split}
\]
and
\[
\begin{split}
M_{2}=& \displaystyle\sum_{l=1}
^{N}d_{i,l}\lambda_i^{-1}
\int_{B_\delta(\xi_{i})}
\bigg(
\frac{\partial V(x)}{\partial x_{j}}-
\frac{\partial V(\xi_{i})}{\partial x_{j}}
\bigg)
 U_{\lambda_i,\xi_{i}}
 Z_{\lambda_i,\xi_i}^l \\
=&\frac{(N-2s)}{N}\gamma_{s,N}^2\lambda_i^{-1-2s}\displaystyle\sum_{l=1}
^{N}
d_{i,l}\left(\frac{\partial^{2} V(\xi_{i})}{\partial x_{j} \partial x_{l}}+o(1)\right)
\int_{\R^N}
\frac{|x|^{2}}
{(1+|x|^{2})^{N-2s+1}}
\end{split}
\]
 By  Lemmas\, \ref{thm1.1}, \ref{lem4.2} and  \ref{lem4.4}, we deduce
 $$
\begin{aligned}
\text{LHS\;of\;}\eqref{4.3}=&\frac{(N-2s)}{N}\gamma_{s,N}^2\lambda_i^{-1-2s}\displaystyle\sum_{l=1}
^{N}
d_{i,l}\left(\frac{\partial^{2} V(\xi_{i})}{\partial x_{j} \partial x_{l}}+o(1)\right)
\int_{\R^N}
\frac{|x|^{2}}
{(1+|x|^{2})^{N-2s+1}}\\
&+O(\epsilon^{2-\frac{\sigma}{2s}})+O(\epsilon^2).
\end{aligned}
$$
Applying Lemmas \ref{lem4.1}, \ref{lemGA.5} and  \ref{newlemGA.6}, we find that 
 $$\text{RHS\;of\;}\eqref{4.3}= O\big(\epsilon^{2-\frac{\sigma}{2s}}
\big).$$
Hence we obtain 
 \begin{equation*}|d_{i,l}|\leq C\epsilon^{1-\frac{1+\sigma}{2s}}, \ \ \ \  i=1,\cdots,k, \ \ l=1,\cdots,N.\end{equation*} 
 This concludes the proof of Lemma \ref{lem4.04}.
\end{proof}

\subsection{Completion of the proof of Theorem \ref{NDth}}
\begin{proof}[\bf Proof of Theorem \ref{NDth}]
Thanks to Lemmas \ref{lem4.1} and  \ref{lemLL2.3}, we have 
\[
\begin{aligned}
|\omega_\epsilon|&\leq C\int_{\mathbb{R}^N}\frac{1}{|x-y|^{N-2s}} u_\epsilon^{p_{s}-1-\epsilon}(y)
\omega_\epsilon(y)\\
&\leq C\|\omega_\epsilon\|_*
\displaystyle\sum_{i=1}^{k}
\int_{\mathbb{R}^N}\frac{1}{|x-y|^{N-2s}}
\frac{\lambda_{i}
^{\frac{N+2s}{2}}}
{(1+\lambda_{i}|y-\xi_{i}|)
^{\frac{N}{2}+3s+\sigma-(N-2s)\epsilon} }\\
&\leq C\|\omega_\epsilon\|_*
\displaystyle\sum_{i=1}^{k}
\int_{\mathbb{R}^N}\frac{1}
{|z-\lambda_{i}(x-\xi_{i})|^{N-2s}}
\frac{\lambda_{i}
^{\frac{N-2s}{2}}}
{(1+|z|)
^{\frac{N}{2}+3s-\sigma_{0}}} 
\\
&
\leq C\|\omega_\epsilon\|_*
\displaystyle\sum_{i=1}^{k}\frac{\lambda_{i}
^{\frac{N-2s}{2}}}
{(1+\lambda_{i}|x-\xi_{i}|)
^{\frac{N+2s}{2}-\sigma_{0}}},
\end{aligned}
\]
where $\sigma_{0}>0$ is a small fixed constant. Hence  we obtain
\begin{equation*}
\bigg(
\displaystyle\sum_{i=1}^{k}
\frac{\lambda_{i}^{\frac{N-2s}{2}}}
{\big(1+\lambda_{i}|x-\xi_{i}|\big)
^{\frac{N-2s}{2}+\sigma}}\bigg)^{-1}
|\omega_\epsilon(x)|
\leq  C\|\omega_\epsilon\|_*
\displaystyle\sum_{i=1}^{k}\frac{1}
{(1+\lambda_{i}|x-\xi_{i}|)
^{2s-\sigma-\sigma_{0}}}.
\end{equation*}
It follows from Lemmas \ref{lem4.2}, \ref{lem4.4} and \ref{lem4.04}  that for any $R>0$ and $i=1,\cdots,k$,
\begin{equation*}
\lambda_{i}^{-\frac{N-2s}{2}}\omega_\epsilon
\rightarrow0\ \ 
\mbox{in}\ \  B_{R\lambda_{i}^{-1}}
(\xi_{i}) \ \ \text{as}\;\;\epsilon\rightarrow0.
\end{equation*}
Since  $\|\omega_\epsilon\|_*=1$, 
$\bigg(
\displaystyle\sum_{i=1}^{k}
\frac{\lambda_{i}^{\frac{N-2s}{2}}}
{\big(1+\lambda_{i}|x-\xi_{i}|\big)
^{\frac{N-2s}{2}+\sigma}}\bigg)^{-1}
|\omega_\epsilon(x)|$
cannot attain its maximum in $B_{R\lambda_{i}^{-1}}(\xi_{i})$. Choosing $R$ large enough such that
\[
\frac{C}{(1+R
)^{2s-\sigma-\sigma_{0}}}\leq\frac{1}{2k},
\]
we deduce that
\begin{equation*}
\bigg(
\displaystyle\sum_{i=1}^{k}
\frac{\lambda_{i}^{\frac{N-2s}{2}}}
{\big(1+\lambda_{i}|x-\xi_{i}|\big)
^{\frac{N-2s}{2}+\sigma}}\bigg)^{-1}
|\omega_\epsilon(x)|
\leq\frac{1}{2}
\|\omega_\epsilon\|_*,\ \ x \in \R^N\setminus \cup_{i=1}^k B_{R\lambda_{i}^{-1}}
(\xi_{i})
.
\end{equation*}
This implies that for $\epsilon$ sufficiently small, 
\[
\|\omega_\epsilon\|_*\leq\frac{1}{2}
\|\omega_\epsilon\|_*,
\]
which contradicts the fact that  $\|\omega_\epsilon\|_*=1$. 
This completes the proof of Theorem \ref{NDth}.
\end{proof}  
\section{Local uniqueness of multi-peak solutions}

In this section, we investigate  the local uniqueness of multi-peak solutions to \eqref{eq:1}. 
 To this end, let $u_\epsilon^{(1)}$ and $u_\epsilon^{(2)}$ be two peak  solutions constructed in Theorem \ref{Mth}.
  We aim to prove that there exists $\epsilon_0>0$ such that $u_\epsilon^{(1)}\equiv u_\epsilon^{(2)}$ for $\epsilon\in(0,\epsilon_0)$.
We proceed by contradiction and assume that $u_\epsilon^{(1)}\neq u_\epsilon^{(2)}$. 
Define \[
\psi_\epsilon=\frac{u_\epsilon^{(1)}-u_\epsilon^{(2)}}{\big\|u_\epsilon^{(1)}-u_\epsilon^{(2)}\big\|_*}.
\]
Then  $\psi_\epsilon$ satisfies $\|\psi_\epsilon\|_*=1$ and 
\begin{equation*}
(-\Delta)^{s}\psi_\epsilon+V(x)\psi_\epsilon
=(p_{s}-\epsilon) c_\epsilon(x)\psi_\epsilon,
\end{equation*}
where 
\[
c_\epsilon(x)=\int_0^1\left(tu_\epsilon^{(1)}+(1-t)u_\epsilon^{(2)}\right)
^{p_{s}-1-\epsilon}\mathrm{d}t.
\]

\subsection{Quantitative estimates of multi-peak solutions via local  Pohozaev identities}

In this subsection, we employ several local Pohozaev identities to derive quantitative estimates for  multi-peak solutions,
which play an essential role in establishing the local uniqueness of such solutions.

Let $\widetilde{u}_\epsilon^{(m)}$ denote the $s$-harmonic extension of $u_\epsilon^{(m)}$, $m=1,2.$ Multiplying  \eqref{eq-ND1} respectively  by $\widetilde{u}_\epsilon^{(m)}$ and  $\langle  X-(\xi_{i}^{(1)},0),\nabla \widetilde{u}_\epsilon^{(m)}\rangle$, we obtain  the following local  Pohozaev identity for dilations.

\begin{lemma}\label{lem5.1}
Assume  that $V(x)\in C^2\big(B_\rho(\xi_{i}^{(1)})\big)$, 
$ i=1,\cdots,k$. Then  for $m=1,2$ and $\rho>0$, we have
\begin{equation}\label{5.2}
\begin{split}
&\left(\frac{N}{p_{s}+1-\epsilon}-
\frac{N-2s}{2}\right)
\int_{B_\rho(\xi_{i}^{(1)})} \big(u_\epsilon^{(m)}\big)^{p_{s}
+1-\epsilon}
-\frac{1}{2}\int_{B_\rho(\xi_{i}^{(1)})}\big\langle x-\xi_{i}^{(1)},\nabla V(x)\big\rangle\big(u_\epsilon^{(m)}\big)^{2}
\\&-s\int_{B_\rho(\xi_{i}^{(1)})}V(x)
\big(u_\epsilon^{(m)}\big)^{2}
=\frac{1}{p_{s}+1-\epsilon}\int_{\partial B_\rho(\xi_{i}^{(1)})} \big(u_\epsilon^{(m)}\big)^{p_{s}+1-\epsilon}\langle x-\xi_{i}^{(1)},\nu\rangle\\
&
-\frac{1}{2}\int_{\partial B_\rho(\xi_{i}^{(1)})}V(x)\big(u_\epsilon^{(m)}
\big)^2 \langle x-\xi_{i}^{(1)},\nu\rangle+\int_{\partial^{''} \mathfrak{B}^{+}_\rho(\xi_{i}^{(1)})}
t^{1-2s}\frac{\partial \widetilde{u}_\epsilon^{(m)}}{\partial\nu}\big\langle X-(\xi_{i}^{(1)},0),\nabla \widetilde{u}_\epsilon^{(m)}\big\rangle
\\
&-\frac{1}{2}\int_{\partial^{''} \mathfrak{B}^{+}_\rho(\xi_{i}^{(1)})}
t^{1-2s}|\nabla \widetilde{u}_\epsilon^{(m)}|^2 \big\langle X-(\xi_{i}^{(1)},0),\nu\big\rangle
+\frac{N-2s}{2}\int_{\partial^{''} \mathfrak{B}^{+}_\rho(\xi_{i}^{(1)})}
t^{1-2s}\frac{\partial \widetilde{u}^{(m)}_\epsilon}{\partial\nu}\widetilde{u}_\epsilon^{(m)}.
\end{split}
\end{equation}
\end{lemma}

Similarly, multiplying \eqref{eq-ND1} by $\frac{\partial \widetilde{u}_\epsilon^{(m)}}{\partial x_{l}}$
yields the following local Pohozaev identity for translations.
\begin{lemma}\label{lem14.30}
Under the same assumptions as in Lemma \ref{lem5.1},  for $m=1,2$, $\rho>0$ and $l=1,\cdots,N$, we have
\begin{equation}\label{14.3}
\begin{split}
&\frac{1}{2}\int_{B_\rho(\xi_{i}^{(1)})}\frac{\partial V(x)}{\partial x_{l}}\big(u_\epsilon^{(m)}\big)^{2}
=\frac{1}{2}\int_{\partial^{''} \mathfrak{B}^{+}_\rho(\xi_{i}^{(1)})}
t^{1-2s}|\nabla \widetilde{u}_\epsilon^{(m)}|^{2}\nu_{l}
+\frac{1}{2}\int_{\partial B_\rho(\xi_{i}^{(1)})}V(x)\big(u_\epsilon^{(m)}\big)^{2}
\nu_{l}
\\&
-\int_{\partial^{''} \mathfrak{B}^{+}_\rho(\xi_{i}^{(1)})}
t^{1-2s}\frac{\partial \widetilde{u}_\epsilon^{(m)}}{\partial x_{l}}\frac{\partial \widetilde{u}_\epsilon^{(m)}}{\partial\nu}
-\frac{1}{p_{s}+1-\epsilon}\int_{\partial B_\rho(\xi_{i}^{(1)})} \big( u_\epsilon^{(m)}\big)^{p_{s}+1-\epsilon} \nu_{l}.
\end{split}
\end{equation}

\end{lemma}


As a direct consequence of Lemmas \ref{lem5.1} and \ref{lem14.30}, we readily  deduce the  following local Pohozaev identities for $\psi_\epsilon$.

\begin{lemma}\label{lem5.2}
Under the same assumptions as in Lemma \ref{lem5.1}, for  $\rho>0$, it holds that
\begin{equation}\label{5.3}
\begin{split}
&\bigg(\frac{N}{p_{s}+1-\epsilon}
-\frac{N-2s}{2}
\bigg)
\int_{B_\rho(\xi_{i}^{(1)})} b_\epsilon(x)\psi_\epsilon
-\frac{1}{2}\int_{B_\rho(\xi_{i}^{(1)})}\big\langle x-\xi_{i}^{(1)},\nabla V(x)\big\rangle\big(u_\epsilon^{(1)}
+u_\epsilon^{(2)}\big)\psi_\epsilon
\\
-&s\int_{B_\rho(\xi_{i}^{(1)})}V(x)
\big(u_\epsilon^{(1)}+u_\epsilon^{(2)}
\big)\psi_\epsilon
=
-\frac{1}{2}\int_{\partial B_\rho(\xi_{i}^{(1)})}V(x)\big(u_\epsilon^{(1)}
+u_\epsilon^{(2)}\big)\psi_\epsilon \langle x-\xi_{i}^{(1)},\nu\rangle\\
+&\frac{1}{p_{s}+1-\epsilon}\int_{\partial B_\rho(\xi_{i}^{(1)})} b_\epsilon(x)\psi_\epsilon\langle x-\xi_{i}^{(1)},\nu\rangle
+\int_{\partial^{''} \mathfrak{B}^{+}_\rho(\xi_{i}^{(1)})}
t^{1-2s}\frac{\partial \widetilde{u}_\epsilon^{(2)}}{\partial\nu}
\big\langle X-(\xi_{i}^{(1)},0),
\nabla\widetilde{\psi}_\epsilon\big\rangle
\\
-&\frac{1}{2}\int_{\partial^{''} \mathfrak{B}^{+}_\rho(\xi_{i}^{(1)})}
t^{1-2s}\nabla\big(\widetilde{u}_\epsilon^{(1)}
+\widetilde{u}_\epsilon^{(2)}\big)
\nabla\widetilde{\psi}_\epsilon \langle X-(\xi_{i}^{(1)},0),\nu\rangle
+\frac{N-2s}{2}\int_{\partial^{''} \mathfrak{B}^{+}_\rho(\xi_{i}^{(1)})}
t^{1-2s}\frac{\partial
\widetilde{\psi}_\epsilon}{\partial\nu}
\widetilde{u}_\epsilon^{(1)}
\\
+&\frac{N-2s}{2}\int_{\partial^{''} \mathfrak{B}^{+}_\rho(\xi_{i}^{(1)})}
t^{1-2s}\frac{\partial \widetilde{u}_\epsilon^{(2)}}{\partial\nu}
\widetilde{\psi}_\epsilon
+
\int_{\partial ^{''}\mathfrak{B}^{+}_\rho(\xi_{i}^{(1)})}
t^{1-2s}\frac{\partial\widetilde{\psi}_\epsilon}
{\partial\nu}\big\langle X-(\xi_{i}^{(1)},0),\nabla \widetilde{u}_\epsilon^{(1)}\big\rangle
,
\end{split}
\end{equation}
where
\[
b_\epsilon(x)=(p_{s}+1-\epsilon)\int_0^1
\Big(tu_\epsilon^{(1)}+(1-t)u_\epsilon^{(2)}
\Big)^{p_{s}-\epsilon}\mathrm{d}t.
\]
\end{lemma}

\begin{lemma}\label{lem14.33}
Under the same assumptions as in Lemma \ref{lem5.1}, for  $\rho>0$ and $l=1,\cdots,N$, we have
\begin{equation}\label{14.73}
\begin{split}
&\frac{1}{2}\int_{B_\rho(\xi_{i}^{(1)})}\frac{\partial V(x)}{\partial x_{l}}\big(u_\epsilon^{(1)}
+u_\epsilon^{(2)}\big)
\psi_\epsilon
=\frac{1}{2}\int_{\partial^{''} \mathfrak{B}^{+}_\rho(\xi_{i}^{(1)})}
t^{1-2s}\nabla\big(\widetilde{u}_\epsilon^{(1)}
+\widetilde{u}_\epsilon^{(2)}\big)
\nabla\widetilde{\psi}_\epsilon \nu_{l}\\
&
-\int_{\partial ^{''} \mathfrak{B}^{+}_\rho(\xi_{i}^{(1)})}
t^{1-2s}\frac{\partial \widetilde{u}_\epsilon^{(1)}}{\partial x_{l}}\frac{\partial \widetilde{\psi}_\epsilon}{\partial\nu}
-\int_{\partial^{''} \mathfrak{B}^{+}_\rho(\xi_{i}^{(1)})}
t^{1-2s}\frac{\partial \widetilde{\psi}_\epsilon
}{\partial x_{l}}\frac{\partial \widetilde{u}_\epsilon^{(2)} }{\partial\nu}
\\
&+\frac{1}{2}\int_{\partial B_\rho(\xi_{i}^{(1)})}V(x)\big(u_\epsilon^{(1)}
+u_\epsilon^{(2)}\big)
\psi_\epsilon
\nu_{l}
-\frac{1}{p_{s}+1-\epsilon
}\int_{\partial B_\rho(\xi_{i}^{(1)})} b_\epsilon(x)\psi_\epsilon \nu_{l}.
\end{split}
\end{equation}
\end{lemma}

Next, we exploit the above local Pohozaev identities to derive quantitative estimates for the concentration parameters  $\lambda_{i}^{(m)}$, where   $m=1,2$ and  $i=1,\cdots,k$.

\begin{lemma}\label{lem5.3}
Assume 
$\xi^*_i$, $i=1,2,\cdots,k$ are the $k$ different   non-degenerate critical points  of $V(x)$ with $V(\xi^*_i)>0$ and $V(x)\in C^2(B_{5\delta}(\xi^*_i))$.
 Then for $m=1,2$, we have
\begin{equation*}
\lambda_{i}^{(m)}=\mu_i\epsilon^{
-\frac{1}{2s}}
+O\big(\epsilon^{1-\frac{1+\sigma}{2s}}\big),
\end{equation*}
where  $\mu_i$ is a positive constant, $i=1,\cdots,k$.
\end{lemma}
\begin{proof}For simplicity, we omit the superscript $(m)$.
By Lemma \ref{thm1.1}, we have
 \[
 \begin{split}
 \int_{B_\rho(\xi_{i})} u_\epsilon^{p_{s}+1-\epsilon}
 =&\int_{B_\delta(\xi_{i})} U_{\lambda_{i},\xi_{i}}
 ^{p_{s}+1-\epsilon}
 +O\bigg(\int_{B_\delta(\xi_{i})} U_{\lambda_{i},\xi_{i}}
 ^{p_{s}-\epsilon}|\phi_\epsilon|
 +\int_{B_\delta(\xi_{i})} |\phi_\epsilon|^{p_{s}+1-\epsilon}\bigg)+O\big(\lambda_i^{-N}\big)
 \\
 =&\gamma_{s,N}^{p_{s}+1-\epsilon}\lambda_{i}^{-\frac{N-2}{2}\epsilon}
 \int_{B_{\lambda_{i}\delta}(0)}\frac{1}{(1+|x|^{2})^{N-\frac{N-2}{2}\epsilon}}
 +O\big(\|\phi_\epsilon\|_*\big)+O\big(\|\phi_\epsilon\|_*^{p_{s}+1}\big)
 +O\big(\lambda_i^{-N}\big)\\
 =&\gamma_{s,N}^{p_{s}+1}
 \int_{\mathbb{R}^N}\frac{1}{(1+|x|^{2})
 ^{N}}
 +O\big(\epsilon^{1-\frac{\sigma}{2s}}\big).
 \end{split}
 \]
 Since
 $$
 \frac{1}{\lambda_{i}}\leq\frac{C}
 {(1+\lambda_{i}^{2}|x-\xi_{i}|^{2})^{\frac{1}{2}}} \;\;
 \text{in}\;\;B_\delta(\xi_{i})\;\;\text{for}\;\;i=1,\cdots,k,
 $$ we infer 
 \begin{align*}
 \int_{B_\rho(\xi_{i})}V(x)u_\epsilon^2
 =&
 \int_{B_\delta(\xi_{i})}V(x)
 U_{\lambda_{i},\xi_{i}}^2
 +2\int_{B_\delta(\xi_{i})}V(x)
 U_{\lambda_{i},\xi_{i}}\phi_\epsilon
 +\int_{B_\delta(\xi_{i})}V(x)\phi_{\epsilon}^2+O\big(\lambda_i^{-N+2s}\big)
 \\
 =&V(\xi_{i})\int_{B_\delta(\xi_{i})}U_{\lambda_{i},\xi_{i}}^2
 +O\big(\lambda_{i}^{-2s}\|\phi_\epsilon\|_*\big)
 +O\big(\lambda_{i}^{-2\sigma}\|\phi_\epsilon\|_*^2\big)+O\big(\lambda_i^{-N+2s}\big)\\
 =&\gamma_{s,N}^2V(\xi_{i})
 \lambda_{i}^{-2s}\int_{\mathbb{R}^N}
 \frac{1}{(1+|x|^{2})^{N-2s}}
+
O\big(\epsilon^{2-\frac{\sigma}{2s}}\big),
\\
 \end{align*}
 and
 \begin{align*}
&\int_{B_\rho(\xi_{i})}\big\langle x-\xi_{i},\nabla V(x)\big\rangle u_\epsilon^2\\
 =&
\int_{B_\delta(\xi_{i})}\big\langle x-\xi_{i},\nabla V(x)\big\rangle U_{\lambda_{i},\xi_{i}}^2
 +2\int_{B_\delta(\xi_{i})}\big\langle x-\xi_{i},\nabla V(x)\big\rangle U_{\lambda_{i},\xi_{i}}\phi_\epsilon
 \\[0.02mm]
 &+\int_{B_\delta(\xi_{i})}\big\langle x-\xi_{i},\nabla V(x)\big\rangle\phi_{\epsilon}^2+O\big(\lambda_i^{-N+2s}\big)\\[0.02mm]
 =& O\big(\lambda_i^{-2-2s}\big) +O\big(\lambda_i^{-4s+\sigma}\big)+O\big(\lambda_{i}^{-2s}\|\phi_\epsilon\|_*\big)+O\big(\lambda_{i}^{-2\sigma}\|\phi_\epsilon\|_*^2\big)
 +O\big(\lambda_i^{-N+2s}\big)
 \\[0.02mm]
 =&O\big(\epsilon^{2-\frac{\sigma}{2s}}\big).
 \end{align*}
On the other hand, it follows from  Lemmas \ref{lem4.1}, \ref{lemGA.5} 
and \ref{newlemGA.6} 
that 
\begin{align*}
\text{RHS\;of}\;\eqref{5.2}=
O\big(\lambda_{i}^{2s-N}\big)
+O(\epsilon^{2})=
O(\epsilon^{2}).
\end{align*}
Hence, by \eqref{5.2}, we conclude 
\begin{align*}
\lambda_{i}=\mu_i\epsilon^{
-\frac{1}{2s}}
+O\big(\epsilon^{1-\frac{1+\sigma}{2s}}\big).
\end{align*}
This completes the proof of Lemma \ref{lem5.3}. 
\end{proof}

By virtue of Lemma \ref{lem5.3}, we are able to  compare $u_\epsilon^{(1)}$ and $u_\epsilon^{(2)}$.

\begin{lemma}\label{cor5.4}
Under the same assumptions as in Lemma \ref{lem5.3}, it holds that
\begin{align*}
  \big|u_\epsilon^{(1)}-u_\epsilon^{(2)}\big|=o(1)\displaystyle\sum_{i=1}^{k}\eta_{i}^{(1)}(x)U_{\lambda_{i}^{(1)},\xi_{i}^{(1)}
  }+
 O\Big(
 \sum_{i=1}^k  \lambda_i^{-\frac{N-2s}{2}}
        \chi_{B_{\frac{5\delta}{2}}(\xi_{i}^{(1)})\setminus B_{\frac{\delta}{2}}(\xi_{i}^{(1)})}
        \Big)
  +
 O\big( |\phi_{\epsilon}^{(1)}|
  +|\phi_{\epsilon}^{(2)}|\big).
\end{align*}
\end{lemma}
\begin{proof}
Since $\xi_i^*$, $i=1,2,\cdots,k$ are the non-degenerate critical points of $V(x)$. Then it follows from Lemma \ref{thm1.1} that 
$$|\xi_i^{(m)}-\xi_i^*|=o(\epsilon^{\frac{1}{2s}})\;\;\text{for}\;\; m=1,2 \;\;\text{and}\;\; i=1,\cdots,k.$$
Thanks to Lemma \ref{lem5.3},  we obtain 
    \begin{align*}
        \left|u_\epsilon^{(1)}-u_\epsilon^{(2)}\right|\leq& \displaystyle\sum_{i=1}^{k}\left|\eta_{i}^{(1)}(x) U_{\lambda_{i}^{(1)},\xi_{i}^{(1)}
  }-\eta_{i}^{(2)}(x) U_{\lambda_{i}^{(2)},\xi_{i}^{(2)}
  }\right|
  +
  |\phi_{\epsilon}^{(1)}|
  +|\phi_{\epsilon}^{(2)}|\\
  \leq& \displaystyle C\sum_{i=1}^{k} \eta_{i}^{(1)}(x) \left|\frac{\partial U_{\lambda_{i}^{(1)},\xi_{i}^{(1)}
  }}{\partial\lambda_i^{(1)}}\right|\left|\lambda_i^{(1)}-\lambda_i^{(2)}\right|+\displaystyle C\sum_{i=1}^{k}\eta_{i}^{(1)}(x) \left|\nabla U_{\lambda_{i}^{(1)},\xi_{i}^{(1)}
  }\right|\left|\xi_i^{(1)}-\xi_i^{(2)}\right|\\
  &+
 O\Big(
 \sum_{i=1}^k
        \chi_{B_{\frac{5\delta}{2}}(\xi_{i}^{(1)})\setminus B_{\frac{\delta}{2}}(\xi_{i}^{(1)})}\big(\lambda_{i}^{(1)}\big)^{-\frac{N-2s}{2}}
        \Big)
  +
  |\phi_{\epsilon}^{(1)}|
  +|\phi_{\epsilon}^{(2)}|\\
  =&o(1)\displaystyle\sum_{i=1}^{k}\eta_{i}^{(1)}(x) U_{\lambda_{i}^{(1)},\xi_{i}^{(1)}
  }+
 O\Big(
 \sum_{i=1}^k \big(\lambda_{i}^{(1)}\big)^{-\frac{N-2s}{2}}
        \chi_{B_{\frac{5\delta}{2}}(\xi_{i}^{(1)})\setminus B_{\frac{\delta}{2}}(\xi_{i}^{(1)})} 
        \Big)
  +
 O\big( |\phi_{\epsilon}^{(1)}|
  +|\phi_{\epsilon}^{(2)}|\big).
    \end{align*}
\end{proof}

\subsection{Lyapunov-Schmidt decomposition of $\psi_\epsilon$}

In this subsection, we carry out a Lyapunov–Schmidt type decomposition for  $\psi_\epsilon$. To achieve this goal, 
we define for $i=1,\cdots,k$, $l=1,\cdots,N$, 
\begin{equation}
Z_{\lambda_i^{(1)},\xi_i^{(1)}}^0=\frac{\partial U_{\lambda_i^{(1)}, \xi_i^{(1)}}}{\partial \lambda_i^{(1)}}, \ \ 
Z_{\lambda_i^{(1)},\xi_i^{(1)}}^l=\frac{\partial U_{\lambda_i^{(1)}, \xi_i^{(1)}}}{\partial \xi_{i,l}^{(1)}},
 \end{equation}
 \begin{equation}\label{newdefz}
Z_{i,0}^{(1)}=\frac{\partial W_{\lambda_i^{(1)}, \xi_i^{(1)}}}{\partial \lambda_i^{(1)}}, \ \   
Z_{i,l}^{(1)}=\frac{\partial W_{\lambda_i^{(1)}, \xi_i^{(1)}}}{\partial \xi_{i,l}^{(1)}}.
 \end{equation}
We decompose $\psi_\epsilon$ into
\begin{align}\label{eq-psi}
\begin{cases}
\psi_\epsilon(x)=\displaystyle\sum_{i=1}^{k}
d_{i,0}^{(1)}
\lambda_{i}^{(1)}Z_{i,0}^{(1)}
+\displaystyle\sum_{i=1}^{k}\displaystyle\sum_{l=1}
^{N} d_{i,l}^{(1)}
\big(\lambda_{i}^{(1)}\big)^{-1}
Z_{i,l}^{(1)}
+\psi_\epsilon^*(x),\\[5mm]
\displaystyle
\int_{\mathbb{R}^N} W_{\lambda_{i}^{(1)},\xi_{i}^{(1)}}^{p_s-1}Z_{i,l}^{(1)}
\psi_\epsilon^*(x)=0,\;i=1,2,\cdots,k,\ l=0,1,\cdots,N.
\end{cases}
\end{align}
 
To obtain finer estimates for $\psi_\epsilon$, we introduce the following weighted norm associated with the parameters
$\big(\lambda_i^{(1)},\xi_i^{(1)}\big)$:
     \begin{equation*}
\|\phi\|_*=\sup\limits_{x\in \mathbb{R}^N}\left(\displaystyle\sum_{i=1}^{k}
\frac{\big(\lambda_{i}^{(1)}\big)
^{\frac{N-2s}{2}}}
{(1+\lambda_{i}^{(1)}
\big|x-\xi_{i}^{(1)}\big |)
^{\frac{N-2s}{2}+\sigma}}\right)^{-1}|\phi(x)|,
\end{equation*}
$$
\|h\|_{**}=\sup\limits_{x\in \mathbb{R}^N}\left(\displaystyle\sum_{i=1}^{k}
\frac{\big(\lambda_{i}^{(1)}\big)
^{\frac{N+2s}{2}}}
{(1+\lambda_{i}^{(1)}
\big|x-\xi_{i}^{(1)}\big |)
^{\frac{N+2s}{2}+\sigma}}\right)^{-1}|h(x)|,
$$
where $s>\frac{1}{2}$, $\sigma\in\left(0,\min\big\{\frac{s}{2},2s-1\big\}\right)$.
Then we have  $$\big\|\phi_{\epsilon}^{(m)}\big\|_*
=O\big(\epsilon^{1-\frac{\sigma}{2s}}\big),\ m=1,2,  
$$
where we made use of the comparability of $U_1$ and $U_2$. More precisely, 
from the quantitative estimates of 
the  concentration parameters $(\lambda_i^{(m)},\xi_i^{(m)})$ with  $m=1,2$, we can  obtain 
$$
C^{-1} U_{\lambda^{(1)}_i,\xi^{(1)}_i}  \leq   U_{\lambda^{(2)}_i,\xi^{(2)}_i} \leq C U_{\lambda^{(1)}_i,\xi^{(1)}_i}.
$$

Moreover, it is easy to check  that the projected coefficients $d_{i,l}^{(1)}$ are bounded for $i=1,\cdots,k $, $l=0,\cdots,N$. 
With the aid of Lemma\,\ref{lemA}, we deduce that $\psi_\epsilon^*$ satisfies 
\begin{align*}
&(-\Delta )^{s} \psi^*_\epsilon+V(x)\psi^*_\epsilon
-(p_{s}-\epsilon) c_\epsilon(x)\psi^*_\epsilon
\\
=&(p_{s}-\epsilon)\displaystyle\sum_{i=1}^{k}
d_{i,0}^{(1)}\eta_{i}^{(1)}
\lambda_{i}^{(1)}Z_{\lambda_i^{(1)},\xi_i^{(1)}}^0
\left(c_\epsilon(x)- \big(u^{(1)}_\epsilon\big)
^{p_{s}-1-\epsilon}\right)
\\
&+\displaystyle\sum_{i=1}^{k}
d_{i,0}^{(1)}\eta_{i}^{(1)}
\lambda_{i}^{(1)}Z_{\lambda_i^{(1)},\xi_i^{(1)}}^0
\left((p_{s}-\epsilon) \big(u^{(1)}_\epsilon\big)^{p_{s}-1
-\epsilon}-p_{s} \big(u^{(1)}_\epsilon\big)^{p_{s}-1}\right)
\\
&+p_{s}\displaystyle\sum_{i=1}^{k}
d_{i,0}^{(1)}\eta_{i}^{(1)}
\lambda_{i}^{(1)}Z_{\lambda_i^{(1)},\xi_i^{(1)}}^0
\Big(\big(u^{(1)}_\epsilon\big)^{p_{s}-1}- U_{\lambda_{i}^{(1)},\xi_{i}^{(1)}}
^{p_{s}-1}\Big)
-V(x)\displaystyle\sum_{i=1}^{k}
d_{i,0}^{(1)}\lambda_{i}^{(1)}
\eta_{i}^{(1)}
Z_{\lambda_i^{(1)},\xi_i^{(1)}}^0
\\
&-\displaystyle\sum_{i=1}^{k}
 d_{i,0}^{(1)}
\lambda_{i}^{(1)}J_{i,0}^{(1)}
+(p_{s}-\epsilon)
\displaystyle\sum_{i=1}^{k}\displaystyle\sum_{l=1}
^{N}d_{i,l}^{(1)}\eta_{i}^{(1)}
\big(\lambda_{i}^{(1)}\big)^{-1}Z_{\lambda_i^{(1)},\xi_i^{(1)}}^l
\Big(c_\epsilon(x)- \big(u^{(1)}_\epsilon\big)^{p_{s}-1-\epsilon}
\Big)
\\
&+
\displaystyle\sum_{i=1}^{k}\displaystyle\sum_{l=1}
^{N}d_{i,l}^{(1)}\eta_{i}^{(1)}
\big(\lambda_{i}^{(1)}\big)^{-1}
Z_{\lambda_i^{(1)},\xi_i^{(1)}}^l 
\left((p_{s}-\epsilon) \big(u^{(1)}_\epsilon\big)^{p_{s}-1
-\epsilon}-p_{s}\big(u^{(1)}_\epsilon\big)^{p_{s}-1}\right)
\\
&
+p_{s}
\displaystyle\sum_{i=1}^{k}\displaystyle\sum_{l=1}
^{N}d_{i,l}^{(1)}\eta_{i}^{(1)}
\big(\lambda_{i}^{(1)}\big)^{-1}
Z_{\lambda_i^{(1)},\xi_i^{(1)}}^l
\Big(\big(u^{(1)}_\epsilon\big)^{p_{s}-1}- U_{\lambda_{i}^{(1)}, \xi_{i}^{(1)}}
^{p_{s}-1}\Big)
\\
&-V(x)\displaystyle\sum_{i=1}^{k}
\displaystyle\sum_{l=1}^{N}
d_{i,l}^{(1)}
\big(\lambda_{i}^{(1)}\big)^{-1}
\eta_{i}^{(1)}Z_{\lambda_i^{(1)},\xi_i^{(1)}}^l-\displaystyle\sum_{i=1}^{k}
\displaystyle\sum_{l=1}^{N}
 d_{i,l}^{(1)}
\big(\lambda_{i}^{(1)}\big)^{-1} J_{i,l}^{(1)}
\\
&+(p_s-\epsilon)c_\epsilon(x)\displaystyle\sum_{i=1}^{k}\displaystyle\sum_{l=1}
^{N}d_{i,l}^{(1)}\big(\lambda_{i}^{(1)}\big)^{-1}
U_{\lambda_{i}^{(1)}, \xi_{i}^{(1)}}\frac{\partial\eta_{i}^{(1)} }{\partial \xi_{i,l}^{(1)}}
-\displaystyle\sum_{i=1}^{k}\displaystyle\sum_{l=1}
^{N}d_{i,l}^{(1)}\big(\lambda_{i}^{(1)}\big)^{-1}
U^{p_s}_{\lambda_{i}^{(1)},\xi_{i}^{(1)}}
\frac{\partial \eta_{i}^{(1)} }{\partial \xi_{i,l}^{(1)}}
\\
&-\displaystyle \sum_{i=1}^{k}
\displaystyle\sum_{l=1}^{N}
d_{i,l}^{(1)}
\big(\lambda_{i}^{(1)}\big)^{-1} \Lambda_{i,l}^{(1)}-V(x)\displaystyle\sum_{i=1}^{k}\displaystyle\sum_{l=1}
^{N}d_{i,l}^{(1)}\big(\lambda_{i}^{(1)}\big)^{-1}
U_{\lambda_{i}^{(1)}, \xi_{i}^{(1)}}\frac{\partial \eta_{i}^{(1)}}{\partial \xi_{i,l}^{(1)}}
\\
:=& L_1+\cdots+L_5+P_1+\cdots+P_5+Q_1+\cdots+Q_4,
\end{align*}
where
\[
J_{i,0}^{(1)}=C(N,s)\lim_{\varepsilon\rightarrow
0^{+}}\int_{\mathbb{R}^N\backslash B_\varepsilon(x)}\frac{\eta_{i}^{(1)}(x)-
\eta_{i}^{(1)}(y)}{|x-y|^{N+2s}}Z_{\lambda_i^{(1)},\xi_i^{(1)}}^0(y),
\]
\[
J_{i,l}^{(1)}=C(N,s)\displaystyle\lim_{\varepsilon\rightarrow
0^{+}}\displaystyle\int_{\mathbb{R}^N\backslash B_\varepsilon(x)}\frac{\eta_{i}^{(1)}(x)-
\eta_{i}^{(1)}(y)}{|x-y|^{N+2s}}Z_{\lambda_i^{(1)},\xi_i^{(1)}}^l(y),
\]
and 
\[
\Lambda_{i,l}^{(1)}=C(N,s)\displaystyle\lim_{\varepsilon\rightarrow
0^{+}}\int_{\mathbb{R}^N\backslash B_\varepsilon(x)} 
\frac{\frac{\partial \eta_{i}^{(1)} }{\partial \xi_{i,l}^{(1)}}(x)-
\frac{\partial \eta_{i}^{(1)} }{\partial \xi_{i,l}^{(1)}}(y)}{|x-y|^{N+2s}}U_{\lambda_{i}^{(1)}, \xi_{i}^{(1)}}(y).
\]
We start by estimating the projected remainder term  $\psi_*.$
\begin{lemma}\label{lem5.5}
Assume that $\xi^*_i$, $i=1,2,\cdots,k$ are the $k$ different   non-degenerate critical points  of $V(x)$ with $V(\xi^*_i)>0$ and $V(x)\in C^2(B_{5\delta}(\xi^*_i))$. Then 
\[
\|\psi_\epsilon^*\|_*=o(1).
\]
\end{lemma}
\begin{proof}[\bf Proof.]
Since $0<p_{s}-1-\epsilon\leq1$,  then it follows from  Lemmas \ref{lemL2.8} and  \ref{cor5.4} that 
\[
\begin{split}
|L_1|\leq& C\displaystyle\sum_{i=1}^{k}
\eta_{i}^{(1)}U_{\lambda_{i}^{(1)}, \xi_{i}^{(1)}}
\big|c_\epsilon(x)- \big(u^{(1)}_\epsilon\big)^{p_{s}
-1-\epsilon}\big|
\leq C\displaystyle\sum_{i=1}^{k}
\eta_{i}^{(1)}U_{\lambda_{i}^{(1)}, \xi_{i}^{(1)}}
\big|u^{(1)}_{\epsilon}-u^{(2)}_{\epsilon}\big|^{p_{s}
-1-\epsilon}\\
= &o\Big(\sum_{i=1}^{k}
\eta_{i}^{(1)}U_{\lambda_{i}^{(1)}, \xi_{i}^{(1)}}\Big)^{p_{s}
-\epsilon}
+O\Big(
 \sum_{i=1}^k (\lambda_i^{(1)})^{-\frac{N+2s}{2}}
        \chi_{B_{\frac{5\delta}{2}}(\xi_{i}^{(1)})\setminus B_{\frac{\delta}{2}}(\xi_{i}^{(1)})}
        \Big)\\
&+O\Big( \displaystyle\sum_{i=1}^{k}
\eta_{i}^{(1)}U_{\lambda_{i}^{(1)}, \xi_{i}^{(1)}}
\big(|\phi_{\epsilon}^{(1)}|^{p_{s}-1-\epsilon}
+|\phi_{\epsilon}^{(2)}|^{p_{s}-1-\epsilon}\big)\Big)\\ 
=&o(1)
\sum_{i=1}^{k}
\frac{\big(\lambda_{i}^{(1)}\big)
^{\frac{N+2s}{2}-\frac{N-2s}{2}\epsilon}}
{(1+\lambda_{i}^{(1)}
\big|x-\xi_{i}^{(1)}\big |)
^{N+2s-(N-2s)\epsilon}}+O\Big(
 \sum_{i=1}^k (\lambda_i^{(1)})^{-\frac{N+2s}{2}}
        \chi_{B_{\frac{5\delta}{2}}(\xi_{i}^{(1)})\setminus B_{\frac{\delta}{2}}(\xi_{i}^{(1)})}
        \Big)
\\
&+O\left( \|\phi_{\epsilon}^{(1)}\|^{p_{s}-1-\epsilon}_*
+\|\phi_{\epsilon}^{(2)}\|^{p_{s}-1-\epsilon}_*\right)
\sum_{i=1}^{k}
\frac{\big(\lambda_{i}^{(1)}\big)
^{\frac{N+2s}{2}-\frac{N-2s}{2}\epsilon}}
{(1+\lambda_{i}^{(1)}
\big|x-\xi_{i}^{(1)}\big |)
^{N+(p_s-1-\epsilon)\sigma-\frac{N-2s}{2}\epsilon}}
\\
=&o(1)\displaystyle\sum_{i=1}^{k}
\frac{\big(\lambda_{i}^{(1)}\big)
^{\frac{N+2s}{2}}}
{(1+\lambda_{i}^{(1)}
\big|x-\xi_{i}^{(1)}\big |)
^{\frac{N+2s}{2}+\sigma}},
\end{split}
\]
which  implies that \[
\|L_1\|_{**}=o(1).
\]
By the mean value theorem, we have 
\[
\begin{split}
|L_2|\leq& C\epsilon\displaystyle\sum_{i=1}^{k}\eta_{i}^{(1)} U_{\lambda_{i}^{(1)},\xi_{i}^{(1)}}
\big(u_\epsilon^{(1)}\big)^{p_{s}-1-\theta\epsilon}\big(1+|\log u_\epsilon^{(1)}|\big)\\
\leq&C\epsilon|\log\epsilon|\sum_{i=1}^{k}\eta_{i}^{(1)} 
\frac{\big(\lambda_{i}^{(1)}\big)
^{\frac{N+2s}{2}-\frac{N-2s}{2}\theta\epsilon}}
{(1+\lambda_{i}^{(1)}
\big|x-\xi_{i}^{(1)}\big |)
^{N+2s-(N-2s)\theta\epsilon}}
\left(1+\log(1+\big(\lambda_{i}^{(1)}\big)^{2}
\big|x-\xi_{i}^{(1)}\big |^{2})\right)\\
\leq& C\epsilon|\log\epsilon|
\displaystyle\sum_{i=1}^{k}
\frac{\big(\lambda_{i}^{(1)}\big)
^{\frac{N+2s}{2}}}
{(1+\lambda_{i}^{(1)}
\big|x-\xi_{i}^{(1)}\big |)
^{\frac{N+2s}{2}+\sigma}},\;\;\text{where}\;\;\theta\in(0,1),
\end{split}
\]
which gives
\[
\|L_2\|_{**}\leq C\epsilon|\log\epsilon|.
\]
Next we estimate $\|L_3\|_{**}$.
Since $p_{s}\leq2$, we infer 
\[
\begin{split}
|L_3|\leq& C\displaystyle\sum_{i=1}^{k}\eta_{i}^{(1)} U_{\lambda_{i}^{(1)},\xi_{i}^{(1)}}
|\phi_{\epsilon}^{(1)}|^{p_{s}-1}\\
\leq&C\|\phi_{\epsilon}^{(1)}\|_*^{p_{s}-1}\displaystyle\sum_{i=1}^{k}
\frac{\big(\lambda_{i}^{(1)}\big)
^{\frac{N+2s}{2}}}
{(1+\lambda_{i}^{(1)}
\big|x-\xi_{i}^{(1)}\big |)
^{\frac{N+2s}{2}+\sigma}}.
\end{split}
\]
Thus we obtain 
\[
\|L_3\|_{**}\leq C\|\phi_{\epsilon}^{(1)}
\|_*^{p_{s}-1}.
\]
Moreover, we have 
\[
\begin{split}
|L_4|\leq&
C
\displaystyle\sum_{i=1}^{k}\eta_{i}^{(1)} 
\frac{\big(\lambda_{i}^{(1)}
\big)^{\frac{N-2s}{2}}}
{(1+\big(\lambda_{i}^{(1)}
\big)^{2}
\big|x-\xi_{i}^{(1)}\big |^{2})
^{\frac{N-2s}{2}}}
\\
\leq& C\displaystyle\sum_{i=1}^{k}\eta_{i}^{(1)} \frac{1}
{\big(\lambda_{i}^{(1)}\big)^{2s}}
\frac{\big(\lambda_{i}^{(1)}\big)
^{\frac{N+2s}{2}}}
{(1+\lambda_{i}^{(1)}
\big|x-\xi_{i}^{(1)}\big |)
^{N-2s}}\\
\leq&C\displaystyle\sum_{i=1}^{k}\frac{1}
{\big(\lambda_{i}^{(1)}\big)^{2s
-\sigma}}\frac{\big(\lambda_{i}^{(1)}\big)
^{\frac{N+2s}{2}}}
{(1+\lambda_{i}^{(1)}
\big|x-\xi_{i}^{(1)}\big |)
^{\frac{N+2s}{2}+\sigma}}
\end{split}
\]
and
\begin{equation*}
    |L_5|\leq C\displaystyle\sum_{i=1}^{k}\frac{(\lambda_i^{(1)})^{-\frac{N-2s}{2}}}{(1+|x-\xi_i^{(1)}|)^{N+2s}}\leq
    C \sum_{i=1}^{k}(\lambda_i^{(1)})^{-\frac{N-2s}{2}+\sigma}\frac{\big(\lambda_{i}^{(1)}\big)
^{\frac{N+2s}{2}}}
{(1+\lambda_{i}^{(1)}
\big|x-\xi_{i}^{(1)}\big |)
^{\frac{N+2s}{2}+\sigma}}.
\end{equation*}
Then we derive 
\[
\|L_4\|_{**}+\|L_5\|_{**}\leq C\epsilon^{1-\frac{\sigma}{2s}}.
\]

By analogous arguments as above, we derive 
\[
\begin{split}
\|P_1\|_{**}=o(1),\quad\|P_2\|_{**}\leq C\epsilon|\log\epsilon|,\quad\|P_3\|_{**}\leq C\big\|\phi_{\epsilon}^{(1)}\big\|_*^{p_{s}-1}
\end{split}
\]
and
\[
\|P_4\|_{**}+\|P_5\|_{**}+\sum_{j=1}^4\|Q_j\|_{**}\leq C\epsilon^{1-\frac{\sigma}{2s}}.
\]
It follows from  the same line of reasoning as Proposition \ref{prop2.3} that 
there exists $C>0$ such that \[
\|\psi_\epsilon^*\|_*\leq C\bigg(\sum^5_{i=1}\|L_i\|_{**}+\sum^5_{i=1}
\|P_i\|_{**}+\sum^4_{i=1}
\|Q_i\|_{**}\bigg)=o(1).
\]
This completes the proof of Lemma \ref{lem5.5}.
\end{proof}

From Lemmas \ref{lem4.1} and \ref{lem5.5}, we can establish the following finer decay estimate of $\psi_\epsilon^*$.
\begin{lemma}\label{lem-psi}
Under the same assumptions as in Lemma \ref{lem5.5}, we have 
    \begin{equation*}
        |\psi_\epsilon^*|=o(1)\sum_{i=1}^k\frac{(\lambda_i^{(1)})^{\frac{N-2s}{2}}}{(1+\lambda_i^{(1)}|x-\xi_i^{(1)}|)^{\frac{N-2s}{2}+1+\sigma}}.
    \end{equation*}
\end{lemma}
\begin{proof}
From the proof of Lemma \ref{lem5.5}, it follows that
\begin{equation}\label{eq_LP}
       \sum_{i=1}^5L_i(y)+\sum_{i=1}^5P_i(y)+\sum_{i=1}^4Q_i(y)=o(1)\displaystyle\sum_{i=1}^{k}
       \frac{\big(\lambda_{i}^{(1)}\big)^{\frac{N+2s}{2}}}{\big(1+\lambda_{i}^{(1)}\big|x-\xi_{i}^{(1)}\big|\big)^{\frac{N+2s}{2}+1+\sigma}}.
    \end{equation}
Applying Green's representation formula, we have
\begin{equation*}
    \psi_\epsilon^*(x)=\int_{\R^N}G(x,y)\Big((p_{s}-\epsilon) c_\epsilon(y)\psi^*_\epsilon(y)+\sum_{i=1}^5L_i(y)+\sum_{i=1}^5P_i(y)+\sum_{i=1}^4Q_i(y)\Big).
\end{equation*}
By Lemma \ref{lemLL2.3},  we deduce 
    \begin{align*}
        |\psi_\epsilon^*|
        =&o(1)\int_{\R^N}\frac{1}{|x-y|^{N-2s}}\sum_{i=1}^k\frac{\big(\lambda_i^{(1)}\big)^{\frac{N+2s}{2}-\frac{N-2s}{2}\epsilon}}{\big(1+\lambda_i^{(1)}|y-\xi_i^{(1)}|\big)^{\frac{N+6s}{2}+\sigma-(N-2s)\epsilon}}\\
        &+o(1)\int_{\R^N}\frac{1}{|x-y|^{N-2s}}\sum_{i=1}^k\frac{\big(\lambda_i^{(1)}\big)^{\frac{N+2s}{2}}}{\big(1+\lambda_i^{(1)}|y-\xi_i^{(1)}|\big)^{\frac{N+2s}{2}+1+\sigma}}\\
       =& o(1)\sum_{i=1}^{k}\frac{\big(\lambda_{i}^{(1)}\big)^{\frac{N-2s}{2}}}{\big(1+\lambda_{i}^{(1)}\big|x-\xi_{i}^{(1)}\big |\big)^{\frac{N-2s}{2}+1+\sigma}}.
    \end{align*}
The proof of Lemma \ref{lem-psi} is concluded.
\end{proof}

\subsection{ The estimates of the projected coefficients}
This subsection aims to establish estimates for the projected coefficients  $d_{i,l}^{(1)}$ with  $i=1,\cdots,k$ and $l=0,1,\cdots,N$ 
by means of the local Pohozaev identities \eqref{5.3} and \eqref{14.73}. For this purpose, we first need to estimate the $s$-harmonic extension $\widetilde{\psi}^*_\epsilon$ of $\psi_\epsilon^*$. The relevant estimates of  $\widetilde{\psi}^*_\epsilon$ are provided in Appendix  $B$.

\begin{lemma}\label{lem5.6}
Assume that $\xi^*_i$, $i=1,2,\cdots,k$ are the $k$ different   non-degenerate critical points  of $V(x)$ with $V(\xi^*_i)>0$ and $V(x)\in C^2(B_{5\delta}(\xi^*_i))$. Then 
\begin{equation}\label{tta}
\big|d_{i,0}^{(1)}\big|=o(1)\;\;\text{for}\;\;i=1,\cdots,k.
\end{equation}
\end{lemma}
\begin{proof}[\bf Proof.]
Note that \[\begin{split}
b_\epsilon(x)=(p_{s}+1-\epsilon)\big(u_\epsilon^{(1)}\big)
^{p_{s}-\epsilon}+O\left((u_\epsilon^{(1)})
^{p_{s}-1-\epsilon}|u_\epsilon^{(1)}
-u_\epsilon^{(2)}|+|u_\epsilon^{(1)}
-u_\epsilon^{(2)}|^{p_{s}-\epsilon}\right).
\end{split}
\]
Then we have 
\[
\begin{split}
\int_{B_\rho(\xi_{i}^{(1)})}  b_\epsilon(x)\psi_\epsilon
=&(p_{s}+1-\epsilon)\int_{B_\rho(\xi_{i}^{(1)})} \big(u_\epsilon^{(1)}\big)^{p_{s}-\epsilon}
\psi_\epsilon
+O\bigg(\int_{B_\rho(\xi_{i}^{(1)})} \big(u_\epsilon^{(1)}\big)^{p_{s}-1-\epsilon}
|u_\epsilon^{(1)}-u_\epsilon^{(2)}|
|\psi_\epsilon|\bigg)\\
&+O\bigg(\int_{B_\rho(\xi_{i}^{(1)})}|u_\epsilon^{(1)}
-u_\epsilon^{(2)}|^{p_{s}-\epsilon}
|\psi_\epsilon|\bigg).
\end{split}
\]

It follows from  Lemma \ref{lem4.1}   that 
\begin{equation*}
\begin{split}
&\int_{B_\rho(\xi_{i}^{(1)})}   \big(u_\epsilon^{(1)}\big)^{p_{s}-\epsilon}
\psi_\epsilon
\\
=&d_{i,0}^{(1)}\lambda_{i}^{(1)}
\int_{B_\delta(\xi_{i}^{(1)})}  \big(u_\epsilon^{(1)}\big)^{p_{s}-\epsilon}
Z_{\lambda_i^{(1)},\xi_i^{(1)}}^0
+\displaystyle\sum_{l=1}^{N} d_{i,l}^{(1)}
\big(\lambda_{i}^{(1)}\big)^{-1}
\int_{B_\delta(\xi_{i}^{(1)})}\big(u_\epsilon^{(1)}\big)
^{p_{s}-\epsilon}Z_{\lambda_i^{(1)},\xi_i^{(1)}}^l\\
&
+\int_{B_\delta(\xi_{i}^{(1)})}
\big(u_\epsilon^{(1)}\big)^{p_{s}-\epsilon}\psi_\epsilon^*(x)+O\Big(\big(\lambda_{i}^{(1)}\big)^{-N}\Big)
+O\Big(\big(\lambda_{i}^{(1)}\big)^{-\frac{N+2s}{2}}\|\psi_\epsilon^*\|_*\Big)
\\
=&d_{i,0}^{(1)}\lambda_{i}^{(1)}
\int_{B_\delta(\xi_{i}^{(1)})} \left[\big(U_{\lambda_{i}^{(1)},
\xi_{i}^{(1)}}
+\phi_{\epsilon}^{(1)}\big)^{p_{s}-\epsilon}
-U_{\lambda_{i}^{(1)},\xi_{i}^{(1)}}
^{p_{s}-\epsilon}\right]Z_{\lambda_i^{(1)},\xi_i^{(1)}}^0
\\
&+\displaystyle\sum_{l=1}^{N}d_{i,l}^{(1)}
\big(\lambda_{i}^{(1)}\big)^{-1}
\int_{B_\delta(\xi_{i}^{(1)})} \left[\big(U_{\lambda_{i}^{(1)},
\xi_{i}^{(1)}}
+\phi_{\epsilon}^{(1)}\big)^{p_{s}-\epsilon}
-U_{\lambda_{i}^{(1)},
\xi_{i}^{(1)}}
^{p_{s}-\epsilon}\right]Z_{\lambda_i^{(1)},\xi_i^{(1)}}^l
\\
&+d_{i,0}^{(1)}\lambda_{i}^{(1)}
\int_{B_\delta(\xi_{i}^{(1)})} \left[U_{\lambda_{i}^{(1)},
\xi_{i}^{(1)}}
^{p_{s}-\epsilon}
-U_{\lambda_{i}^{(1)},
\xi_{i}^{(1)}}
^{p_{s}}\right]Z_{\lambda_i^{(1)},\xi_i^{(1)}}^0\\
&+O\big(\|\psi_\epsilon^*\|_*\big)+O\Big(\big(\lambda_{i}^{(1)}\big)^{-N}\Big)
+O\Big(\big(\lambda_{i}^{(1)}\big)^{-\frac{N+2s}{2}}\|\psi_\epsilon^*\|_*\Big)\\
:=&F_{1}+F_{2}+F_{3}+O\big(\|\psi_\epsilon^*\|_*\big)
+O\Big(\big(\lambda_{i}^{(1)}\big)^{-N}\Big)
.
\end{split}
\end{equation*}
A direct computation yields that for $i=1,2,$ 
\[
\begin{split}
|F_i|\leq& C\int_{B_\delta(\xi_{i}^{(1)})} U_{\lambda_{i}^{(1)},\xi_{i}^{(1)}
}
^{p_{s}-\epsilon}|\phi_{\epsilon}^{(1)}|+C
\int_{B_\delta(\xi_{i}^{(1)})} U_{\lambda_{i}^{(1)},\xi_{i}^{(1)}
}
|\phi_{\epsilon}^{(1)}|^{p_{s}-\epsilon}
\leq C\big(\|\phi_{\epsilon}^{(1)}\|_*
+\|\phi_{\epsilon}^{(1)}\|_*\big)^{p_{s}-\epsilon}.
\end{split}
\]
Using the mean value theorem, we find
\[
|F_3|\leq C\epsilon \int_{B_\delta(\xi_{i}^{(1)})} U_{\lambda_{i}^{(1)},\xi_{i}^{(1)}
}
^{p_{s}+1-\theta\epsilon}|\log U_{\lambda^{(1)}_{i},\xi_{i}^{(1)}
}|\leq C\epsilon|\log\epsilon|,
\]
where  $\theta\in(0,1)$. Thus we conclude that 
\[
\int_{B_\rho(\xi_{i}^{(1)})} \big(u_\epsilon^{(1)}\big)^{p_{s}-\epsilon}
\psi_\epsilon=o(1).
\]
In view of Lemmas \ref{lem4.1} and \ref{cor5.4}, we obtain 
\[
\begin{split}
 &\int_{B_\rho(\xi_{i}^{(1)})} \big(u_\epsilon^{(1)}\big)^{p_{s}-1-\epsilon}
 \big|u_\epsilon^{(1)}-u_\epsilon^{(2)}\big| |\psi_\epsilon|\\
=&o(1)\int_{
B_\rho(\xi_{i}^{(1)})} U_{\lambda_{i}^{(1)},\xi_{i}^{(1)}
}
^{p_{s}-\epsilon}
|\psi_\epsilon|
+O\bigg(\int_{B_\rho(\xi_{i}^{(1)})} U_{\lambda_{i}^{(1)},\xi_{i}^{(1)}
}^{p_{s}-1-\epsilon}
|\psi_\epsilon|\big(|\phi_{\epsilon}^{(1)}|
+|\phi_{\epsilon}^{(2)}|\big)\bigg)
\\
&
+
 O\bigg(\big(\lambda_i^{(1)}\big)^{-\frac{N-2s}{2}}\int_{B_\rho(\xi_{i}^{(1)})} \big(u_\epsilon^{(1)}\big)^{p_{s}-1-\epsilon}|\psi_\epsilon|
        \chi_{B_{\frac{5\delta}{2}}(\xi_{i}^{(1)})\setminus B_{\frac{\delta}{2}}(\xi_{i}^{(1)})}
        \bigg)\\
=&o(1).
\end{split}
\]
Similarly, we have 
\[
\begin{split}
 \int_{B_\rho(\xi_{i}^{(1)})} \big|u_\epsilon^{(1)}-u_\epsilon^{(2)}\big|
 ^{p_{s}-\epsilon}|\psi_\epsilon|=o(1).
\end{split}
\]
Then we  derive 
\[
\int_{B_\rho(\xi_{i}^{(1)})} b_\epsilon(x)\psi_\epsilon
=o(1).
\]
On the other hand, we have 
\[
\begin{split}
&-s\int_{B_\rho(\xi_{i}^{(1)})}  V(x)\big(u_\epsilon^{(1)}
+u_\epsilon^{(2)}\big)\psi_\epsilon\\
=&-2s d_{i,0}^{(1)}\lambda_{i}^{(1)}
\int_{B_\delta(\xi_{i}^{(1)})}V(x)u_\epsilon^{(1)}
Z_{\lambda_i^{(1)},\xi_i^{(1)}}^0-2s\big(\lambda_{i}^{(1)}\big)^{-1}\displaystyle\sum_{l=1}
^{N}d_{i,l}^{(1)}\int_{B_\delta(\xi_{i}^{(1)})}V(x)
u_\epsilon^{(1)}Z_{\lambda_i^{(1)},\xi_i^{(1)}}^l
\\
&+s d_{i,0}^{(1)}\lambda_{i}^{(1)}
\int_{B_\delta(\xi_{i}^{(1)})}V(x)
\big(u_\epsilon^{(1)}-
u_\epsilon^{(2)}\big)Z_{\lambda_i^{(1)},\xi_i^{(1)}}^0
\\
&+s\big(\lambda_{i}^{(1)}\big)^{-1}\displaystyle\sum_{l=1}
^{N} d_{i,l}^{(1)}\int_{B_\delta(\xi_{i}^{(1)})}
V(x)\big(u_\epsilon^{(1)}-u_\epsilon^{(2)}
\big)Z_{\lambda_i^{(1)},\xi_i^{(1)}}^l\\
&-s
\int_{B_\delta(\xi_{i}^{(1)})}V(x)
\big(u_\epsilon^{(1)}
+u_\epsilon^{(2)}\big)\psi_\epsilon^*+O\Big(\big(\lambda_{i}^{(1)}\big)^{-N+2s}\Big)\\
:=&M_{1}+M_{2}+M_{3}+M_{4}+M_{5}+O\Big(\big(\lambda_{i}^{(1)}\big)^{-N+2s}\Big).
\end{split}
\]
Then we infer 
\[
\begin{split}
M_{1}=&-2s d_{i,0}^{(1)}\lambda_{i}^{(1)}
\int_{B_\delta(\xi_{i}^{(1)})}V(x)
U_{\lambda_{i}^{(1)},\xi_{i}^{(1)}
}Z_{\lambda_i^{(1)},\xi_i^{(1)}}^0-2s d_{i,0}^{(1)}
\lambda_{i}^{(1)}
\int_{B_\delta(\xi_{i}^{(1)})}V(x)Z_{\lambda_i^{(1)},\xi_i^{(1)}}^0 
\phi_{\epsilon}^{(1)}\\
=&-2s\left(d_{i}^{0}\right)^{(1)}\lambda_{i}^{(1)}V(\xi_{i}^{(1)})
\int_{B_\delta(\xi_{i}^{(1)})}
U_{\lambda_{i}^{(1)},\xi_{i}^{(1)}
}
Z_{\lambda_i^{(1)},\xi_i^{(1)}}^0+O\left(\big(\lambda_{i}^{(1)}
\big)^{-2s-1}
 \right)+
O\left(\big( \lambda_{i}^{(1)}
\big)^{-2s}
\|\phi_{\epsilon}^{(1)}\|_*\right)\\
=&s(2s-N)\left(d_{i}^{0}\right)^{(1)}
\gamma_{s,N}^2
\big(\lambda_{i}^{(1)}\big)^{-2s}
V(\xi_{i}^{(1)})
\int_{\mathbb{R}^N}
\frac{1-|x|^{2}}{(1+|x|^{2})
^{N-2s+1}}
+O\Big(\big(\lambda_{i}^{(1)}
\big)^{2s-N}
 \Big)\\
 &+O\Big(\big(\lambda_{i}^{(1)}
\big)^{-2s-1}
 \Big)+
O\Big(\big(\lambda_{i}^{(1)}
\big)^{-2s}
\|\phi_{\epsilon}^{(1)}\|_*
\Big).
\end{split}
\]
By a direct calculation, we get 
\[
\begin{split}
M_{2}=&-2s\big(\lambda_{i}^{(1)}\big)^{-1}\displaystyle\sum_{l=1}
^{N}d_{i,l}^{(1)}
\int_{B_\delta(\xi_{i}^{(1)})}V(x)
\big(U_{\lambda_{i}^{(1)},\xi_{i}^{(1)}
}+\phi_\epsilon\big)
Z_{\lambda_i^{(1)},\xi_i^{(1)}}^l\\
=&-2s\big(\lambda_{i}^{(1)}\big)^{-1}
\displaystyle\sum_{l=1}
^{N}d_{i,l}^{(1)}
\int_{B_\delta(\xi_{i}^{(1)})}
\big\langle\nabla V(\xi_{i}^{(1)}),x-\xi_{i}^{(1)}\big\rangle
U_{\lambda_{i}^{(1)},\xi_{i}^{(1)}
}
Z_{\lambda_i^{(1)},\xi_i^{(1)}}^l\\
&+
O\Big(\big(\lambda_{i}^{(1)}
\big)^{-2s}
\|\phi_{\epsilon}^{(1)}\|_*
\Big) + O\left(\epsilon^{2-\frac{\sigma}{2s}}
\right)\\
=&-\frac{2s(N-2s)}
{N}\gamma_{s,N}^2\big(\lambda_{i}^{(1)}\big)^{-1-2s}
\displaystyle\sum_{l=1}
^{N} d_{i,l}^{(1)}
\frac{\partial V(\xi_{i}^{(1)})}{\partial x_{l}}
\int_{\mathbb{R}^N}
\frac{|x|^{2}}{(1+|x|^{2})
^{N-2s+1}}\\
&
+O\left(\big(\lambda_{i}^{(1)}
\big)^{2s-1-N}
\right)+
O\left(\epsilon^{2-\frac{\sigma}{2s}}
\right).
\end{split}
\]
From Lemma \ref{cor5.4}, it follows that 
\[
\begin{split}
M_{3}\leq &C\int_{B_\delta(\xi_{i}^{(1)})}\big|
u_\epsilon^{(1)}-u_\epsilon^{(2)}
\big|
U_{\lambda_{i}^{(1)},\xi_{i}^{(1)}
}
 \\
=&o(1)\int_{B_\delta
(\xi_{i}^{(1)})}
U_{\lambda_{i}^{(1)},\xi_{i}^{(1)}
}^2
+O\bigg(\int_{B_\delta(\xi_{i}^{(1)})}
U_{\lambda_{i}^{(1)},\xi_{i}^{(1)}
}
\big(|\phi_{\epsilon}^{(1)}|
+|\phi_{\epsilon}^{(2)}|\big)\bigg)+
 O\big(\big(\lambda_i^{(1)})^{2s-N}  
        \big)\\
=&o(\epsilon)+
O
\big(\epsilon \|\phi_{\epsilon}^{(1)}\|_*\big)+
 O\big(\big(\lambda_i^{(1)})^{2s-N}  
        \big)\\
=&o(\epsilon).
\end{split}\]
In a similar way, we get 
\[
\begin{split}
|M_{4}|+|M_5|
= o(\epsilon).
\end{split}\]
Moreover, we have 
\[\begin{split}
&\ \ \ \int_{B_\rho(\xi_{i}^{(1)})} \big\langle x-\xi_{i}^{(1)},\nabla V(x)\big\rangle\big(u_\epsilon^{(1)}
+u_\epsilon^{(2)}\big)\psi_\epsilon
\\
&=\int_{B_\rho(\xi_{i}^{(1)})} \big\langle x-\xi_{i}^{(1)},\nabla V(\xi_{i}^{(1)})\big\rangle\big(u_\epsilon^{(1)}
+u_\epsilon^{(2)}\big)\psi_\epsilon +o(\epsilon)
\\&=O\left(\nabla V(\xi_{i}^{(1)})\big(\lambda_{i}^{(1)}
\big)^{-2s}\right)+o(\epsilon)
\\&
=o(\epsilon). 
\end{split}
\]
By an argument analogous to that in Lemma \ref{lemGA.5}, we derive that
\begin{equation}\label{eq0512-3}
\Big|\nabla 
\left(\lambda_{i}^{(1)}\widetilde{Z}_{i,0}^{(1)}\right)\Big|
+
\Big|\nabla\left(\big(\lambda_{i}^{(1)}\big)^{-1}
\widetilde{Z}_{i,l}^{(1)}\right)\Big|\leq \frac{C}{(\lambda_i^{(1)})^{\frac{N-2s}{2}}}\frac{1}{(1+|x-\xi_i^{(1)}|)^{N-2s+1}} ~\text{on $\partial^{''} \mathfrak{B}^{+}_\rho(\xi_{i}^{(1)})$,}
\end{equation}
\begin{equation}\label{eq0512-1}
    |\widetilde{u}_\epsilon^{(1)}|\leq \frac{C}{(\lambda_i^{(1)})^{\frac{N-2s}{2}}}\frac{1}{(1+|x-\xi_i^{(1)}|)^{N-2s}} ~\text{on $\partial^{''} \mathfrak{B}^{+}_\rho(\xi_{i}^{(1)})$,}
\end{equation}
and
\begin{equation}\label{eq0512-2}
    |\widetilde{\phi}_\epsilon^{(1)}|\leq \frac{C\|{\phi}_\epsilon^{(1)}\|_*}{(\lambda_i^{(1)})^{\sigma}}\frac{1}{(1+|x-\xi_i^{(1)}|)^{\frac{N-2s}{2}+\sigma}} ~\text{on $\partial^{''} \mathfrak{B}^{+}_\rho(\xi_{i}^{(1)})$},
\end{equation}
where $\widetilde{Z}_{i,l}^{(1)}$ is the $s$-harmonic extension of ${Z}_{i,l}^{(1)}$, $i=1,\cdots,k$, $l=0,1,\cdots,N$. 

Thanks to Lemma \ref{F0514lem1}, we find 
\begin{equation}\label{eq0512-5}
    \int_{\partial^{''} \mathfrak{B}^{+}_\rho(\xi_{i}^{(1)})}
t^{1-2s}|\nabla \widetilde{ \psi}_\epsilon|^{2}=o(\epsilon^{\frac{1}{s}+\frac{\sigma}{s}}),
\end{equation}
which gives
\begin{equation}\label{eq0512-6}
    \int_{\partial^{''} \mathfrak{B}^{+}_\rho(\xi_{i}^{(1)})}
t^{1-2s}|\nabla \widetilde{u}_\epsilon^{(1)}|^{2}
\leq  C \epsilon^2.
\end{equation}
Then we obtain 
\[
\text{RHS \;of \;\eqref{5.3}}=o\left(\epsilon^{1+\frac{1}{2s}+\frac{\sigma}{2s}}\right),
\]
From \eqref{5.3} in Lemma \ref{lem5.2}, we conclude that \eqref{tta} holds.
This completes the proof. 
\end{proof}
\begin{lemma}\label{lem5.61}
Under the same assumptions as in Lemma \ref{lem5.6},  it holds  that for $i=1,\cdots,k$, $l=1,\cdots,N$, 
\begin{equation*}\label{ttta}
\big|d_{i,l}^{(1)}\big|=o(1).
\end{equation*}
\end{lemma}
\begin{proof} Note that 
\[
\begin{split}
&\ \ \  \ \ \frac{1}{2}
\displaystyle\int_{B_\rho(\xi_{i}^{(1)})}\frac{\partial V(x)}{\partial x_{l}}\left(u_\epsilon^{(1)}+u_\epsilon^{(2)}\right)
\psi_\epsilon
\\
&\ 
=\int_{B_\delta(\xi_{i}^{(1)})} \frac{\partial V(x)}{\partial x_{l}}
\Big(U_{\lambda_{i}^{(1)},
\xi_{i}^{(1)}}
+\phi_{\epsilon}^{(1)}\Big)
\Big(d_{i,0}^{(1)}\lambda_{i}^{(1)}Z_{\lambda_i^{(1)},\xi_i^{(1)}}^0
+\displaystyle\sum_{m=1}
^{N}{d_{i,m}^{(1)}\big(\lambda_{i}^{(1)}\big)^{-1}
Z_{i,m}^{(1)}}
\Big)\\
&\ \ \quad+\frac{1}{2}\int_{B_\delta(\xi_{i}^{(1)})}\frac{\partial V(x)}{\partial x_{l}}
\big(u_\epsilon^{(2)}-
u_\epsilon^{(1)}\big)
\Big(d_{i,0}^{(1)}\lambda_{i}^{(1)}Z_{\lambda_i^{(1)},\xi_i^{(1)}}^0
+\displaystyle\sum_{m=1}
^{N}{d_{i,m}^{(1)}\big(\lambda_{i}^{(1)}\big)^{-1}
Z_{i,m}^{(1)}}
\Big)
\\
&\ \ \quad+\frac{1}{2}
\int_{B_\delta(\xi_{i}^{(1)})}\frac{\partial V(x)}{\partial x_{l}}\left(u_\epsilon^{(1)}+u_\epsilon^{(2)}\right)
\psi_\epsilon^*(x)
+\frac{1}{2}
\displaystyle\int_{B_\rho(\xi_{i}^{(1)})\setminus B_\delta(\xi_{i}^{(1)})}\frac{\partial V(x)}{\partial x_{l}}\left(u_\epsilon^{(1)}+u_\epsilon^{(2)}\right)
\psi_\epsilon
\\
&\ =d_{i,0}^{(1)}\lambda_{i}^{(1)}
\int_{B_\delta(\xi_{i}^{(1)})}
\frac{\partial V(x)}{\partial x_{l}}U_{\lambda_{i}^{(1)},
\xi_{i}^{(1)}}
Z_{\lambda_i^{(1)},\xi_i^{(1)}}^0
+\displaystyle\sum_{m=1}
^{N}d_{i,m}^{(1)}\big(\lambda_{i}^{(1)}\big)^{-1}
\int_{B_\delta(\xi_{i}^{(1)})} \frac{\partial V(x)}{\partial x_{l}}
U_{\lambda_{i}^{(1)},\xi_{i}^{(1)}
}
Z_{i,m}^{(1)}
\\
&\  \quad+\int_{B_\delta(\xi_{i}^{(1)})}\frac{\partial V(x)}{\partial x_{l}}\bigg(
\frac{u_\epsilon^{(2)}-
u_\epsilon^{(1)}}{2}\bigg)
\Big(d_{i,0}^{(1)}\lambda_{i}^{(1)}Z_{\lambda_i^{(1)},\xi_i^{(1)}}^0
+\displaystyle\sum_{m=1}
^{N}{d_{i,m}^{(1)}\big(\lambda_{i}^{(1)}\big)^{-1}
Z_{i,m}^{(1)}}
\Big) 
\\
&\ \  \quad+\frac{1}{2}
\int_{B_\delta(\xi_{i}^{(1)})}\frac{\partial V(x)}{\partial x_{l}}\left(u_\epsilon^{(1)}+u_\epsilon^{(2)}\right)
\psi_\epsilon^*(x)+O\left(\big(\lambda_{i}^{(1)}
\big)^{-2s}
\|\phi_{\epsilon}^{(1)}\|_*
\right)+O\left(\big(\lambda_{i}^{(1)}
\big)^{2s-N}
\right)
\\
&
:=G_{1}+G_{2}+O\left(\big(\lambda_{i}^{(1)}
\big)^{-2s}
\|\phi_{\epsilon}^{(1)}\|_*
\right)+O\left(\big(\lambda_{i}^{(1)}
\big)^{2s-N}
\right)+o\Big(\big(\lambda_i^{(1)}\big)^{-1-2s}\Big).\\
\end{split}
\]
Since $\nabla V(\xi_{i}^{(1)})=o(\epsilon^{\frac{1}{2s}})$,  we have
\[
\begin{split}
G_{1}
=& d_{i,0}^{(1)}\lambda_{i}^{(1)}\frac{\partial V(\xi_{i}^{(1)}) }{\partial x_{l}} \int_{B_\delta(\xi_{i}^{(1)})}
U_{\lambda_{i}^{(1)},
\xi_{i}^{(1)}}
Z_{\lambda_i^{(1)},\xi_i^{(1)}}^0
+o\Big(\big(\lambda_i^{(1)}\big)^{-1-2s}\Big)
\\
=&\frac{N-2s}{2}\gamma_{s,N}^2d_{i,0}^{(1)} \frac{\partial V(\xi_{i}^{(1)}) }{\partial x^{l}} \big(\lambda_{i}^{(1)}\big)^{-2s}\int_{\mathbb{R}^N}\frac{1-|x|^{2}}
{(1+|x|^{2})^{N-2s+1}}\\
&
+O\Big(
\big(\lambda_{i}^{(1)}\big)^{2s-N}
 \Big)
+o\Big(\big(\lambda_{i}^{(1)}\big)^{-2s-1}\Big)
\\
=&
o\left(\epsilon^{1+\frac{1}{2s}}\right)
\end{split}
\]and
 \[
\begin{split}
G_{2}=& \displaystyle\sum_{m=1}
^{N}d_{i,m}^{(1)}\big(\lambda_{i}^{(1)}\big)^{-1}
\int_{B_\delta
(\xi_{i}^{(1)})}
\Big(
\frac{\partial V(x)}{\partial x_{l}}-
\frac{\partial V(\xi_{i}^{(1)})}{\partial x_{l}}
\Big)
U_{\lambda_{i}^{(1)},\xi_{i}^{(1)}
}
Z_{i,m}^{(1)}\\
=&\frac{\gamma_{s,N}^2}{N}\displaystyle\sum_{j=1}
^{N}
d_{i,m}^{(1)}
\big(\lambda_{i}^{(1)}\big)^{-1-2s}
\frac{\partial^{2} V(\xi_{i}^{(1)})}{\partial x_{l}
x_{j}}
\int_{\R^N}
\frac{|x|^{2}}
{(1+|x|^{2})^{N-2s+1}}+o\Big(\big(\lambda_i^{(1)}\big)^{-1-2s}\Big).
\end{split}
\]
From Lemma \ref{F6.9} and \eqref{eq0512-1}-\eqref{eq0512-6}, it follows  that
\[
\text{RHS \;of \;\eqref{14.73}}=o\left(\epsilon^{1+\frac{1}{2s}+\frac{\sigma}{2s}}\right).
\]
By virtue of \eqref{14.73} in Lemma \ref{lem14.33}, we derive 
\begin{equation*}
\big|d_{i,l}^{(1)}\big|=o(1)\ \; \hbox{for}\ \; l=1,2,\cdots,N.
\end{equation*}
This  concludes the proof of Lemma \ref{lem5.61}.
\end{proof}

\subsection{Completion of the proof of Theorem \ref{thm1.3}}
\begin{proof}[\bf Proof of Theorem \ref{thm1.3}] 
As in the proof of Theorem \ref{NDth}, we have 
\[
\begin{split}
|\psi_\epsilon|&\leq C\int_{\mathbb{R}^N}\frac{1}{|x-y|^{N-2s}} c_\epsilon(y)|\psi_\epsilon(y)|
\leq C\|\psi_\epsilon\|_*
\displaystyle\sum_{i=1}^{k}
\frac{(\lambda_{i}^{(1)})^{\frac{N-2s}{2}}}
{\big(1+\lambda_{i}^{(1)}
|x-\xi_{i}^{(1)}|\big)^{\frac{N+2s}{2}-\sigma_{0}}}
,\end{split}
\]
where $\sigma_0>0$ is a small fixed constant. 
 Thus, we infer 
\begin{equation*}
\bigg(
\displaystyle\sum_{i=1}^{k}
\frac{(\lambda_{i}^{(1)})^{\frac{N-2s}{2}}}
{\big(1+\lambda_{i}^{(1)}|x-\xi_{i}^{(1)}|\big)
^{\frac{N-2s}{2}+\sigma}}\bigg)^{-1}
|\psi_\epsilon(x)|
\leq  C\|\psi_\epsilon\|_*
\displaystyle\sum_{i=1}^{k}\frac{1}
{\big(1+\lambda_{i}^{(1)}|x-\xi_{i}^{(1)}|\big)
^{2s-\sigma-\sigma_0}}.
\end{equation*}
For any $R>0$, by Lemmas \ref{lem5.5}, \ref{lem5.6} and \ref{lem5.61}, we obtain 
 \begin{equation*}
 \big(\lambda_{i}^{(1)}\big)^{-\frac{N
-2s}{2}}\psi_\epsilon\rightarrow0\,\,\,\mbox{in}\,\,\, B_{R(\lambda_{i}^{(1)})^{-1}}(\xi_{i}^{(1)}),\ \  i=1,\cdots,k.\end{equation*}
Then it follows from $\|\psi_\epsilon\|_*=1$ that 
$$
\bigg(\displaystyle\sum_{i=1}^{k}
\frac{\big(\lambda_{i}^{(1)}\big)^{\frac{N-2s}{2}}}
{\big(1+(\lambda_{i}^{(1)})|x-\xi_{i}^{(1)}|\big)
^{\frac{N-2s}{2}+\sigma}}\bigg)^{-1}|\psi_\epsilon(x)|
$$
cannot attain its maximum in $B_{R(\lambda_{i}^{(1)})^{-1}}(\xi_{i}^{(1)})$. Taking  $R$ sufficiently large  such that
\[
\frac{C}{(1+R)^{2s-\sigma-\sigma_0}}\leq\frac{1}{2k},
\]
we find
\begin{equation*}
\bigg(
\displaystyle\sum_{i=1}^{k}
\frac{\big(\lambda_{i}^{(1)}\big)^{\frac{N-2s}{2}}}
{\big(1+\lambda_{i}^{(1)}|x-\xi_{i}^{(1)}|\big)
^{\frac{N-2s}{2}+\sigma}}\bigg)^{-1}
|\psi_\epsilon(x)|
\leq\frac{1}{2}
\|\psi_\epsilon\|_*\ \ \text{for}\;x\in \R^N\setminus \cup_{i=1}^k B_{R(\lambda_{i}^{(1)})^{-1}}
(\xi_{i}^{(1)}).
\end{equation*}
Therefore, for $\epsilon$ sufficiently small, we derive  
\[
\|\psi_\epsilon\|_*\leq\frac{1}{2}\|\psi_\epsilon\|_*.
\]
This completes the proof of Theorem \ref{thm1.3}.
\end{proof}    

\appendix

\section{Asymptotic expansions of  the energy functional}

In this appendix, we present the asymptotic expansion of the energy functional of  the approximate solutions. 
Let  $\xi_i$, $i=1,\cdots, k$ be $k$ distinct points in $\mathbb{R}^N$. We set
 $$d=\min\limits_{ i\neq j,i,j=1,\cdots,k}|\xi_i-\xi_j|\;\;\text{and}\;\;\delta=\frac{d}{10}. $$
 Let  $\eta\in C^\infty(\R^N)$ be a smooth cut-off function satisfying $\eta(x)=1$ for $|x|\le \delta$, 
$\eta(x)=0$ for $ |x|\ge 2\delta$ and $0\le \eta\le 1$. For $i=1,\cdots,k$, let $\eta_i(x)=\eta(x-\xi_i)$.
 Then 
the approximate solutions are defined by 
$$
W_{\lambda_i, \xi_i}=\eta_iU_{\lambda_i, \xi_i},~~ {\bm W}_{\bm \lambda, \bm\xi}=\sum_{i=1}^kW_{\lambda_i, \xi_i}.
$$
The energy functional associated with these approximate solutions reads
$$
I(u)=\frac{1}{2}\int_{\mathbb{R}^N}\left([(-\Delta)^{\frac{s}{2}}u]^2+V(x)u^2\right)-\frac{1}{p_s+1-\epsilon}\int_{\R^N}u_+^{p_s+1-\epsilon},
$$
where $u_+=\max(u,0)$.
\begin{lemma}\label{lemA.1}
Let $N>4s$. Then for $i=1,2,\cdots,k$,
   $$
   \begin{aligned}
    \Big\langle I'(W_{\lambda_i,\xi_i}),\,\frac{\partial W_{\lambda_i, \xi_i}}{\partial\lambda_i}\Big\rangle=&A\epsilon\lambda_i^{-1}+BV(\xi_i)\lambda_i^{-2s-1}+O\left(\epsilon\lambda_i^{-1-N}\log \lambda_i\right)+O\left(\epsilon^2\lambda_i^{-1}\log ^2\lambda_i\right)\\
    &+O\left(\lambda_i^{-2s-1-\tau}\right),
    \end{aligned}
    $$
    where $\tau\in(0,\frac{N-4s}{2})$ is a fixed small constant and 
    $$
    A=-\left(\frac{N-2s}{2}\right)^2\gamma_{s,N}^{p_s+1}\int_{\R^N}\frac{1-|x|^2}{(1+|x|^2)^{N+1}}\log{(1+|x|^2)}>0,
    $$ 
    $$
    B=\frac{N-2s}{2}\gamma_{s,N}^{2}\int_{\R^N}\frac{1-|x|^2}{(1+|x|^2)^{N-2s+1}}<0.
    $$
\end{lemma}
\begin{proof}
 Note that  
    $$
    \begin{aligned}
       &\displaystyle\Big\langle I'(W_{\lambda_i,\xi_i}),\,\frac{\partial W_{\lambda_i, \xi_i}}{\partial\lambda_i}\Big\rangle\\
        =&\displaystyle\int_{\R^N}\left(\eta_i^2 U_{\lambda_i,\xi_i}^{p_s}-\eta_i^{p_s+1-\epsilon}U_{\lambda_i,\xi_i}^{p_s-\epsilon}\right) \frac{\partial U_{\lambda_i, \xi_i}}{\partial\lambda_i}
        +\int_{\R^N}V(x) \eta_i^2 U_{\lambda_i, \xi_i}\frac{\partial U_{\lambda_i, \xi_i}}{\partial\lambda_i}
        +\int_{\R^N}\Upsilon_i(x) \eta_i\frac{\partial U_{\lambda_i, \xi_i}}{\partial\lambda_i}\\
        =&\int_{B_\delta(\xi_i)}\left(U_{\lambda_i,\xi_i}^{p_s}-U_{\lambda_i,\xi_i}^{p_s-\epsilon}\right) \frac{\partial U_{\lambda_i, \xi_i}}{\partial\lambda_i}+\int_{B_\delta(\xi_i)}V(x) U_{\lambda_i, \xi_i}\frac{\partial U_{\lambda_i, \xi_i}}{\partial\lambda_i}
        +O\big(\lambda_i^{-N-1+2s}\big).
    \end{aligned}
    $$
Then it follows from  the mean value theorem that 
    $$
\begin{aligned}
    &\int_{B_\delta(\xi_i)}\left(U_{\lambda_i,\xi_i}^{p_s}-U_{\lambda_i,\xi_i}^{p_s-\epsilon}\right) \frac{\partial U_{\lambda_i, \xi_i}}{\partial\lambda_i}\\
    =&\epsilon\int_{B_\delta(\xi_i)}U_{\lambda_i,\xi_i}^{p_s}\frac{\partial U_{\lambda_i, \xi_i}}{\partial\lambda_i}\log U_{\lambda_i, \xi_i} +C\epsilon^2\int_{B_\delta(\xi_i)}U_{\lambda_i,\xi_i}^{p_s-\theta\epsilon}\frac{\partial U_{\lambda_i, \xi_i}}{\partial\lambda_i}\log^2 U_{\lambda_i, \xi_i} \\
    =&\epsilon\int_{\R^N}U_{\lambda_i,\xi_i}^{p_s}\frac{\partial U_{\lambda_i, \xi_i}}{\partial\lambda_i}\log U_{\lambda_i, \xi_i}+C\epsilon^2\int_{B_\delta(\xi_i)}U_{\lambda_i,\xi_i}^{p_s-\theta\epsilon}\frac{\partial U_{\lambda_i, \xi_i}}{\partial\lambda_i}\log^2 U_{\lambda_i, \xi_i}+O\left(\epsilon\lambda_i^{-1-N}\log \lambda_i\right) \\
    =&-\left(\frac{N-2s}{2}\right)^2\gamma_{s,N}^{p_s+1}\epsilon\lambda_i^{-1}\int_{\R^N}\frac{1-|x|^2}{(1+|x|^2)^{N+1}}\log{(1+|x|^2)}\\
    &+O\left(\epsilon\lambda_i^{-1-N}\log \lambda_i\right)+O\left(\epsilon^2\lambda_i^{-1}\log ^2\lambda_i\right),
\end{aligned}
    $$
    where $\theta \in (0,1) $. A  direct computation yields
    \begin{align*}
        &\int_{B_\delta(\xi_i)}V(x) U_{\lambda_i, \xi_i}\frac{\partial U_{\lambda_i, \xi_i}}{\partial\lambda_i}\\
        =&\frac{N-2s}{2}\gamma_{s,N}^{2}\lambda_i^{-2s-1}\int_{B_{\delta\lambda_i}(0)} V\big(\xi_i+\frac{x}{\lambda_i}\big)\frac{1-|x|^2}{(1+|x|^2)^{N-2s+1}} \\
        =&\frac{N-2s}{2}\gamma_{s,N}^{2}V(\xi_i)\lambda_i^{-2s-1}\int_{\R^N}\frac{1-|x|^2}{(1+|x|^2)^{N-2s+1}}+O\left(\lambda_i^{-2s-1-\tau}\right),\\
    \end{align*}
where  $\tau\in (0,\frac{N-4s}{2})$. 
    Hence,  
    \begin{align*}
        \Big\langle I'(W_{\lambda_i,\xi_i}),\,\frac{\partial W_{\lambda_i, \xi_i}}{\partial\lambda_i}\Big\rangle&=A\epsilon\lambda_i^{-1}+BV(\xi_i)\lambda_i^{-2s-1}+O\left(\epsilon\lambda_i^{-1-N}\log \lambda_i\right)\\
        &\ \ \ +O\left(\epsilon^2\lambda_i^{-1}\log ^2\lambda_i\right)+O\left(\lambda_i^{-2s-1-\tau}\right).
    \end{align*}
\end{proof}
\section{Local Pohozaev identities}

In this section, we  give the proof of  the local Pohozaev identities in Lemma \ref{lem4.3} corresponding to \eqref{eq-ND1} and \eqref{eq-ND2}.  Recall that the $s$-harmonic extension  $\widetilde{u}_\epsilon$ and $\widetilde\omega_\epsilon$
  satisfy
    \begin{equation}\label{Feq-ND1}
        \begin{cases}
            \text{div}(t^{1-2s}\nabla\widetilde{u}_\epsilon)=0 \quad &\text{in} \quad \R^{N+1}_+,\\
            -\lim\limits_{t\to 0^+}t^{1-2s}\partial_t\widetilde{u}_\epsilon=-V(x){u}_\epsilon+{u}_\epsilon^{p_s-\epsilon}\quad &\text{on}  \quad\R^{N},
        \end{cases}
         \end{equation}
         and 
    \begin{equation}\label{Feq-ND2}
        \begin{cases}
            \text{div}(t^{1-2s}\nabla\widetilde\omega_\epsilon)=0 \quad &\text{in} \quad\R^{N+1}_+,\\
            -\lim\limits_{t\to 0^+}t^{1-2s}\partial_t\widetilde\omega_\epsilon=-V(x)\omega_\epsilon+(p_s-\epsilon)u_\epsilon^{p_s-1-\epsilon}\omega_\epsilon\quad &\text{on}  \quad\R^{N}.
        \end{cases}
         \end{equation}

\begin{lemma}\label{Flem4.3}
Assume that $V(x)\in C^2(B_{5\delta}(\xi_i))$, $i=1,\cdots,k$. Then  we have
\begin{equation}\label{F4.13}
\begin{split}
&\frac{N-2}{2}\epsilon\int_{B_\rho(
\xi_{i})} u_\epsilon^{p_{s}-\epsilon}\omega_\epsilon
-2s\int_{B_\rho(\xi_{i})}V(x)u_\epsilon
\omega_\epsilon
-\int_{B_\rho(\xi_{i})}\big\langle x-\xi_{i},\nabla V(x)\big\rangle u_\epsilon\omega_\epsilon\\
=&\int_{\partial^{''} \mathfrak{B}^{+}_\rho(\xi_{i})}
t^{1-2s}\frac{\partial \widetilde{u}_\epsilon}{\partial\nu}\langle X-(\xi_{i},0),\nabla\widetilde\omega_\epsilon\rangle
+\int_{\partial^{''} \mathfrak{B}^{+}_\rho(\xi_{i})}
t^{1-2s}\frac{\partial \widetilde\omega_\epsilon}{\partial\nu}\langle X-(\xi_{i},0),\nabla \widetilde{u}_\epsilon\rangle
\\
&+
\frac{N-2s}{2}\int_{\partial^{''} \mathfrak{B}^{+}_\rho(\xi_{i})}
t^{1-2s}\left(\frac{\partial \widetilde{u}_\epsilon}{\partial\nu}\widetilde\omega_\epsilon+\frac{\partial \widetilde\omega_\epsilon}{\partial\nu}\widetilde{u}_\epsilon\right)
-\int_{\partial^{''} \mathfrak{B}^{+}_\rho(\xi_{i})}
t^{1-2s}\langle\nabla \widetilde{u}_\epsilon,\nabla\widetilde\omega_\epsilon\rangle\langle X-(\xi_{i},0),\nu\rangle
\\
&+\int_{\partial B_\rho(\xi_{i})} u_\epsilon^{p_{s}-\epsilon}\omega_\epsilon \langle x-\xi_{i},\nu\rangle -\int_{\partial B_\rho(\xi_{i})}V(x)u_\epsilon
\omega_\epsilon\langle x-\xi_{i},\nu\rangle
\end{split}
\end{equation}
where $0<\rho<5\delta$ and 
            $$
         \begin{aligned}
             \displaystyle\mathfrak{B}_\rho^+(\xi_i)=&\{X=(x,t):|X-(\xi_i,0)|\leq \rho \text{ and } t>0\}\subseteq\R^{N+1}_+,\\
             \displaystyle\partial'\mathfrak{B}_\rho^+(\xi_i)=&\{X=(x,t):|x-\xi_i|\leq \rho \text{ and } t=0\}\subseteq\R^{N+1}_+,\\
           \displaystyle \partial'' \mathfrak{B}_\rho^+(\xi_i)=&\{X=(x,t):|X-(\xi_i,0)|= \rho \text{ and } t>0\}\subseteq\R^{N+1}_+,\\
           \displaystyle \partial\mathfrak{B}_\rho^+(\xi_i)=&\partial'\mathfrak{B}_\rho^+(\xi_i)\cup\partial''\mathfrak{B}_\rho^+(\xi_i),\\
           \displaystyle B_\rho(\xi_i)=&\{x:|x-\xi_i|\leq \rho\}\subseteq\R^{N}.
         \end{aligned}
         $$
        
\end{lemma}

\begin{proof}
For simplicity, we denote $x^{N+1}=t$.
     Multiplying \eqref{Feq-ND1} by $\langle\nabla\widetilde\omega_\epsilon,\;X-(\xi_{i},0)\rangle$ and  integrating over  $\mathfrak{B}_\rho^{+}(\xi_i)$, we have
     \begin{equation}\label{eq4.4.0}
         \begin{aligned}
             0=&\int_{\mathfrak{B}_\rho^{+}(\xi_i)}\text{div}(t^{1-2s}\nabla\widetilde{u}_\epsilon)\langle\nabla\widetilde\omega_\epsilon,\;X-(\xi_{i},0)\rangle\\
             =&\int_{\partial\mathfrak{B}_\rho^{+}(\xi_i)}t^{1-2s}\frac{\partial\widetilde{u}_\epsilon}{\partial \nu}\langle\nabla\widetilde\omega_\epsilon,\;X-(\xi_{i},0)\rangle
             -\int_{\mathfrak{B}_\rho^{+}(\xi_i)}t^{1-2s}\nabla\widetilde{u}_\epsilon\cdot \nabla\widetilde\omega_\epsilon\\
             &-\int_{\mathfrak{B}_\rho^{+}(\xi_i)}t^{1-2s}\sum_{l=1}^{N+1}\frac{\partial\widetilde{u}_\epsilon}{\partial x_l}\langle\nabla\frac{\partial\widetilde\omega_\epsilon}{\partial x_l},\;X-(\xi_{i},0)\rangle\\
             :=&A_1+A_2+A_3.
         \end{aligned}
     \end{equation}
Next, we compute $A_1$.
     \begin{equation}\label{eq4.4.1}
         \begin{aligned}
             A_1= &  \int_{\partial''\mathfrak{B}_\rho^{+}(\xi_i)\cup \partial'\mathfrak{B}_\rho^{+}(\xi_i)}t^{1-2s}\frac{\partial\widetilde{u}_\epsilon}{\partial \nu}\langle\nabla\widetilde\omega_\epsilon,\;X-(\xi_{i},0)\rangle\\
             =&\int_{\partial''\mathfrak{B}_\rho^{+}(\xi_i)}t^{1-2s}\frac{\partial\widetilde{u}_\epsilon}{\partial \nu}\langle\nabla\widetilde\omega_\epsilon,\;X-(\xi_{i},0)\rangle
             +\int_{B_\rho(\xi_i)}\left(-V(x){u}_\epsilon+{u}_\epsilon^{p_s-\epsilon}\right)\langle\nabla\omega_\epsilon,\;x-\xi_{i}\rangle\\
             =&\int_{\partial''\mathfrak{B}_\rho^{+}(\xi_i)}t^{1-2s}\frac{\partial\widetilde{u}_\epsilon}{\partial \nu}\langle\nabla\widetilde\omega_\epsilon,\;X-(\xi_{i},0)\rangle
             +\int_{\partial B_\rho(\xi_i)}\left(-V(x){u}_\epsilon+{u}_\epsilon^{p_s-\epsilon}\right)\omega_\epsilon\langle\nu,\;x-\xi_{i}\rangle
             \\
             &+\int_{ B_\rho(\xi_i)}{u}_\epsilon\omega_\epsilon\langle\nabla V(x),\;x-\xi_{i}\rangle+\int_{ B_\rho(\xi_i)}V(x)\omega_\epsilon\langle\nabla {u}_\epsilon,\;x-\xi_{i}\rangle\\
             &
             -\int_{ B_\rho(\xi_i)}(p_s-\epsilon){u}_\epsilon^{p_s-1-\epsilon}\omega_\epsilon\langle\nabla {u}_\epsilon,\;x-\xi_{i}\rangle
             -N\int_{ B_\rho(\xi_i)}\left(-V(x){u}_\epsilon+{u}_\epsilon^{p_s-\epsilon}\right)\omega_\epsilon.\\
         \end{aligned}
     \end{equation}
For the computation of $A_2$, we rely on a Pohozaev-type identity. Specifically, we multiply  \eqref{Feq-ND1} by $\widetilde\omega_\epsilon$ and \eqref{Feq-ND2} by $\widetilde{u}_\epsilon$, integrate both products over $\mathfrak{B}_\rho^{+}(\xi_i)$, and apply integration by parts to obtain 
     $$
     \begin{aligned}
       \int_{\mathfrak{B}_\rho^{+}(\xi_i)}t^{1-2s}\nabla\widetilde{u}_\epsilon\nabla\widetilde\omega_\epsilon
       &=\int_{\partial''\mathfrak{B}_\rho^{+}(\xi_i)}t^{1-2s}\frac{\partial\widetilde{u}_\epsilon}{\partial \nu}\widetilde\omega_\epsilon  
       +\int_{{B}_\rho(\xi_i)}\left(-V(x){u}_\epsilon+{u}_\epsilon^{p_s-\epsilon}\right)\omega_\epsilon,  \\
       \int_{\mathfrak{B}_\rho^{+}(\xi_i)}t^{1-2s}\nabla\widetilde{u}_\epsilon\nabla\widetilde\omega_\epsilon
       &=\int_{\partial''\mathfrak{B}_\rho^{+}(\xi_i)}t^{1-2s}\frac{\partial\widetilde\omega_\epsilon}{\partial \nu}\widetilde{u}_\epsilon  
       +\int_{{B}_\rho(\xi_i)}\left(-V(x)\omega_\epsilon+(p_s-\epsilon)u_\epsilon^{p_s-1-\epsilon}\omega_\epsilon\right)u_\epsilon. \\
     \end{aligned}
     $$
Adding these two identities and dividing by 2, we find 
     \begin{equation}\label{eq4.4.3}
          A_2=-\frac{1}{2}\int_{\partial''\mathfrak{B}_\rho^{+}(\xi_i)}t^{1-2s}(\frac{\partial\widetilde{u}_\epsilon}{\partial \nu}\widetilde\omega_\epsilon +\frac{\partial\widetilde\omega_\epsilon}{\partial \nu}\widetilde{u}_\epsilon )
          +\int_{{B}_\rho(\xi_i)}V(x)u_\epsilon\omega_\epsilon-\frac{p_s+1-\epsilon}{2}\int_{{B}_\rho(\xi_i)}{u}_\epsilon^{p_s-\epsilon}\omega_\epsilon.
     \end{equation}
It remains to estimate $A_3$. Multiplying \eqref{Feq-ND2} by $\langle\nabla\widetilde{u}_\epsilon,\;X-(\xi_{i},0)\rangle$ and integrating over  $\mathfrak{B}_\rho^{+}(\xi_i)$, we get     \begin{equation}\label{eq4.2.2}
         \begin{aligned}
             0=&\int_{\mathfrak{B}_\rho^{+}(\xi_i)}\text{div}(t^{1-2s}\nabla\widetilde\omega_\epsilon)\langle\nabla\widetilde{u}_\epsilon,\;X-(\xi_{i},0)\rangle\\
             =&\int_{\partial\mathfrak{B}_\rho^{+}(\xi_i)}t^{1-2s}\frac{\partial\widetilde\omega_\epsilon}{\partial \nu}\langle\nabla\widetilde{u}_\epsilon,\;X-(\xi_{i},0)\rangle
             -\int_{\mathfrak{B}_\rho^{+}(\xi_i)}t^{1-2s}\nabla\widetilde{u}_\epsilon\nabla\widetilde\omega_\epsilon\\
             &-\int_{\mathfrak{B}_\rho^{+}(\xi_i)}t^{1-2s}\sum_{l=1}^{N+1}\frac{\partial\widetilde\omega_\epsilon}{\partial x_l}\langle\nabla\frac{\partial\widetilde{u}_\epsilon}{\partial x_l},\;X-(\xi_{i},0)\rangle\\
             =& \int_{\partial''\mathfrak{B}_\rho^{+}(\xi_i)}t^{1-2s}\frac{\partial\widetilde\omega_\epsilon}{\partial \nu}\langle\nabla\widetilde{u}_\epsilon,\;X-(\xi_{i},0)\rangle +\int_{B_\rho(\xi_i)}\left(-V(x)\omega_\epsilon+(p_s-\epsilon)u_\epsilon^{p_s-1-\epsilon}\omega_\epsilon\right)\langle\nabla u_\epsilon,\;x-\xi_{i}\rangle
            \\
             &  -\int_{\mathfrak{B}_\rho^{+}(\xi_i)}t^{1-2s}\nabla\widetilde{u}_\epsilon\nabla\widetilde\omega_\epsilon -\int_{\mathfrak{B}_\rho^{+}(\xi_i)}t^{1-2s}\sum_{l=1}^{N+1}\frac{\partial\widetilde\omega_\epsilon}{\partial x_l}\langle\nabla\frac{\partial\widetilde{u}_\epsilon}{\partial x_l},\;X-(\xi_{i},0)\rangle .
         \end{aligned}
     \end{equation}
Therefore, integrating by parts and using \eqref{eq4.2.2} yields
      \begin{equation}\label{eq4.4.2}
         \begin{aligned}
             A_3=&-\int_{\partial''\mathfrak{B}_\rho^{+}(\xi_i)}t^{1-2s}\nabla\widetilde{u}_\epsilon\nabla\widetilde\omega_\epsilon\langle X-(\xi_{i},0),\;\nu\rangle
             +\int_{\mathfrak{B}_\rho^{+}(\xi_i)}t^{1-2s}\sum_{l=1}^{N+1}\frac{\partial\widetilde\omega_\epsilon}{\partial x_l}\langle\nabla\frac{\partial\widetilde{u}_\epsilon}{\partial x_l},\;X-(\xi_{i},0)\rangle\\
             &+(N+2-2s)\int_{\mathfrak{B}_\rho^{+}(\xi_i)}t^{1-2s}\nabla\widetilde{u}_\epsilon\nabla\widetilde\omega_\epsilon,\\
             =&- \int_{\partial''\mathfrak{B}_\rho^{+}(\xi_i)}t^{1-2s}\nabla\widetilde{u}_\epsilon\nabla\widetilde\omega_\epsilon\langle X-(\xi_{i},0),\;\nu\rangle+\int_{\partial''\mathfrak{B}_\rho^{+}(\xi_i)}t^{1-2s}\frac{\partial\widetilde\omega_\epsilon}{\partial \nu}\langle\nabla\widetilde{u}_\epsilon,\;X-(\xi_{i},0)\rangle \\
             &+\int_{B_\rho(\xi_i)}\left(-V(x)\omega_\epsilon+(p_s-\epsilon)u_\epsilon^{p_s-1-\epsilon}\omega_\epsilon\right)\langle\nabla u_\epsilon,\;x-\xi_{i}\rangle +(N+1-2s)\int_{\mathfrak{B}_\rho^{+}(\xi_i)}t^{1-2s}\nabla\widetilde{u}_\epsilon\nabla\widetilde\omega_\epsilon \\
             \end{aligned}
    \end{equation}
     Combining \eqref{eq4.4.0}-\eqref{eq4.4.2}, we obtain the result.
\end{proof}

\begin{lemma}\label{Flem4.30}
Assume that $V(x)\in C^2(B_{5\delta}(\xi^*_i))$, $i=1,\cdots,k$. Then, for $0<\rho<5\delta$ and $l=1,\cdots,N$, we have 
\begin{equation}\label{F4.3}
\begin{split}
\int_{B_\rho(\xi_{i})}\frac{\partial V(x)}{\partial x_l}u_\epsilon\omega_\epsilon
=&\int_{\partial B_\rho(\xi_{i})}V(x)u_\epsilon
\omega_\epsilon
\nu_{l}
-\int_{\partial B_\rho(\xi_{i})} u_\epsilon^{p_{s}-\epsilon}\omega_\epsilon \nu_{l}
+
\int_{\partial^{''} \mathfrak{B}^{+}_\rho(\xi_{i})}
t^{1-2s}\langle\nabla \widetilde{u}_\epsilon,\nabla\widetilde\omega_\epsilon\rangle\nu_{l}\\
&- \int_{\partial^{''} \mathfrak{B}^{+}_\rho(\xi_{i})}
t^{1-2s}\frac{\partial \widetilde{u}_\epsilon}{\partial x_l}\frac{\partial\widetilde\omega_\epsilon}{\partial\nu}
-\int_{\partial^{''} \mathfrak{B}^{+}_\rho(\xi_{i})}
t^{1-2s}\frac{\partial
\widetilde\omega_\epsilon}{\partial x_l}\frac{\partial \widetilde{u}_\epsilon}{\partial\nu},
\end{split}
\end{equation}
\end{lemma}
\begin{proof}
 Multiplying \eqref{Feq-ND1} by $\frac{\partial \widetilde\omega_\epsilon}{\partial x_l}$ and  integrating over  $\mathfrak{B}_\rho^{+}(\xi_i)$, we have
         \begin{equation}\label{eq4.5.1}
         \begin{aligned}
             0=&\int_{\mathfrak{B}_\rho^{+}(\xi_i)}\frac{\partial \widetilde\omega_\epsilon}{\partial x_l}\text{div}(t^{1-2s}\nabla\widetilde{u}_\epsilon)=\int_{\partial\mathfrak{B}_\rho^{+}(\xi_i)}t^{1-2s}
             \frac{\partial\widetilde{u}_\epsilon}{\partial \nu}\frac{\partial \widetilde\omega_\epsilon}{\partial x_l}-\int_{\mathfrak{B}_\rho^{+}(\xi_i)}t^{1-2s}\nabla\widetilde{u}_\epsilon\nabla\frac{\partial \widetilde\omega_\epsilon}{\partial x_l}\\
             =&\int_{\partial''\mathfrak{B}_\rho^{+}(\xi_i)}t^{1-2s}\frac{\partial\widetilde{u}_\epsilon}{\partial \nu}\frac{\partial \widetilde\omega_\epsilon}{\partial x_l}+\int_{{B}_\rho(\xi_i)}(-V(x)u_\epsilon+u_\epsilon^{p_s-\epsilon})\frac{\partial \omega_\epsilon}{\partial x_l}-\int_{\mathfrak{B}_\rho^{+}(\xi_i)}t^{1-2s}\nabla\widetilde{u}_\epsilon\nabla\frac{\partial \widetilde\omega_\epsilon}{\partial x_l}\\
              =&\int_{\partial''\mathfrak{B}_\rho^{+}(\xi_i)}t^{1-2s}\frac{\partial\widetilde{u}_\epsilon}{\partial \nu}\frac{\partial \widetilde\omega_\epsilon}{\partial x_l}+\int_{\partial B_\rho(\xi_i)}(-V(x)u_\epsilon+u_\epsilon^{p_s-\epsilon})\omega_\epsilon\nu_l+\int_{B_\rho(\xi_i)}\Big(\frac{\partial V(x)}{\partial x_l}u_\epsilon\omega_\epsilon+V(x)\frac{\partial u_\epsilon}{\partial x_l} \omega_\epsilon\Big)\\
              &-(p_s-\epsilon)\int_{B_\rho(\xi_i)}u_\epsilon^{p_s-1-\epsilon}\frac{\partial u_\epsilon}{\partial x_l} \omega_\epsilon-\int_{\partial''\mathfrak{B}_\rho^{+}(\xi_i)}t^{1-2s}\nabla\widetilde{u}_\epsilon\nabla\widetilde\omega_\epsilon\nu_l
              +\int_{\mathfrak{B}_\rho^{+}(\xi_i)}t^{1-2s}\nabla\widetilde\omega_\epsilon\nabla\frac{\partial \widetilde{u}_\epsilon}{\partial x_l}.\\
         \end{aligned}
         \end{equation}
         Similarly, multiplying \eqref{Feq-ND2} by $\frac{\partial \widetilde u_\epsilon}{\partial x_l}$ and integrating over $\mathfrak{B}_\rho^{+}(\xi_i)$ gives  
          \begin{equation}\label{eq4.5.2}
         \begin{aligned}
             0=&\int_{\mathfrak{B}_\rho^{+}(\xi_i)}\frac{\partial \widetilde{u}_\epsilon}{\partial x_l}\text{div}(t^{1-2s}\nabla\widetilde\omega_\epsilon)=
             \int_{\partial\mathfrak{B}_\rho^{+}(\xi_i)}t^{1-2s}\frac{\partial\widetilde\omega_\epsilon}{\partial \nu}\frac{\partial \widetilde{u}_\epsilon}{\partial x_l}-\int_{\mathfrak{B}_\rho^{+}(\xi_i)}t^{1-2s}\nabla\widetilde\omega_\epsilon\nabla\frac{\partial \widetilde{u}_\epsilon}{\partial x_l}\\
             =&\int_{\partial''\mathfrak{B}_\rho^{+}(\xi_i)}t^{1-2s}\frac{\partial\widetilde\omega_\epsilon}{\partial \nu}\frac{\partial \widetilde{u}_\epsilon}{\partial x_l}+\int_{{B}_\rho(\xi_i)}(-V(x)\omega_\epsilon+(p_s-\epsilon)u_\epsilon^{p_s-1-\epsilon}\omega_\epsilon)\frac{\partial u_\epsilon}{\partial x_l}-\int_{\mathfrak{B}_\rho^{+}(\xi_i)}t^{1-2s}\nabla\widetilde\omega_\epsilon\nabla\frac{\partial \widetilde{u}_\epsilon}{\partial x_l}.\\
         \end{aligned}
         \end{equation}
         Combining \eqref{eq4.5.1} and \eqref{eq4.5.2}, we can get the Pohozaev identity.
\end{proof}

\section{Technical  estimates}

In this appendix, we collect  various technical estimates and auxiliary results that are repeatedly used in the previous sections.
\begin{lemma}\textnormal{\cite[Lemma 2.3]{LLLY}}\label{lemLL2.3}
		Assume that  $0<\beta<N, \gamma>\beta$. Then
		\begin{equation}
			\int_{\mathbb{R}^N}\frac{1}{|x-y|^{N-\beta}}\frac{1}{(1+|y|)^\gamma}\,dy
			\leq C\begin{cases}
				\displaystyle\frac{1}{(1+|x|)^{\gamma-\beta}}, \quad&\text{if}\;\; \gamma<N,\vspace{2mm}\\
				\displaystyle\frac{1+|\log|x||}{(1+|x|)^{N-\beta}}, \quad &\text{if}\;\;\gamma=N,\vspace{2mm}\\
				\displaystyle\frac{1}{(1+|x|)^{N-\beta}}, \quad &\text{if}\;\;\gamma>N.
			\end{cases}
		\end{equation}
	\end{lemma}
    \begin{lemma}\textnormal{\cite[Lemma~A.3]{YGuo-2020}}\label{lemGA.3}
        Let $\mu>0$. For any constants $0<\beta<N$ there exists a constant $C>0$, independent of $\mu$, such that 
        \[
        \int_{\R^N\setminus B_\mu(y)}\frac{1}{|y-z|^{N+2s}}\frac{1}{(1+|z|)^\beta}\,dz\leq C\left(\frac{1}{(1+|y|)^{\beta+2s}}+\frac{1}{\mu^{2s}}\frac{1}{(1+|y|)^\beta}\right).
        \]
    \end{lemma}

\begin{lemma}\textnormal{\cite[Lemma~A.4]{YGuo-2020}}\label{lemGA.4}
    Let $\rho>0$ and suppose $(y-\xi)^2+t^2=\rho^2$ with $t>0$. Then for $\alpha>N$ and $0<\beta<N$, we have 
    $$
    \int_{\R^N}\frac{1}{(t+|z|)^\alpha}\frac{1}{|y-\xi-z|^\beta}\,dz\leq C\left(\frac{1}{(1+|y-\xi|)^\beta}\frac{1}{t^{\alpha-N}}+\frac{1}{(1+|y-\xi|)^{\alpha+\beta-N}}\right).
    $$
\end{lemma}

\begin{lemma}\textnormal{\cite[Lemma~A.5]{YGuo-2020}}\label{lemGA.5}
     Let $\rho>0$ and suppose $(x-\xi_i)^2+t^2=\rho^2$. Then there exists a constant $C>0$ such that 
     $$
     |\widetilde{W}_{\lambda_i,\xi_i}(x)|\leq C\frac{\lambda_i^{-\frac{N-2s}{2}}}{(1+|x-\xi_i|)^{N-2s}} ~~\text{and}~~ |\nabla\widetilde{W}_{\lambda_i,\xi_i}(x)|\leq C\frac{\lambda_i^{-\frac{N-2s}{2}}}{(1+|x-\xi_i|)^{N-2s+1}} .
     $$
\end{lemma}

\begin{lemma}\label{newlemGA.6}
Assume that $\xi^*_i$, $i=1,2,\cdots,k$ are the $k$ different   non-degenerate critical points  of $V(x)$ with $V(\xi^*_i)>0$ and $V(x)\in C^2(B_{5\delta}(\xi^*_i))$. Then 
    $$
    \int_{\partial''\mathfrak{B}_\rho^+(\xi_j)}t^{1-2s}|\nabla\widetilde{\phi}|^2\leq C\left(\lambda_j^{-2\sigma}\|\phi\|_*^2+\lambda_j^{-2\sigma}\|\mathcal{R}\|_{**}\|\phi\|_*\right).
    $$
    where  $\mathcal{R}$ is defined in \eqref{eq-2.24R}. 
\end{lemma}
\begin{proof}
The proof follows a similar approach to that of  Lemma A.6 in  \cite{YGuo-2020}. 
    For any $\delta>0$, we introduce the two sets,
    $D_1=\{X=(x,t):\delta<|X-(\xi_j,0)|<6\delta\}$ and $D_2=\{X=(x,t):2\delta<|X-(\xi_j,0)|<5\delta\}$. By Lemma \ref{lemGA.4}, for $(x,t)\in D_1$, we have 
    \begin{equation}\label{eqA.6}
    \begin{aligned}
        \widetilde{\phi}(x,t)=&\left|\int_{\R^N}\beta(N,s)\frac{t^{2s}}{(|x-z|^2+t^2)^{\frac{N+2s}{2}}}\phi(z)\right|\\
        \leq&C\|\phi\|_*\sum_{i=1}^k\int_{\R^N}\frac{t^{2s}}{(|x-z|^2+t^2)^{\frac{N+2s}{2}}}\frac{\lambda_i^{\frac{N-2s}{2}}}{(1+\lambda_i|z-\xi_i|)^{\frac{N-2s}{2}+\sigma}}\\
        \leq &C\|\phi\|_*t^{2s}\sum_{i=1}^k\lambda_i^{-\sigma}\int_{\R^N}\frac{1}{(|x-z|+t)^{N+2s}}\frac{1}{|z-\xi_i|^{\frac{N-2s}{2}+\sigma}}\\
        \leq &C\|\phi\|_*t^{2s}\sum_{i=1}^k\lambda_i^{-\sigma}\int_{\R^N}\frac{1}{(|z|+t)^{N+2s}}\frac{1}{|x-z-\xi_i|^{\frac{N-2s}{2}+\sigma}}\\
        \leq &C\|\phi\|_*t^{2s}\sum_{i=1}^k\lambda_i^{-\sigma}\left(\frac{1}{(1+|x-\xi_i|)^{\frac{N-2s}{2}+\sigma}}\frac{1}{t^{2s}}+\frac{1}{(1+|x-\xi_i|)^{\frac{N+2s}{2}+\sigma}}\right)\\
        \leq &C\|\phi\|_*\epsilon^{\frac{\sigma}{2s}}\sum_{i=1}^k\frac{1}{(1+|x-\xi_i|)^{\frac{N-2s}{2}+\sigma}}.\\
    \end{aligned}
    \end{equation}
    Let $\varphi\in C^\infty_0(\R^{N+1})$ be  a function with $\varphi(x,t)=1$ in $D_2$, $\varphi(x,t)=0$ in $\R^{N+1}\setminus D_1$ and $|\nabla \varphi|\leq C$. Note that $\widetilde{\phi}$ satisfies 
   $$
        \begin{cases}
            \text{div}(t^{1-2s}\nabla\widetilde{\phi})=0 \quad &\text{in} \quad\R^{N+1}_+,\\
            -\lim\limits_{t\to 0^+}t^{1-2s}\partial_t\widetilde{\phi}=-V(x)\phi+p_s{\bm W}_{\bm \lambda, \bm \xi}^{p_s-1}\phi+\mathcal{N}(\phi)+\mathcal{R}+\sum\limits_{i=1}^k\sum\limits_{l=0}^N c_i^l W_{\lambda_i, \xi_i}^{p_s-1}Z_{i,l}\quad &\text{on}  \quad\R^{N}.
        \end{cases}
       $$
       Multiplying both sides of the equation by $\varphi^2\widetilde{\phi}$ and integrating by parts over $D_1\cap\{t>0\}$, we have 
       \begin{equation}\label{eqA.61}
       \begin{aligned}
       0=&\int_{\partial (D_1\cap\{t>0\})}\varphi^2\widetilde{\phi}t^{1-2s}(\nabla\widetilde{\phi}\cdot\nu)-\int_{D_1\cap\{t>0\}}t^{1-2s}\nabla\widetilde{\phi}\nabla(\varphi^2\widetilde{\phi})\\
       =&\int_{B_{6\delta}(\xi_j)\setminus B_\delta(\xi_j)}\varphi^2(x,0)\phi\left(-V(x)\phi+p_s{\bm W}_{\bm \lambda, \bm \xi}^{p_s-1}\phi+\mathcal{N}(\phi)+\mathcal{R}+\sum\limits_{i=1}^k\sum\limits_{l=0}^N c_i^l W_{\lambda_i, \xi_i}^{p_s-1}Z_{i,l}\right)\\
       &-\int_{D_1\cap\{t>0\}}t^{1-2s}\nabla\widetilde{\phi}(\varphi^2\nabla\widetilde{\phi}+2\varphi\widetilde{\phi}\nabla\varphi).
        \end{aligned}
       \end{equation}
       By Proposition \ref{prop2.3}, Lemma \ref{lem2.5},  Lemma \ref{lem2.6} and direct computation, we obtain 
       \begin{equation}\label{eqB-0614-1}
           \begin{aligned}
              &\ \ \ \ \left| \int_{B_{6\delta}(\xi_j)\setminus B_\delta(\xi_j)}\varphi(x,0)^2\phi\left(-V(x)\phi+p_s{\bm W}_{\bm \lambda, \bm \xi}^{p_s-1}\phi+\mathcal{N}(\phi)+\mathcal{R}+\sum\limits_{i=1}^k\sum\limits_{l=0}^N c_i^l W_{\lambda_i, \xi_i}^{p_s-1}Z_{i,l}\right)\right|\\
              & \leq C\int_{B_{6\delta}(\xi_j)\setminus B_\delta(\xi_j)}\phi^2+{\bm W}_{\bm \lambda, \bm \xi}^{p_s-1}\phi^2+|\mathcal{N}(\phi)||\phi|+|\mathcal{R}||\phi|+\sum\limits_{l=0}^N |c_j^l| W_{\lambda_j, \xi_j}^{p_s-1}|Z_{j,l}||\phi|\\
              & \leq C\left(\lambda_j^{-2\sigma}\|\phi\|_*^2+\lambda_j^{-2s-2\sigma}\|\phi\|_*^2+\lambda_j^{-2\sigma}\|\mathcal{N}(\phi)\|_{**}\|\phi\|_*+\lambda_j^{-2\sigma}\|\mathcal{R}\|_{**}\|\phi\|_*\right)\\
              &\ \ \ +C\left(|c_j^0|\lambda_j^{-\frac{N+2s}{2}-1-\sigma}\|\phi\|_*+\sum_{l=1}^N|c_j^l|\lambda_j^{-\frac{N+2s}{2}+1-\sigma}\|\phi\|_*\right)\\
               & \leq C\left(\lambda_j^{-2\sigma}\|\phi\|_*^2+\lambda_j^{-2\sigma}\|\mathcal{R}\|_{**}\|\phi\|_*\right).
           \end{aligned}
       \end{equation}
       Using Young's inequality, we have 
       \begin{equation}\label{eqA.62}
       \left|\int_{D_1\cap\{t>0\}}t^{1-2s}\nabla\widetilde{\phi}\varphi\widetilde{\phi}\nabla\varphi\right|\leq \frac{1}{4}\int_{D_1\cap\{t>0\}}t^{1-2s}|\nabla\widetilde{\phi}|^2\varphi^2+C\int_{D_1\cap\{t>0\}}t^{1-2s}\widetilde{\phi}^2|\nabla\varphi|^2.
       \end{equation}
       Combining \eqref{eqA.61}-\eqref{eqA.62}, we deduce that
       $$
       \int_{D_1\cap\{t>0\}}t^{1-2s}|\nabla\widetilde{\phi}|^2\varphi^2\leq C\int_{D_1\cap\{t>0\}}t^{1-2s}\widetilde{\phi}^2|\nabla\varphi|^2+C\left(\lambda_j^{-2\sigma}\|\phi\|_*^2+\lambda_j^{-2\sigma}\|\mathcal{R}\|_{**}\|\phi\|_*\right).
       $$
       By means of \eqref{eqA.6}, we further obtain
       $$
       \begin{aligned}
           &\int_{D_2\cap\{t>0\}}t^{1-2s}|\nabla\widetilde{\phi}|^2\\
           \leq &C\int_{D_1\cap\{t>0\}}t^{1-2s}\widetilde{\phi}^2|\nabla\varphi|^2+C\left(\lambda_j^{-2\sigma}\|\phi\|_*^2+\lambda_j^{-2\sigma}\|\mathcal{R}\|_{**}\|\phi\|_*\right)\\
           \leq & C\|\phi\|_*^2\epsilon^{\frac{\sigma}{s}}\int_{D_1\cap\{t>0\}}t^{1-2s}\sum_{i=1}^k\frac{1}{(1+|x-\xi_i|)^{N-2s+2\sigma}}+C\left(\lambda_j^{-2\sigma}\|\phi\|_*^2+\lambda_j^{-2\sigma}\|\mathcal{R}\|_{**}\|\phi\|_*\right)\\
           \leq&C\left(\lambda_j^{-2\sigma}\|\phi\|_*^2+\lambda_j^{-2\sigma}\|\mathcal{R}\|_{**}\|\phi\|_*\right).
       \end{aligned}
       $$
       The mean value theorem for integrals implies the existence of $\rho=\rho(\delta)\in (2\delta,\,5\delta)$ such that 
       $$
    \int_{\partial''\mathfrak{B}_\rho^+(\xi_j)}t^{1-2s}|\nabla\widetilde{\phi}|^2\leq C\left(\lambda_j^{-2\sigma}\|\phi\|_*^2+\lambda_j^{-2\sigma}\|\mathcal{R}\|_{**}\|\phi\|_*\right).
    $$      
\end{proof}
Next we establish estimates for the projected remainder term $\psi_\epsilon^*$, 
which serve as a crucial ingredient in the proof of local uniqueness of multi-peak solutions given in Section 6.
\begin{lemma}\label{F6.9}
    Let $\rho>0$ and suppose $\big|x-\xi_i^{(1)}\big|^2+t^2=\rho^2$ with $t>0$. Then
    \begin{equation*}
      |\widetilde{\psi}^*_\epsilon|=o(\epsilon^{\frac{1}{2s}+\frac{\sigma}{2s}})\sum_{i=1}^k\frac{1}{(1+|x-\xi_i^{(1)}|)^{\frac{N-2s}{2}+1+\sigma}}.
    \end{equation*}
\end{lemma}
\begin{proof}
    By Lemma \ref{lem-psi} and Lemma \ref{lemGA.4}, we have 
    \begin{equation*}
        \begin{aligned}
        |\widetilde{\psi}^*_\epsilon|=&\left|\int_{\R^N}d(N,s)\frac{t^{2s}}{(|x-z|^2+t^2)^{\frac{N+2s}{2}}}\psi^*_\epsilon(z)\right|\\
        =&o(1)\sum_{i=1}^k\int_{\R^N}\frac{t^{2s}}{(|x-z|^2+t^2)^{\frac{N+2s}{2}}}\frac{\big(\lambda_{i}^{(1)}\big)^{\frac{N-2s}{2}}}{(1+\lambda_{i}^{(1)}\big|z-\xi_{i}^{(1)}\big |)^{\frac{N-2s}{2}+1+\sigma}}\\
        =&o(1)t^{2s}\sum_{i=1}^k(\lambda_i^{(1)})^{-1-\sigma}\int_{\R^N}\frac{1}{(|x-z|+t)^{N+2s}}\frac{1}{|z-\xi_i^{(1)}|^{\frac{N-2s}{2}+1+\sigma}}\\
        = &o(\epsilon^{\frac{1}{2s}+\frac{\sigma}{2s}})t^{2s}\sum_{i=1}^k\int_{\R^N}\frac{1}{(|z|+t)^{N+2s}}\frac{1}{|x-z-\xi_i^{(1)}|^{\frac{N-2s}{2}+1+\sigma}}\\
        = &o(\epsilon^{\frac{1}{2s}+\frac{\sigma}{2s}})t^{2s}\sum_{i=1}^k\left(\frac{1}{(1+|x-\xi_i^{(1)}|)^{\frac{N-2s}{2}+1+\sigma}}\frac{1}{t^{2s}}+\frac{1}{(1+|x-\xi_i^{(1)}|)^{\frac{N+2s}{2}+1+\sigma}}\right)\\
        = &o(\epsilon^{\frac{1}{2s}+\frac{\sigma}{2s}})\sum_{i=1}^k\frac{1}{(1+|x-\xi_i^{(1)}|)^{\frac{N-2s}{2}+1+\sigma}}.\\
        \end{aligned}
    \end{equation*}
\end{proof}
\begin{lemma}\label{F0514lem1}
     For any $\delta>0$, there exists $\rho=\rho(\delta)\in (2\delta,\;5\delta)$ such that for $N>4s$,  
    $$
    \int_{\partial''\mathfrak{B}_\rho^+(\xi_j^{(1)})}t^{1-2s}|\nabla\widetilde{\psi}^*_\epsilon|^2=o(\epsilon^{\frac{1}{s}+\frac{\sigma}{s}}).
    $$
\end{lemma}
\begin{proof}
    The proof follows the same lines as Lemma \ref{newlemGA.6}, we therefore only detail the final estimation. Using  Lemmas \ref{newlemGA.6}, \ref{F6.9}, \ref{lem-psi},  and \eqref{eq_LP}, we obtain
     $$
       \begin{aligned}
           &\int_{D_2}t^{1-2s}|\nabla\widetilde{\psi}^*_\epsilon|^2\\
           \leq& C\int_{D_1\cap\{t>0\}}t^{1-2s}|\widetilde{\psi}^*_\epsilon|^2|\nabla\varphi|^2 +C\left| \int_{B_{6\delta}(\xi_j)\setminus B_\delta(\xi_j)}\varphi(x,0)^2\psi^*_\epsilon\left(-V(x)\psi^*_\epsilon+(p_s-\epsilon)c_\epsilon(x)\psi^*_\epsilon+R^*\right)\right|
         \\
           = &o(\epsilon^{\frac{1}{s}+\frac{\sigma}{s}}),
       \end{aligned}
       $$
       where $R^*=\sum_{i=1}^5L_i+\sum_{i=1}^5P_i+\sum_{i=1}^4Q_i$, $D_1=\{X=(x,t):\delta<|X-(\xi_j^{(1)},0)|<6\delta\}$ and $D_2=\{X=(x,t):2\delta<|X-(\xi_j^{(1)},0)|<5\delta\}$.
       By the mean value theorem for integrals, there exists $\rho=\rho(\delta)\in (2\delta,\,5\delta)$ such that 
       $$
    \int_{\partial''\mathfrak{B}_\rho^+(\xi_j^{(1)})}t^{1-2s}|\nabla\widetilde{\psi}^*_\epsilon|^2=o(\epsilon^{\frac{1}{s}+\frac{\sigma}{s}}).
    $$
\end{proof}

\noindent{\bf Conflict of interest}
									
The authors have no conflict of interest to disclose.

\vskip 0.1cm
\noindent{\bf Data Availibility}
									
\vskip 0.1cm
									
Data sharing is not applicable to this article as no new data was created or analyzed in this study.
									
\vskip 0.1cm
									
\noindent{\bf Acknowledgements}:\, Z. Liu  was supported by the  Natural Science Foundation of Henan (Grant No. 252300421051) and
the National Natural Science Foundation of China (Grant Nos. 12371111, 12471104). 


\end{document}